\documentclass[oneside]{amsart}
\usepackage{amssymb}
\usepackage{hyperref}
\usepackage[a4paper]{geometry}
\usepackage{dsfont}

\theoremstyle{plain}
\newtheorem{thm}{Theorem}[section] 
\newtheorem{prop}[thm]{Proposition} 
\newtheorem{lem}[thm]{Lemma} 
\newtheorem{cor}[thm]{Corollary}
\theoremstyle{remark} 
\newtheorem*{rem}{Remark}
\theoremstyle{definition}
\newtheorem*{defin}{Definition}
\numberwithin{equation}{section}

\renewcommand*{\div}{\operatorname{div}}
\newcommand*{\curl}{\operatorname{curl}}

\newcommand*{\supp}{\operatorname{supp}}

\usepackage[textsize=tiny,textwidth=27mm]{todonotes}

\newcommand*{\bydef}{\overset{\rm def}{=}}
\newcommand*{\norm}[1]{\left\Vert #1\right\Vert}

\begin{document}

\title[Damped Strichartz estimates and Euler--Maxwell]{Damped Strichartz estimates and the incompressible Euler--Maxwell system}

\author{Diogo Ars\'enio}
\author{Haroune Houamed}
\address{New York University Abu Dhabi \\
Abu Dhabi \\
United Arab Emirates} 
\email{\href{mailto:diogo.arsenio@nyu.edu}{diogo.arsenio@nyu.edu}, \href{mailto:haroune.houamed@nyu.edu}{haroune.houamed@nyu.edu}}

\thanks{\emph{Acknowledgements.} The second author would like to thank Taoufik Hmidi for the fruitful discussions on many related subjects.}

\date{\today}

\begin{abstract}
Euler--Maxwell systems describe the dynamics of inviscid plasmas. In this work, we consider an incompressible two-dimensional version of such systems and prove the existence and uniqueness of global weak solutions, uniformly with respect to the speed of light $c\in (c_0,\infty)$, for some threshold value $c_0>0$ depending only on the initial data. In particular, the condition $c>c_0$ ensures that the velocity of the plasma nowhere exceeds the speed of light and allows us to analyze the singular regime $c\to\infty$.

The functional setting for the fluid velocity lies in the framework of Yudovich's solutions of the two-dimensional Euler equations, whereas the analysis of the electromagnetic field hinges upon the refined interactions between the damping and dispersive phenomena in Maxwell's equations in the whole space. This analysis is enabled by the new development of a robust abstract method allowing us to incorporate the damping effect into a variety of existing estimates. The use of this method is illustrated by the derivation of damped Strichartz estimates (including endpoint cases) for several dispersive systems (including the wave and Schr\"odinger equations), as well as damped maximal regularity estimates for the heat equation. The ensuing damped Strichartz estimates supersede previously existing results on the same systems.
\end{abstract}

\maketitle

\tableofcontents


\section{Introduction}

We are concerned with the existence and uniqueness of solutions to the incompressible Euler--Maxwell system
\begin{equation}\label{EM}
	\begin{cases}
		\begin{aligned}
			\text{\tiny(Euler's equation)}&&&\partial_t u +u \cdot\nabla u = - \nabla p + j \times B, &\div u =0,&
			\\
			\text{\tiny(Amp\`ere's equation)}&&&\frac{1}{c} \partial_t E - \nabla \times B =- j , &\div E = 0,&
			\\
			\text{\tiny(Faraday's equation)}&&&\frac{1}{c} \partial_t B + \nabla \times E  = 0 , &\div B = 0,&
			\\
			\text{\tiny(Ohm's law)}&&&j= \sigma \big( cE + P(u \times B)\big), &\div j = 0,&
		\end{aligned}
	\end{cases}
\end{equation}
for some initial data $(u,E,B)_{|t=0}=(u_0,E_0,B_0)$,
with the two-dimensional normal structure on the vector fields
\begin{equation}\label{structure:2dim}
	u(t,x)=
	\begin{pmatrix}
		u_1(t,x)\\u_2(t,x)\\0
	\end{pmatrix},
	\qquad
	E(t,x)=
	\begin{pmatrix}
		E_1(t,x)\\E_2(t,x)\\0
	\end{pmatrix}
	\qquad\text{and}\qquad
	B(t,x)=
	\begin{pmatrix}
		0\\0\\b(t,x)
	\end{pmatrix},
\end{equation}
where $(t,x)\in [0,\infty)\times \mathbb{R}^2$ and $P= \mathrm{Id}-\Delta^{-1}\nabla\mathrm{div}$ denotes Leray's projector onto divergence-free vector fields. We will later see that the normal structure \eqref{structure:2dim} is propagated by the flow and therefore persistent.

Taking the divergence of Maxwell's system, which is made up of Amp\`ere and Faraday's equations, notice that the divergence-free conditions $\div E=0$ and $\div B=0$ are also propagated by the evolution of the system, provided they hold initially. (In fact, notice that the condition $\div B=0$ is a trivial consequence of the normal structure \eqref{structure:2dim}. Nevertheless, it is physically relevant, since magnetic fields are solenoidal.)

This model describes the evolution of a plasma, i.e., a charged gas or an electrically conducting fluid, subject to the self-induced electromagnetic Lorentz force $j\times B$. Here, the field $u$ denotes the velocity of the fluid, $E$ and $B$ are the electric and magnetic fields, respectively, whereas $j$ denotes the electric current. Moreover, the positive constants $c$ and $\sigma$ represent the speed of light and the electrical conductivity, respectively. We refer to \cite{bis-book, D-book} for details about the physical principles behind the modeling of plasmas.

It is readily seen that any smooth solution $(u,E,B)\in C_c^1\left([0,\infty)\times \mathbb{R}^2\right)$ of \eqref{EM} satisfies the energy inequality
\begin{equation}\label{energy-inequa}
	\norm {u(t)}_{L^2}^2 + \norm {E(t)}_{L^2}^2 + \norm {B(t)}_{L^2}^2
	+\frac{2}{\sigma}\int_0^t \norm {j(\tau)}_{L^2}^2 d\tau \leq \mathcal{E}_0^2,
\end{equation}
for all $t\geq 0$, where we denote
\begin{equation*}
	\mathcal{E}_0 \bydef \norm {(u_0,E_0,B_0)}_{L^2}.
\end{equation*}
(Observe that \eqref{energy-inequa} actually holds with an equality sign for smooth functions, but this will not be used.) This is the only known global a priori estimate for solutions of \eqref{EM} and the ensuing natural bound
\begin{equation*}
	(u,E,B)\in L^\infty\left([0,\infty);L^2(\mathbb{R}^2)\right)
\end{equation*}
is insufficient to guarantee the existence of global weak solutions to \eqref{EM}. At least, no known method has so far been able to build such solutions and the same holds true for the classical two-dimensional incompressible Euler system
\begin{equation}\label{Euler:1}
	\partial_t u +u \cdot\nabla u = - \nabla p, \qquad \div u =0,
\end{equation}
which corresponds to the case $(E,B)=0$. This is due to the fact that the nonlinear terms in \eqref{EM} and \eqref{Euler:1} are, in general, not stable under weak convergence of solutions.

\subsection{Main results}

Our main result on the Euler--Maxwell system \eqref{EM} establishes the global existence and uniqueness of weak solutions, for any initial data in suitable spaces, provided the speed of light $c$ is sufficiently large. Note that this is seemingly the only known global existence result for incompressible Euler--Maxwell systems. It reads as follows.

\begin{thm}\label{thm-EM}
	Let $p$ and $s$ be any real numbers in $(2,\infty)$ and $(\frac 74,2)$, respectively. For any initial data
	\begin{equation*}
		(u_0,E_0,B_0)\in
		\left((H^1\cap\dot W^{1,p})\times H^s \times H^s \right)(\mathbb{R}^2),
	\end{equation*}
	with $\div u_0=\div E_0=\div B_0$ and the two-dimensional normal structure \eqref{structure:2dim}, there is a constant $c_0>0$, such that, for any speed of light $c\in(c_0,\infty)$, there is a global weak solution $(u,E,B)$ to the two-dimensional Euler--Maxwell system \eqref{EM}, with the normal structure \eqref{structure:2dim}, satisfying the energy inequality \eqref{energy-inequa} and enjoying the additional regularity
	\begin{equation}\label{propagation:damping:2}
		\begin{gathered}
			u\in L^\infty(\mathbb{R}^+; H^1\cap \dot W^{1,p}),
			\quad
			(E,B)\in L^\infty(\mathbb{R}^+; H^s),
			\\
			(cE,B)\in L^2(\mathbb{R}^+;\dot H^1),
			\quad
			cE\in L^2(\mathbb{R}^+;\dot H^s),
			\quad
			(E,B)\in L^2(\mathbb{R}^+;\dot W^{1,\infty}).
		\end{gathered}
	\end{equation}
	It is to be emphasized that the bounds in \eqref{propagation:damping:2} are uniform in $c\in(c_0,\infty)$, for any given initial data.
	
	If, furthermore, the initial vorticity $\omega_0\bydef \nabla\times u_0$ belongs to $L^\infty(\mathbb{R}^2)$, then the solution enjoys the global bound
	\begin{equation*}
		\omega\bydef \nabla\times u \in L^\infty(\mathbb{R}^+;L^\infty),
	\end{equation*}
	and it is unique in the space of all solutions $(\bar u, \bar E,\bar B)$ to the Euler--Maxwell system \eqref{EM} satisfying the bounds, locally in time,
	\begin{equation*}
		(\bar u, \bar E,\bar B)\in L^\infty_tL^2_x,
		\qquad \bar u \in L^{2}_tL^\infty_x,
		\qquad \bar j \in L^{2}_{t,x},
	\end{equation*}
	and having the same initial data.
\end{thm}

Theorem \ref{thm-EM} is a simple and more accessible reformulation of the results from Section \ref{section:perfect fluid}, which are stated therein in full detail in the setting of Besov and Chemin--Lerner spaces (see Appendix \ref{besov:1} for a precise definition of these spaces). Indeed, it is readily seen that Theorem \ref{thm-EM} follows directly from the combination of Theorems \ref{main:1}, \ref{main:2}, \ref{main:3} and Corollary \ref{persistence:s-reg} with straightforward embeddings of functional spaces. The respective proofs of these results are also provided in complete detail in Section \ref{section:perfect fluid}.

One should note that the constant $c_0$ in the above statement depends on norms of the initial data. Thus, for any given $c>0$, the condition $c_0<c$ can be interpreted, in a fully equivalent way, as a smallness condition on the initial data. In fact, a careful inspection of \eqref{initial:3} in the statement of Theorem \ref{main:3} readily provides an explicit expression for $c_0$ in terms of the norms of $(u_0,E_0,B_0)$ in $(H^1\cap\dot W^{1,p})\times H^s \times H^s$. More specifically, for any given initial data, one could set, for example, that
\begin{equation*}
	c_0=\max\left\{1,\left(\norm{u_0}_{H^1\cap\dot W^{1,p}}
	+\norm{(E_0,B_0)}_{H^s}
	\right)
	Ce^{C\mathcal{E}_0^5}\right\},
\end{equation*}
for some suitable large constant $C>0$ which only depends on $p$ and $s$, and is independent of the initial data. Then, with this definition of $c_0$, it is straightforward to show that the condition $c>c_0$ implies the validity of \eqref{initial:3}. In particular, for a given speed of light $c$, we observe that the existence of solutions is a consequence of the smallness of the initial data. Finally, we also note that it is not difficult to provide sharper formulas for $c_0$, with increasing complexity.

A detailed scaling analysis of solutions to the Euler--Maxwell system \eqref{EM} is conducted in Section \ref{section:dimensional}, which further clarifies the significance of the initial conditions of our main results and their dependence on the physical constants $c$ and $\sigma$.

\medskip

We have already emphasized that the bounds \eqref{propagation:damping:2} on the solutions of the incompressible Euler--Maxwell system \eqref{EM} are uniform with respect to the speed of light $c>c_0$. This crucial feature allows us to deduce a simple but powerful convergence result in the asymptotic regime $c\to\infty$, which is of particular interest. We refer to \cite{aim15} for a thorough discussion of this regime in the context of incompressible Navier--Stokes--Maxwell systems.

Generally speaking, the physical relevance of the regime $c\to\infty$ in Euler--Maxwell systems stems from the fact that the limiting magnetohydrodynamic systems are suitable to describe the behavior of flows which are influenced by self-induced magnetic fields. This is the case, for instance, of the terrestrial magnetic field, which is sustained by the earth's core through the dynamo effect, or the solar magnetic field, which is responsible for sunspots, or the galactic magnetic field, which plays a role in the formation of stars. We refer to \cite{D-book} for more details on the physical background of magnetohydrodynamic systems.

The next result follows directly from Theorem \ref{thm-EM} and establishes a magnetohydrodynamic system by taking the limit of the Euler--Maxwell system \eqref{EM} in the singular regime $c\to\infty$. Observe that it recovers the classical Yudovich Theorem for the incompressible Euler system \eqref{Euler:1} by setting $B\equiv 0$.

\begin{cor}\label{coro:singular limit}
	For any given initial data $(u_0,E_0,B_0)$ as in Theorem \ref{thm-EM} (for some $p\in (2,\infty)$ and $s\in (\frac 74, 2)$), consider the global solution $(u^c,E^c,B^c)$ constructed therein, for each $c\in (c_0,\infty)$. Then, the set of solutions $\{(u^c,E^c,B^c)\}_{c>c_0}$ is relatively compact in $L^2_{t,x,\mathrm{loc}}$. In particular, for any sequence $\{(u^{c_n},E^{c_n},B^{c_n})\}_{n\in\mathbb{N}}$, with $c_n\to\infty$, there is a convergent subsequence (which we do not distinguish, for simplicity)
	\begin{equation}\label{convergence:1}
		(u^{c_n},E^{c_n},B^{c_n})\stackrel{n\to\infty}{\longrightarrow} (u,0,B),\quad\text{in }L^2_{t,x,\mathrm{loc}},
	\end{equation}
	where $(u,B)=\big((u_1,u_2,0),(0,0,b)\big)$ has the normal structure \eqref{structure:2dim} and is a global weak solution of the system
	\begin{equation}\label{mhd}
		\left\{
		\begin{aligned}
			&\partial_t u +u \cdot\nabla u = -\nabla p , &\div u =0,&
			\\ 
			&\partial_t b - \frac{1}{\sigma}\Delta b + u \cdot\nabla b = 0 , &&
		\end{aligned}
		\right.
	\end{equation}
	with the bounds
	\begin{equation*}
		u\in L^\infty(\mathbb{R}^+;H^1\cap \dot W^{1,p}),
		\quad
		b\in L^\infty(\mathbb{R}^+;H^1)\cap L^2(\mathbb{R}^+;\dot H^1\cap \dot H^2\cap \dot W^{1,\infty}).
	\end{equation*}
	
	If, furthermore, the initial vorticity $\omega_0$ belongs to $L^\infty(\mathbb{R}^2)$, then the solution $(u,b)$ to \eqref{mhd} satisfies the additional bound $\omega\in L^\infty(\mathbb{R}^+;L^\infty)$ and is unique in the space of all solutions $(\bar u,\bar b)$ satisfying the bounds, locally in time,
	\begin{equation*}
		\bar u\in L^\infty_tL^2_x\cap L^{2}_tL^\infty_x,
		\qquad \bar b \in L^\infty_tL^2_x\cap L^2_t\dot H^1_x,
	\end{equation*}
	and having the same initial data. Moreover, one has the convergence
	\begin{equation}\label{convergence:2}
		(u^{c},E^{c},B^{c})\stackrel{c\to\infty}{\longrightarrow} (u,0,B),\quad\text{in }L^2_{t,x,\mathrm{loc}},
	\end{equation}
	without extraction of subsequences.
\end{cor}

\begin{rem}
	Note that \eqref{mhd} is a simple form of magnetohydrodynamic system. Indeed, the equations for $u$ and $b$ are not genuinely coupled, for the incompressible Euler equation does not contain an external magnetic force. This can be interpreted as a consequence of the two-dimensional normal structure \eqref{structure:2dim}. More specifically, whenever the electric current is given by $j=\nabla\times B$, a straightforward calculation exploiting \eqref{structure:2dim} shows that the Lorentz force satisfies that
	\begin{equation*}
		j\times B=(\nabla\times B)\times B=-\frac 12\nabla(b^2),
	\end{equation*}
	which can be absorbed in the pressure gradient. In particular, since $u$ is independent of $b$ in this regime, there can be no Alfv\'en waves (see \cite{D-book} for an introduction to Alfv\'en waves). Therefore, in this case, the limiting magnetohydrodynamic system loses the feature of some important physical effects (such as Alfv\'en waves). This suggests that extending the results of the present article beyond the two-dimensional normal structure \eqref{structure:2dim} is of particular interest and significance.
\end{rem}

\begin{proof}
	We begin by showing the relative compactness of the set of solutions $\{(u^c,E^c,B^c)\}_{c>c_0}$ in $L^2_{t,x}(K)$, for any compact set $K\subset \mathbb{R}^+\times\mathbb{R}^2$. To that end, note that the energy inequality \eqref{energy-inequa} and the global bounds \eqref{propagation:damping:2} on the solutions hold uniformly in $c$. In particular, it is readily seen that $E^c\to 0$ in $L^2_{t,x,\mathrm{loc}}$, as $c\to\infty$. Therefore, we only need to focus on $\{(u^c,B^c)\}_{c>c_0}$.
	
	Now, one can show from \eqref{propagation:damping:2} that $u^c\in L^\infty_{t,x}$ and $B^c\in L^2_tL^\infty_x$. (For instance, using the Gagliardo--Nirenberg convexity inequality \eqref{convexity:1}, which is recalled later on.) It therefore follows directly from \eqref{EMc} that $\partial_t u^c  = P(j^c \times B^c)-P(u^c \cdot\nabla u^c)$ is uniformly bounded in $L^1_{t,\mathrm{loc}}L^2_x$. Similarly, it is readily seen from Faraday's equation $\partial_t B^c=-c\nabla\times E^c$ that $\partial_t B^c$ is uniformly bounded in $L^2_{t,x}$. Then, further combining these controls of $\partial_t u^c$ and $\partial_t B^c$ with the uniform bound $(u^{c_n},B^{c_n})\in L^\infty_tH^1_x$ and the compactness of the embedding $H^1_{\rm loc}\subset L^2_{\rm loc}$, we deduce that $\left\{(u^c,B^c)\right\}_{c>c_0}$ is relatively compact in the topology of $L^2_{t,x,{\rm loc}}$, by a classical compactness result by Aubin and Lions (see \cite{s87} for a thorough discussion of such compactness results and, in particular, Section 9 therein, for convenient results which are easily applicable to our setting).
	
	Next, for any convergent subsequence \eqref{convergence:1}, employing Ohm's law to substitute $c_nE^{c_n}$ into Faraday's equation in \eqref{EM}, observe that we only have to pass to the limit in the system
	\begin{equation}\label{EMc}
		\begin{cases}
			\begin{aligned}
				&\partial_t u^{c_n} +u^{c_n} \cdot\nabla u^{c_n} = - \nabla p^{c_n} + j^{c_n} \times B^{c_n}, &\div u^{c_n} =0,&
				\\
				&\frac{1}{{c_n}} \partial_t E^{c_n} - \nabla \times B^{c_n} =- j^{c_n} , &\div B^{c_n} = 0,&
				\\
				&\partial_t B^{c_n} + \frac 1\sigma \nabla \times j^{c_n} + u^{c_n}\cdot\nabla B^{c_n}= 0. &&
			\end{aligned}
		\end{cases}
	\end{equation}
	Moreover, up to further extraction of subsequences, it is also possible to assume that one has the weak convergence
	\begin{equation*}
		j^{c_n}\rightharpoonup j\quad\text{in }L^2_{t,x}.
	\end{equation*}
	All in all, passing to the limit $n\to\infty$ in \eqref{EMc} in the sense of distributions and exploiting the strong convergence \eqref{convergence:1}, we find that
	\begin{equation*}
		\begin{cases}
			\begin{aligned}
				&\partial_t u +u \cdot\nabla u = - \nabla p + j\times B, &\div u =0,&
				\\
				& \nabla \times B=j , &\div B = 0,&
				\\
				&\partial_t B + \frac 1\sigma \nabla \times j + u\cdot\nabla B= 0. &&
			\end{aligned}
		\end{cases}
	\end{equation*}
	Then, recalling the vector identity $\nabla\times(\nabla\times B)=\nabla(\div B)-\Delta B$ and noticing that $(\nabla\times B)\times B=-\frac 12\nabla(b^2)$, we conclude that $(u,b)$ is a solution of \eqref{mhd}.
	
	Finally, if we further assume the pointwise boundedness of the initial vorticity $\omega_0$, then $\omega^c=\nabla\times u^c$ remains uniformly bounded in $L^\infty_{t,x}$, thereby yielding a similar bound for the limiting system \eqref{mhd}. These bounds then fall in the framework of Yudovich's uniqueness theorem (see \cite[Section 8.2.4]{mb02}, for instance), which guarantees the uniqueness of the solution $u$ to the incompressible two-dimensional Euler system. Alternatively, one can also deduce the uniqueness of $u$ by reproducing the arguments from Section \ref{section:main:2}, below, by setting $(E,B)=0$. As for the uniqueness of $b$, it easily follows from classical energy estimates on the heat equation.
	
	At last, the uniqueness of the limit point $(u,0,B)$ allows us to deduce the validity of \eqref{convergence:2}, which completes the proof of the corollary.
\end{proof}

\subsection{Other models of incompressible plasmas}

The Euler--Maxwell system \eqref{EM} can be seen as the inviscid version of the Navier--Stokes--Maxwell system given by
\begin{equation}\label{NSM0}
	\begin{cases}
		\begin{aligned}
			&\partial_t u +u \cdot\nabla u - \nu \Delta u = - \nabla p + j \times B, &\div u =0,&
			\\
			&\frac{1}{c} \partial_t E - \nabla \times B =- j , &\div E = 0,&
			\\
			&\frac{1}{c} \partial_t B + \nabla \times E  = 0 , &\div B = 0,&
			\\
			&j= \sigma \big( cE + P(u \times B)\big), &\div j = 0,&
		\end{aligned}
	\end{cases}
\end{equation}
where $\nu>0$ denotes the viscosity of the fluid.

The derivation of \eqref{NSM0} has been established rigorously in \cite{as} through the analysis of the viscous incompressible hydrodynamic regimes of Vlasov--Maxwell--Boltzmann systems. In particular, it follows from the results therein that \eqref{NSM0} can be obtained by letting $\delta\to 0$, with $\delta>0$, in the more complete system
\begin{equation}\label{NSM1}
	\begin{cases}
		\begin{aligned}
			&\partial_t u +u \cdot\nabla u - \nu \Delta u = - \nabla p + \delta cnE + j \times B, &\div u =0,&
			\\
			&\frac{1}{c} \partial_t E - \nabla \times B =- j , &\div E = \delta n,&
			\\
			&\frac{1}{c} \partial_t B + \nabla \times E  = 0 , &\div B = 0,&
			\\
			&j-\delta nu= \sigma \big(-\frac c\delta \nabla n + cE + u \times B\big), &&
		\end{aligned}
	\end{cases}
\end{equation}
which takes the Coulomb force $nE$ into account, where $n$ is the electric charge density.

The work performed in \cite{as} addresses the viscous incompressible regimes of Vlasov--Maxwell--Boltzmann systems, only. However, inviscid incompressible regimes can also be achieved as an asymptotic limit of collisional kinetic equations. For instance, the incompressible Euler limit of the Boltzmann equation has been established in \cite{sr03,sr09}. (A general discussion of hydrodynamic regimes of the Boltzmann equation can also be found in \cite{srbook09}.) Similarly, in the vein of the results from \cite{as}, it is possible to derive \eqref{EM} by considering the incompressible Euler regime of Vlasov--Maxwell--Boltzmann systems, at least formally. However, this remains to be done rigorously.

The well-posedness theory established in this article only concerns \eqref{EM} and does not encompass the inviscid version of \eqref{NSM1} (i.e., the corresponding Euler--Maxwell system obtained by setting $\nu=0$ in \eqref{NSM1}). However, we are hopeful that some adaptation of our results can be implemented to show the existence and uniqueness of solutions to \eqref{NSM1}, with $\nu=0$. Nevertheless, for the sake of simplicity, we are going to stick to \eqref{EM}.

It turns out that there is yet another version of incompressible Navier--Stokes--Maxwell systems which is commonly found in the literature. It reads
\begin{equation}\label{NSM2}
	\begin{cases}
		\begin{aligned}
			&\partial_t u +u \cdot\nabla u - \nu \Delta u = - \nabla p + j \times B, &\div u =0,&
			\\
			&\frac{1}{c} \partial_t E - \nabla \times B =- j , &\div B = 0,&
			\\
			&\frac{1}{c} \partial_t B + \nabla \times E  = 0 , &&
			\\
			&j= \sigma \big( cE + u \times B\big), &&
		\end{aligned}
	\end{cases}
\end{equation}
and a corresponding incompressible Euler--Maxwell system is given by setting $\nu=0$. We refer to \cite{ag20,gim,MN} for details on the contruction of global solutions to \eqref{NSM2}, with $\nu>0$.

Unlike \eqref{NSM0} and \eqref{NSM1}, it is to be emphasized that this model is not obtained as an asymptotic regime of Vlasov--Maxwell--Boltzmann systems, as shown in \cite{as}. Furthermore, when compared to \eqref{NSM0} and \eqref{NSM1}, it has the major drawback of not providing a strong control of $\div E$. For this reason, we do not make any claim concerning the extension of our work to the above model. It would, however, be interesting to clarify the well-posedness of the nonviscous version of \eqref{NSM2}.

Finally, we observe that there is also a rich family of compressible Euler--Maxwell systems which are commonly used to model the behavior of plasmas. The study of such systems is challenging and corresponding results tend to focus on the stability of smooth solutions near specific equilibrum states. We refer to the articles \cite{gm14} and \cite{gip16} for foundational results on three-dimensional compressible Euler--Maxwell systems. We note that the results therein do not require any specific vector structure, such as the normal structure \eqref{structure:2dim}. However, they are sensitive to the speed of light $c$ and, therefore, may not provide uniform bounds as $c$ tends to infinity.

\subsection{Strategy of proof}\label{strategy:0}

We lay out now the strategy and the key ideas leading to the proof of Theorem \ref{thm-EM}, which will be implemented later on, in Section \ref{section:perfect fluid}, to establish the more precise Theorems \ref{main:1}, \ref{main:2} and \ref{main:3}.

Observe first that, even if we add a dissipation term $-\Delta u$ to the first equation of \eqref{EM}, thereby yielding the incompressible Navier--Stokes--Maxwell system \eqref{NSM0}, it is still unknown whether or not global weak solutions do exist when the initial data are only square-integrable. This is due to the lack of strong compactness (or regularity) in electromagnetic fields $(E,B)$, combined with the lack of stability of the source term $j\times B$ in weak topologies (see \cite{ag20} for further details). The same difficulty persists in the inviscid version of the same system, which stems from the propagation of singularities in Maxwell's system, as a result of its hyperbolic nature. The construction of solutions in $L^2$ to \eqref{EM} is thus highly challenging---all the more so than in the viscous case.

One should therefore treat this system in some higher-regularity spaces. To this end, inspired by known results on the well-posedness of the two-dimensional Euler system \eqref{Euler:1}, we shall look at the equivalent vorticity formulation of \eqref{EM}, which reads as
\begin{equation}\label{EM-omega:form}
	\begin{cases}
		\begin{aligned}
			&\partial_t \omega +u \cdot\nabla \omega = -j \cdot\nabla B, &\div u &=0,
			\\
			&\frac{1}{c} \partial_t E - \nabla \times B + \sigma c E =- \sigma P(u\times B) , &\div E &= 0,
			\\
			&\frac{1}{c} \partial_t B + \nabla \times E  = 0 , &\div B &= 0,
			\\
			&j= \sigma \big( cE + P(u \times B)\big), &\div j &= 0,
		\end{aligned}
	\end{cases}
\end{equation}
where $\omega \bydef \nabla \times u$ and $u$ can be reconstructed from $\omega$ through the Biot--Savart law
\begin{equation}\label{biot}
	u=-\Delta^{-1}\nabla\times\omega.
\end{equation}
Observe that the normal structure \eqref{structure:2dim} has been used in \eqref{EM-omega:form} to write $\nabla\times(j\times B)=-j\cdot\nabla B$. This is crucial.

Much of our analysis of \eqref{EM-omega:form} will hinge on the dispersive properties of the damped Maxwell system
\begin{equation}\label{Maxwell:system:*}
	\begin{cases}
		\begin{aligned}
			\frac{1}{c} \partial_t E - \nabla \times B + \sigma c E & =- \sigma P( u \times B),
			\\
			\frac{1}{c} \partial_t B + \nabla \times E & = 0,
			\\
			\div u=\div E = \div B& =0.
		\end{aligned}
	\end{cases}
\end{equation}
This will require us to interpret the role of the velocity field $u$ in \eqref{Maxwell:system:*}, in the spatial variable $x$, as that of a coefficient in the algebra $L^\infty_x\cap\dot H^1_x$ (or some weaker variant), thereby allowing us to view \eqref{Maxwell:system:*} as a linear system in $(E,B)$ and produce closed estimates on the electromagnetic field.

To be precise, the treatment of the source term $-\sigma P(u\times B)$ in \eqref{Maxwell:system:*} will necessitate the control of the velocity field $u$ in a suitable algebra acting on $\dot{H}_x^s$, for appropriate values of $s$. In particular, according to the classical paradifferential product law
\begin{equation*}
	\norm {fg}_{\dot{H}^s(\mathbb{R}^d)} \lesssim \norm f_{L^\infty \cap \dot{B}^\frac{d}{2}_{2,\infty}(\mathbb{R}^d)} \norm g_{\dot{H}^s(\mathbb{R}^d)},
\end{equation*}
which holds for any $s\in (-\frac{d}{2}, \frac{d}{2})$ and $d\geq 1$, it will be natural to seek the control of $u$ (in the space variable) in the weaker algebra $L^\infty \cap \dot{B}^1_{2,\infty}(\mathbb{R}^2)$ (see Appendix \ref{besov:1} for a definition of Besov spaces).

In the context of two-dimensional viscous flows, such a control is expected in view of the strong bounds provided by the energy dissipation inequality.
For example, in \cite[Theorem 1.2]{ag20}, the existence of weak solutions to a two-dimensional incompressible Navier--Stokes--Maxwell was established by proving a uniform control of the velocity field in the algebra $L^\infty_x\cap\dot H^1_x(\mathbb{R}^2)$. More precisely, by building upon the methods from \cite{MN}, it was shown therein (see \cite[Proposition 2.1]{ag20}) that the control of the velocity field in the space $L_t^2(L^\infty_x\cap\dot H^1_x)$ was sufficient to propagate some $\dot H^s$-regularity, with $-1<s<1$, in Maxwell's equations \eqref{Maxwell:system:*}, uniformly as $c\to\infty$.

In the setting of two-dimensional incompressible electrically-conducting ideal fluids (i.e., plasmas), which is the focus of our work, global energy estimates are nowhere near as good as their viscous counterpart and, thus, fail to yield the control of $u$ in a useful algebra. Instead, we need to take Yudovich's approach of propagating the $L^2_x\cap L^p_x$-norm of the vorticity $\omega$, for some given $p>2$, by exploiting the transport equation
\begin{equation}\label{transport:1}
	\partial_t \omega +u \cdot\nabla \omega = -j \cdot\nabla B,
\end{equation}
thereby providing a bound on $u$ in the algebra $L_t^\infty(L^\infty_x\cap\dot H^1_x)$, by classical Sobolev embeddings combined with standard estimates on the Biot--Savart law \eqref{biot}.
We refer to \cite[Section 7.2]{bcd11} for a modern treatment of global existence results for two-dimensional perfect incompressible fluids and the Yudovich Theorem.

In particular, elementary estimates on transport equations, which are performed in detail in Section \ref{section:vorticity}, show that the control of $\omega$ in $L^\infty_tL^p_x$ follows from the control of the initial vorticity $\omega_0$ and the nonlinear source term $j\cdot\nabla B$ in $L^p_x$ and $L^1_tL^p_x$, respectively. Since $j$ is naturally bounded in $L^2_{t,x}$, by virtue of the energy inequality \eqref{energy-inequa}, we conclude that $\nabla B$ should be controlled in $L^2_tL^\infty_x$.

Now, experience shows that such a Lipschitz bound on $B$ cannot easily follow from energy estimates on the wave system \eqref{Maxwell:system:*}.
Indeed, energy estimates on hyperbolic systems are typically performed in $L^2_x$. Therefore, in order to control $\nabla B$ in $L^\infty_x$, an energy estimate on \eqref{Maxwell:system:*} would lead us, in view of classical Sobolev embeddings, to seek a bound of $B$ in $H^{2+\delta}_x$, with a small parameter $\delta>0$. To that end, the source term $-\sigma P(u\times B)$ in \eqref{Maxwell:system:*}, would also need to be controlled in $H^{2+\delta}_x$. However, employing paradifferential calculus to control $u\times B$ would require that $\nabla u$ be bounded in $L^\infty_x\cap \dot H^1_x$, at least. Unfortunately, such uniform bounds on perfect incompressible two-dimensional flows are largely out of reach in our context.
This is where the damped dispersive properties of \eqref{Maxwell:system:*}, on the whole Euclidean plane $\mathbb{R}^2$, come into play.

Maxwell's system \eqref{Maxwell:system:*} can be rewritten as a system of wave equations (more on this later on, see \eqref{maxwell:2:0}). Thus, heuristically, one expects to be able to employ Strichartz estimates for the wave equation to control the electromagnetic field $(E,B)$. In particular, by paying close attention to the admissibility criteria of functional spaces in Strichartz estimates (see \cite[Section 8.3]{bcd11} or \cite{kt98}), one observes that it is possible to control the Lipschitz norm of a solution to a two-dimensional wave equation, provided one can bound $\frac 74$ derivatives of the initial data and the source term in some appropriate functional spaces (in some Besov spaces, for instance) of $L^2$ space-integrability. (For simplicity, we have omitted here the consideration of time integrability in Strichartz estimates and focused solely on space regularity and integrability.) Loosely speaking, such an estimate is better than a Sobolev embedding, which would require the control of over two derivatives in $L^2(\mathbb{R}^2)$ in order to bound a Lipschitz norm.  This should give the reader some intuition concerning the special role played by the regularity parameter $s=\frac 74$ in Theorem \ref{thm-EM}.

Thus, so far, our strategy seems to yield some promising closed estimates. Indeed, on the one hand, the transport equation \eqref{transport:1} gives us a bound on the $L^\infty_t(L^2_x\cap L^p_x)$-norm of the vorticity $\omega$ provided $\nabla B$ is controlled in $L^2_tL^\infty_x$, while, on the other hand, a control of $\nabla B$ in $L^2_tL^\infty_x$ can be achieved through dispersive estimates on the wave system \eqref{Maxwell:system:*} if the velocity field $u$ is sufficiently smooth (at least $L^\infty_t(L^\infty_x\cap\dot H^1_x)$, say).

However, such a roadmap may not lead to global estimates in time. To see this, we need to take a closer look at the temporal norms associated with our strategy. Specifically, it is important to note that the classical Strichartz estimates for the two-dimensional wave equation do not actually give a global control of $\nabla B$ in $L^2_tL^\infty_x$. Instead, they only allow us to control $\nabla B$ in $L^4_tL^\infty_x$ globally, which then leads to a control in $L^2_tL^\infty_x$ locally in time.
This difficulty is solved by complementing our strategy with a careful study of the damping phenomenon in \eqref{Maxwell:system:*} produced by the term $\sigma c E$. To that end, we provide, in Section \ref{section:damping:0}, a robust analysis of the damping effect on general semigroup flows, which is formulated in precise terms in Lemma \ref{damping:1} (the Damping Lemma). We also give applications of the Damping Lemma to parabolic and dispersive equations, in Sections \ref{damped:parabolic:1} and \ref{section:damped:strichartz:1}, respectively.

Concerning Maxwell's system \eqref{Maxwell:system:*}, the ensuing time decay of the electromagnetic field is encapsulated in Corollary \ref{cor:maxwell}. It is shown therein that \eqref{Maxwell:system:*} forms a damped hyperbolic system which is best understood by decomposing the frequencies of the solutions relatively to the magnitude of the speed of light $c>0$.

Indeed, by appropriately combining Amp\`ere and Faraday's equations from \eqref{Maxwell:system:*} and using that $\nabla\times(\nabla\times B)=-\Delta B$, observe that $B$ solves the damped wave equation
\begin{equation}\label{maxwell:2:0}
	\frac 1{c^2}\partial_t^2 B+\sigma \partial_t B-\Delta B =- \sigma \nabla \times( u \times B),
\end{equation}
where the damping term $\sigma \partial_t B$ comes from the term $\sigma cE$ in \eqref{Maxwell:system:*}.

Heuristically, since waves described by \eqref{maxwell:2:0} typically propagate with a characteristic speed $c$, it is then natural to expect a consistent hyperbolic behavior of the solutions of \eqref{Maxwell:system:*} on the range of frequencies larger than a suitable multiple of the speed of light $c$. In particular, Corollary \ref{cor:maxwell} will confirm that solutions to \eqref{maxwell:2:0} enjoy dispersive properties, for those high frequencies, which are analogous to the non-damped case (obtained by setting $\sigma=0$ in \eqref{maxwell:2:0}) with drastically improved long-time integrability.

On the remaining range of frequencies, i.e., on frequencies slower than $c$, the same result will establish that the behavior of solutions to \eqref{Maxwell:system:*} is largely dictated by the heat equation
\begin{equation*}
	\sigma \partial_t B-\Delta B =- \sigma \nabla \times( u \times B) ,
\end{equation*}
which is formally achieved in the asymptotic regime $c\to\infty$ from \eqref{maxwell:2:0}.

All in all, the application of the sharp damped dispersive estimates from Section \ref{section:damping:0} to Maxwell's equations \eqref{Maxwell:system:*} will allow us to obtain closed estimates on the incompressible Euler--Maxwell system \eqref{EM} which hold globally and lead to Theorem \ref{thm-EM}. The precise nonlinear analysis of \eqref{EM} is detailed in Section \ref{section:perfect fluid} with complete proofs of our main theorems.

It is difficult to pinpoint the exact source of the breakdown of our proofs for small values of light velocity $c$. However, one can argue that the degeneracy of Maxwell's system in the limit $c\to 0$ results in a loss of the	 damped dispersive properties which are central to our nonlinear analysis. We believe that this provides some evidence that our method cannot be extended to the whole range of $c>0$. Nevertheless, we are hopeful that other techniques may be used to construct solutions in the remaining range of light velocities.

\subsection{Notation}\label{notation:sec}

Allow us to clarify a few notations which will be used repeatedly throughout this article.

First of all, for clarity and convenience, note that all relevant functional spaces of Besov and Chemin--Lerner types are introduced in precise detail in Appendix \ref{besov:1}.

Next, Leray's projector $P:L^2(\mathbb{R}^3;\mathbb{R}^3)\to L^2(\mathbb{R}^3;\mathbb{R}^3)$ onto divergence-free vector fields, which is used in \eqref{EM}, and the corresponding orthogonal projector $P^\perp=\mathrm{Id}-P$ onto conservative fields are given by
\begin{equation*}
	P= \mathrm{Id}-\Delta^{-1}\nabla\mathrm{div},
	\qquad P^\perp=\Delta^{-1}\nabla\mathrm{div}.
\end{equation*}

Finally, when necessary, we will employ the letter $C$ to denote a generic constant, which is allowed to differ from one estimate to another and we will resort to the use of indices to distinguish specific constants. We will also often write $A\lesssim B$ to denote $A\leq CB$, for some positive constant $C$ which only depends on fixed parameters, and $A\sim B$ whenever $A\lesssim B$ and $B\lesssim A$ are simultaneously true.


\section{The effect of damping on semigroup flows}\label{section:damping:0}

Here, we analyze the effect of damping on evolution flows, which are generally described by semigroups. More specifically, in Section \ref{section:damping:1}, we begin by establishing a robust and general result---called the Damping Lemma---showing how damping terms act on integral operators. Then, in Sections \ref{section:damped:parabolic:1} and \ref{section:damped:strichartz:1}, this result is applied to the context of damped parabolic and Strichartz estimates, which will be crucial to our analysis of Maxwell's system in Section \ref{section:perfect fluid}. In particular, in Section \ref{section:damped:strichartz:1}, we give complete and sharp formulations of Strichartz estimates for the damped Schr\"odinger, half-wave, wave and Maxwell equations in Euclidean spaces.

\subsection{The damping lemma}\label{section:damping:1}

The result below provides a general and robust principle allowing us to take into account the influence of a damping term $e^{-\alpha t}$, with $\alpha> 0$, on an integral operator.

\begin{lem}[The Damping Lemma]\label{damping:1}
	Let $X$ and $Y$ be Banach spaces and, for each $s,t\in [0,T)$, with $T>0$, let $K(t,s):X\to Y$ be an operator-valued kernel from $X$ to $Y$, such that
	\begin{equation*}
		K(t,s)\in L^1\big([0,T)\times [0,T);\mathcal{L}(X,Y)\big),
	\end{equation*}
	where $\mathcal{L}(X,Y)$ denotes the Banach space of bounded linear operators from $X$ to $Y$.
	Further suppose that there are $0< p_0\leq q_0\leq\infty$, with $q_0\geq 1$, and a constant $A>0$ such that the estimate
	\begin{equation}\label{damping:2}
		\norm{\int_0^T \chi(t,s)K(t,s)f(s)ds}_{L^{q_0}([0,T);Y)}
		\leq A \norm{f}_{L^{p_0}([0,T);X)}
	\end{equation}
	holds, for all $f\in L^{p_0}([0,T);X)$ and any $\chi(t,s)\in L^\infty([0,T)^2;\mathbb{R})$, with $\norm{\chi}_{L^\infty}\leq 1$.
	
	Then, for any $\alpha\geq 0$ and $p_0\leq p\leq q\leq q_0$, with $q\geq 1$, one has the damped estimate
	\begin{equation*}
		\norm{\int_0^T e^{-\alpha |t-s|}\chi(t,s)K(t,s)f(s)ds}_{L^{q}([0,T);Y)}
		\leq C_\beta A\left(\frac{T}{1+\alpha T}\right)^\beta\norm{f}_{L^{p}([0,T);X)},
	\end{equation*}
	for all $f\in L^{p}([0,T);X)$ and any $\chi(t,s)\in L^\infty([0,T)^2;\mathbb{R})$, with $\norm{\chi}_{L^\infty}\leq 1$,
	where $\beta\geq 0$ is defined by
	\begin{equation*}
		\beta=\frac 1q-\frac 1{q_0}+\frac 1{p_0}-\frac 1{p}
	\end{equation*}
	and $C_\beta>0$ only depends on $\beta$.
\end{lem}

\begin{proof}
	For $\alpha=0$, the result follows straightforwardly from H\"older's inequality on the domain $[0,T)$, for all integrability parameters merely satisfying $0<q\leq q_0\leq\infty$ and $0<p_0\leq p\leq\infty$. We assume now that $\alpha>0$ and $0<p_0\leq p\leq q\leq q_0\leq\infty$, with $q\geq 1$.
	
	For convenience of notation, we extend the definition of the kernel $K$ and the functions $\chi$ and $f$ to all real values of $t$ and $s$ by setting them equal to zero whenever $t$ or $s$ fall outside of the interval $[0,T)$.
	
	We begin with the use of a partition
	\begin{equation*}
		\mathds{1}_{\{t\neq s\}}=\sum_{j\in\mathbb{Z}}\mathds{1}_{\{2^j\leq |t-s|< 2^{j+1}\}}
	\end{equation*}
	to deduce that
	\begin{equation}\label{damping:3}
		\norm{\int_0^T e^{-\alpha |t-s|}\chi(t,s)K(t,s)f(s)ds}_{L^{q}(\mathbb{R};Y)}
		\leq \sum_{\substack{j\in\mathbb{Z}\\2^j< T}}e^{-\alpha 2^j}
		\norm{\int_0^T \chi_j(t,s)K(t,s)f(s)ds}_{L^{q}(\mathbb{R};Y)},
	\end{equation}
	where we have denoted
	\begin{equation*}
		\chi_j(t,s)=\mathds{1}_{\{2^j\leq |t-s|< 2^{j+1}\}}e^{-\alpha (|t-s|-2^j)}\chi(t,s).
	\end{equation*}
	Observe that $\norm{\chi_j}_{L^\infty}\leq 1$.
	
	Then, we further decompose the domain of $t$ into the disjoint union
	\begin{equation*}
		\bigcup_{k\in\mathbb{Z}} \big\{2^jk\leq t<2^j(k+1)\big\}
	\end{equation*}
	to write that
	\begin{equation}\label{damping:4}
		\begin{aligned}
			\norm{\int_0^T \chi_j(t,s)K(t,s)f(s)ds}_{L^{q}(\mathbb{R};Y)}
			&=\norm{\norm{\int_0^T \chi_jKf(s)ds}_{L^{q}\big(\big[2^jk,2^j(k+1)\big);Y\big)}}_{\ell^q(k\in\mathbb{Z})}
			\\
			&\leq 2^{j\left(\frac 1q-\frac 1{q_0}\right)}
			\norm{\norm{\int_0^T \chi_jKf(s)ds}_{L^{q_0}\big(\big[2^jk,2^j(k+1)\big);Y\big)}}_{\ell^q(k\in\mathbb{Z})},
		\end{aligned}
	\end{equation}
	where we employed H\"older's inequality.
	
	Now, notice that
	\begin{equation*}
		2^j(k-2)<t-|t-s|\leq s\leq t+|t-s|<2^j(k+3),
	\end{equation*}
	whenever $2^j\leq |t-s|< 2^{j+1}$ and $2^jk\leq t<2^j(k+1)$. In particular, using \eqref{damping:2}, it follows that
	\begin{equation}\label{damping:5}
		\begin{aligned}
			\norm{\int_0^T \chi_jKf(s)ds}_{L^{q_0}\big(\big[2^jk,2^j(k+1)\big);Y\big)}
			&\leq
			\norm{\int_0^T \chi_jKf(s)\mathds{1}_{\{2^j(k-2)<s<2^j(k+3)\}}ds}_{L^{q_0}(\mathbb{R};Y)}
			\\
			&\leq
			A\sum_{n=-2}^2\norm{f}_{L^{p_0}\big(\big[2^j(k+n),2^j(k+1+n)\big);X\big)}
			\\
			&\leq A2^{j\left(\frac 1{p_0}-\frac 1{p}\right)}
			\sum_{n=-2}^2\norm{f}_{L^{p}\big(\big[2^j(k+n),2^j(k+1+n)\big);X\big)},
		\end{aligned}
	\end{equation}
	where we applied H\"older's inequality, again.
	
	All in all, combining \eqref{damping:3}, \eqref{damping:4} with \eqref{damping:5}, and recalling that $\ell^p\subset\ell^q$, because $p\leq q$, we infer that
	\begin{equation}\label{damping:6}
		\begin{aligned}
			\norm{\int_0^T e^{-\alpha |t-s|}\chi(t,s)K(t,s)f(s)ds}_{L^{q}(\mathbb{R};Y)}
			\hspace{-30mm}&
			\\
			&\leq 5A\sum_{\substack{j\in\mathbb{Z}\\2^j< T}}e^{-\alpha 2^j}
			2^{j\left(\frac 1q-\frac 1{q_0}+\frac 1{p_0}-\frac 1{p}\right)}
			\norm{\norm{f}_{L^{p}\big(\big[2^jk,2^j(k+1)\big);X\big)}}_{\ell^q(k\in\mathbb{Z})}
			\\
			&\leq 5A\norm{f}_{L^{p}(\mathbb{R};X)}\sum_{\substack{j\in\mathbb{Z}\\2^j< T}}e^{-\alpha 2^j}
			2^{j\left(\frac 1q-\frac 1{q_0}+\frac 1{p_0}-\frac 1{p}\right)}.
		\end{aligned}
	\end{equation}
	There only remains to evaluate the constant resulting from the above sum in $j\in\mathbb{Z}$.
	If $p=p_0$ and $q=q_0$, the lemma trivially holds and there is nothing to prove. Thus, we may assume that $\beta>0$, thereby ensuring that the sum converges.
	
	Now, observing that the function $e^{-x}(1+x)^{1+\beta}$ reaches its maximum on $[0,\infty)$ at $x=\beta$, we obtain that
	\begin{equation}\label{damping:7}
		\begin{aligned}
			e^{\beta}(1+\beta)^{-(1+\beta)}\sum_{\substack{j\in\mathbb{Z}\\2^j< T}}e^{-\alpha 2^j}2^{j\beta}
			&\leq
			\sum_{\substack{j\in\mathbb{Z}\\2^j< T}}\frac{2^{j\beta}}{(1+\alpha 2^j)^{1+\beta}}
			\leq
			\sum_{\substack{j\in\mathbb{Z}\\2^j< T}}\int_{j-1}^j\frac{2^{(1+u)\beta}}{(1+\alpha 2^u)^{1+\beta}} du
			\\
			&\leq
			\frac{2^\beta}{\log 2} \int_{0}^{T}\frac{x^{\beta-1}}{(1+\alpha x)^{1+\beta}} dx
			=\frac{2^\beta}{\beta\log 2}\left(\frac{T}{1+\alpha T}\right)^\beta.
		\end{aligned}
	\end{equation}
	Therefore, incorporating \eqref{damping:7} into the estimate \eqref{damping:6} concludes the proof of the lemma.
\end{proof}

\subsection{Damped parabolic estimates}\label{section:damped:parabolic:1}

Let us consider the general solution $w(t,x)$ of a damped heat equation on the Euclidean space $\mathbb{R}^d$, for any dimension $d\geq 1$,
\begin{equation}\label{heat:1}
	\left\{
	\begin{aligned}
		\partial_t w+\alpha w-\Delta w&=f,
		\\
		w_{|t=0}&=w_0,
	\end{aligned}
	\right.
\end{equation}
where $(t,x)\in[0,T)\times\mathbb{R}^d$, with $T>0$ ($T=\infty$ may also be considered), the damping constant satisfies $\alpha\geq 0$, the right-hand side $f(t,x)$ is a source term and $w_0(x)$ is an initial datum.

Such equations naturally appear in dissipative physical systems. For instance, the heat equation \eqref{heat:1} provides the linear structure of the damped incompressible Navier--Stokes equations, which arise from hydrodynamic regimes of inelastic particle systems.

Using standard semigroup notation, the solution $w(t,x)$ can be represented as
\begin{equation}\label{heat:3}
	w(t)=e^{-t(\alpha-\Delta)}w_0+\int_0^te^{-(t-s)(\alpha-\Delta)}f(s)ds.
\end{equation}
We are now going explore the jungle of parabolic smoothing estimates in Besov spaces for \eqref{heat:3}, by first reviewing the available results for the case $\alpha=0$ and, then, extending these results to the setting $\alpha>0$.

When $f\equiv 0$ and $\alpha=0$, direct parabolic estimates on the semigroup $e^{t\Delta}$ yield the following result.
\begin{prop}\label{damped:parabolic:1}
	Let $\sigma\in\mathbb{R}$, $p\in[1,\infty]$ and $q\in[1,\infty]$. If $\alpha=0$, $w_0$ belongs to $\dot B^\sigma_{p,q}$ and $f\equiv 0$, then the solution of the heat equation \eqref{heat:1} satisfies
	\begin{equation*}
		\|e^{t\Delta} w_0\|_{L^\infty\big([0,\infty);\dot B_{p,q}^{\sigma}\big)}
		\lesssim
		\|w_0\|_{\dot B_{p,q}^{\sigma}}.
	\end{equation*}
	Furthermore, if $q<\infty$, one also has the estimate
	\begin{equation*}
		\|e^{t\Delta} w_0\|_{L^q\big([0,\infty);\dot B_{p,1}^{\sigma+\frac 2q}\big)}
		\lesssim
		\|w_0\|_{\dot B_{p,q}^{\sigma}}.
	\end{equation*}
\end{prop}

\begin{rem}
	The above result somewhat reinforces the estimate
	\begin{equation*}
		\|e^{t\Delta} w_0\|_{L^q\big([0,\infty);L^p \big)}
		\lesssim
		\|w_0\|_{\dot B_{p,q}^{-\frac 2q}},
	\end{equation*}
	for any $1\leq p,q\leq\infty$, which is commonly found in the literature (see \cite[Theorem 2.34]{bcd11}, for instance).
\end{rem}

\begin{rem}
	Note that taking $p=q=2$ in the above proposition yields the estimate
	\begin{equation*}
		\|e^{t\Delta} w_0\|_{L^2\big([0,\infty);\dot B_{2,1}^{\sigma+1}\big)}
		\lesssim
		\|w_0\|_{\dot H^\sigma},
	\end{equation*}
	where we used that $\dot H^\sigma=\dot B^\sigma_{2,2}$ (see the appendix for a precise definition of all relevant homogeneous spaces).
\end{rem}

\begin{rem}
	Throughout this section, we will routinely use the basic estimate
	\begin{equation}\label{heat:2}
		\|e^{t\Delta}\Delta_ku\|_{L^p}\leq Ce^{-C_*t2^{2k}}\|\Delta_k u\|_{L^p},
	\end{equation}
	for any $t>0$, $p\in[1,\infty]$ and any dyadic block $\Delta_k$, with $k\in\mathbb{Z}$, where $C$ and $C_*$ are positive independent constants.
	We refer to \cite[Lemma 2.4]{bcd11} for a justification of \eqref{heat:2}.
\end{rem}

\begin{proof}
	The first part of the statement is a straightforward consequence of the definition of the homogeneous Besov norm. More precisely, using \eqref{heat:2}, we obtain
	\begin{equation*}
		\begin{aligned}
			\|e^{t\Delta} w_0\|_{\dot B_{p,q}^{\sigma}}
			&=
			\left(\sum_{k\in\mathbb{Z}}\left(2^{k\sigma}\|e^{t\Delta} \Delta_k w_0\|_{L^p}\right)^q\right)^\frac 1q
			\\
			&\lesssim
			\left(\sum_{k\in\mathbb{Z}}e^{-C_*t2^{2k}}\left(2^{k\sigma}\|\Delta_k w_0\|_{L^p}\right)^q\right)^\frac 1q
			\lesssim \|w_0\|_{\dot B_{p,q}^{\sigma}},
		\end{aligned}
	\end{equation*}
	which, upon taking the supremum in $t>0$, concludes the justification of the first estimate.
	
	The second part of the statement is more subtle. Indeed, assuming now that $q<\infty$ and using \eqref{heat:2}, we find that
	\begin{equation*}
		\begin{aligned}
			\|e^{t\Delta} w_0\|_{\dot B_{p,1}^{\sigma+\frac 2q}}
			&=
			\sum_{k\in\mathbb{Z}}2^{k(\sigma+\frac 2q)}\|e^{t\Delta} \Delta_k w_0\|_{L^p}
			\\
			&\lesssim
			\sum_{k\in\mathbb{Z}}e^{-C_*t2^{2k}}2^{k(\sigma+\frac 2q)}\|\Delta_k w_0\|_{L^p}.
		\end{aligned}
	\end{equation*}
	Next, further employing H\"older's inequality and taking a fixed positive value $\lambda>0$ such that $(q-1)\lambda<1$, we infer that
	\begin{equation}\label{parabolic:1}
		\begin{aligned}
			\|e^{t\Delta} w_0\|_{\dot B_{p,1}^{\sigma+\frac 2q}}
			&\lesssim
			\left(\sum_{k\in\mathbb{Z}}\big(t2^{2k}\big)^\lambda e^{-C_*t2^{2k}}\right)^\frac{q-1}{q}
			\\
			&\quad\times\left(\sum_{k\in\mathbb{Z}}\frac 1te^{-C_*t2^{2k}}\big(t2^{2k}\big)^{1-\lambda(q-1)}
			\left(2^{k\sigma}\|\Delta_k w_0\|_{L^p}\right)^q\right)^\frac 1q.
		\end{aligned}
	\end{equation}
	
	Now, for any positive $t$, considering the unique $j\in\mathbb{Z}$ such that $2^{2j}\leq t<2^{2(j+1)}$, we find, since $\lambda>0$, that
	\begin{equation}\label{kernel:1}
		\begin{aligned}
			\sup_{t>0}\sum_{k\in\mathbb{Z}}\big(t2^{2k}\big)^\lambda e^{-C_*t2^{2k}}
			&\leq
			2^{2\lambda}\sup_{j\in\mathbb{Z}}\sum_{k\in\mathbb{Z}}\big(2^{2(j+k)}\big)^\lambda e^{-C_*2^{2(j+k)}}
			\\
			&=
			2^{2\lambda}\sum_{k\in\mathbb{Z}}\big(2^{2k}\big)^\lambda e^{-C_*2^{2k}}
			<\infty,
		\end{aligned}
	\end{equation}
	whereas, since $\lambda(q-1)<1$, we evaluate that
	\begin{equation*}
		\int_0^\infty e^{-C_*t2^{2k}}\big(t2^{2k}\big)^{1-\lambda(q-1)} \frac {dt}t
		=\int_0^\infty e^{-C_*t}t^{-\lambda(q-1)} dt<\infty.
	\end{equation*}
	Therefore, integrating \eqref{parabolic:1} in time, we finally arrive at the estimate
	\begin{equation*}
		\begin{aligned}
			\|e^{t\Delta} w_0\|_{L^q_t\dot B_{p,1}^{\sigma+\frac 2q}}
			&\lesssim
			\sup_{t>0}\left(\sum_{k\in\mathbb{Z}}\big(t2^{2k}\big)^\lambda e^{-C_*t2^{2k}}\right)^\frac{q-1}{q}
			\\
			&\quad\times\left(\sum_{k\in\mathbb{Z}}\int_0^\infty e^{-C_*t2^{2k}}\big(t2^{2k}\big)^{1-\lambda(q-1)}\frac{dt}t
			\left(2^{k\sigma}\|\Delta_k w_0\|_{L^p}\right)^q\right)^\frac 1q
			\\
			&\lesssim \|w_0\|_{\dot B_{p,q}^{\sigma}},
		\end{aligned}
	\end{equation*}
	which concludes the proof of the proposition.
\end{proof}

In view of the preceding result, the effect of the damping term $e^{-\alpha t}$ on the initial data can be taken into account through a straightforward application of H\"older's inequality, thereby providing the following corollary.

\begin{cor}
	Let $\sigma\in\mathbb{R}$, $p\in[1,\infty]$ and $q\in[1,\infty]$. If $\alpha\geq 0$, $w_0$ belongs to $\dot B^\sigma_{p,q}$ and $f\equiv 0$, then the solution of the heat equation \eqref{heat:1} satisfies
	\begin{equation*}
		\|e^{-t(\alpha-\Delta)} w_0\|_{L^m\big([0,T);\dot B_{p,q}^{\sigma}\big)}
		\lesssim \left(\frac T{1+\alpha T}\right)^\frac 1m
		\|w_0\|_{\dot B_{p,q}^{\sigma}},
	\end{equation*}
	for every $0<m\leq\infty$.
	Furthermore, if $0<m\leq q<\infty$, one also has the estimate
	\begin{equation*}
		\|e^{-t(\alpha-\Delta)} w_0\|_{L^m\big([0,T);\dot B_{p,1}^{\sigma+\frac 2q}\big)}
		\lesssim \left(\frac T{1+\alpha T}\right)^{\frac 1m-\frac 1q}
		\|w_0\|_{\dot B_{p,q}^{\sigma}}.
	\end{equation*}
\end{cor}

\begin{proof}
	A direct use of H\"older's inequality followed by an application of Proposition \ref{damped:parabolic:1} yields that
	\begin{equation*}
		\|e^{-t(\alpha-\Delta)} w_0\|_{L^m\big([0,T);\dot B_{p,q}^{\sigma}\big)}
		\leq
		\|e^{-t\alpha}\|_{L^m([0,T))}
		\|e^{t\Delta} w_0\|_{L^\infty\big([0,T);\dot B_{p,q}^{\sigma}\big)}
		\lesssim \left(\frac T{1+\alpha T}\right)^\frac 1m
		\|w_0\|_{\dot B_{p,q}^{\sigma}},
	\end{equation*}
	for all $0<m\leq\infty$, and
	\begin{equation*}
		\begin{aligned}
			\|e^{-t(\alpha-\Delta)} w_0\|_{L^m\big([0,T);\dot B_{p,1}^{\sigma+\frac 2q}\big)}
			& \leq
			\|e^{-t\alpha}\|_{L^{(\frac 1m-\frac 1q)^{-1}}([0,T))}
			\|e^{t\Delta} w_0\|_{L^q\big([0,T);\dot B_{p,1}^{\sigma+\frac 2q}\big)}
			\\
			& \lesssim \left(\frac T{1+\alpha T}\right)^{\frac 1m-\frac 1q}
			\|w_0\|_{\dot B_{p,q}^{\sigma}},
		\end{aligned}
	\end{equation*}
	for all $0<m\leq q<\infty$,
	which completes the proof.
\end{proof}

Parabolic estimates are more involved when one includes a nonzero source term $f$. The coming results contain a wide range of smoothing estimates for the inhomogeneous heat equation. In preparation of these results, in order to reach a broader range of applicability, we are now going to introduce symbols
\begin{equation*}
	a(t,s,\xi)\in L^\infty([0,T)\times [0,T)\times\mathbb{R}^d)
\end{equation*}
which act as multipliers on the Fourier variable $\xi\in\mathbb{R}^d$ and are dependent on the time variables $t,s\in [0,T)$, thereby leading to time-dependent Fourier multipliers $a(t,s,D)$.

\begin{defin}
	For a given $1\leq p\leq \infty$, we say that $a(t,s,D)$ is \emph{bounded} if there is a constant $C_{a}>0$, independent of $t$ and $s$, such that
	\begin{equation}\label{multiplier:1}
		\|a(t,s,D)f\|_{\dot B^0_{p,\infty}}\leq C_{a}\|f\|_{\dot B^0_{p,\infty}},
	\end{equation}
	for every $f\in\dot B^0_{p,\infty}(\mathbb{R}^d)$ and almost every $(t,s)\in [0,T)^2$. That is, the multiplier $a(t,s,D)$ is bounded if it is bounded over the Besov space $\dot B^0_{p,\infty}(\mathbb{R}^d)$, uniformly in $t$ and $s$. The \emph{norm} of $a(t,s,D)$, which we denote by
	\begin{equation*}
		\norm{a(t,s,D)}_{M_p},
	\end{equation*}
	is defined as the smallest possible constant $C_a>0$ that fits in \eqref{multiplier:1}.
\end{defin}

\begin{rem}
	Equivalently, it is readily seen that \eqref{multiplier:1} holds if and only if there is a constant $C_a'>0$, independent of $t$ and $s$, such that
	\begin{equation}\label{multiplier:2}
		\|a(t,s,D)\Delta_k f\|_{L^p}\leq C_a'\|f\|_{L^p},
	\end{equation}
	for every $k\in\mathbb{Z}$, $f\in L^p(\mathbb{R}^d)$ and almost every $(t,s)\in [0,T)^2$.
\end{rem}

\begin{rem}
	Observe that \eqref{multiplier:1} and \eqref{multiplier:2} hold if and only if one has
	\begin{equation*}
		\|a(t,s,D)f\|_{\dot B^\sigma_{p,q}}\leq C_{a,\sigma,q}\|f\|_{\dot B^\sigma_{p,q}},
	\end{equation*}
	with $C_{a,\sigma,q}>0$, for all $\sigma\in\mathbb{R}$, $q\in [1,\infty]$ and every $f\in\dot B^\sigma_{p,q}(\mathbb{R}^d)$.
\end{rem}

Since the space of Fourier multipliers over $L^2(\mathbb{R}^d)$ is isomorphic to $L^\infty(\mathbb{R}^d)$, it is readily seen, when $p=2$, that proving \eqref{multiplier:1} and \eqref{multiplier:2} is equivalent to establishing a bound
\begin{equation*}
	a(t,s,\xi)\in L^\infty\big([0,T)\times[0,T)\times\mathbb{R}^d\big).
\end{equation*}
More generally, when $p\neq 2$, in order to ensure that \eqref{multiplier:1} or \eqref{multiplier:2} hold, it is sufficient to require that
\begin{equation*}
	\mathcal{F}^{-1}\big[a(t,s,\xi)\varphi(2^{-k}\xi)\big]\in L^1(\mathbb{R}^d),
\end{equation*}
uniformly in $t$, $s$ and $k$, where $\varphi(2^{-k}\xi)$ is a smooth compactly supported cutoff function used to define a Littlewood--Paley dyadic decomposition (see Appendix \ref{besov:1}). Therefore, in view of the straightforward classical estimate
\begin{equation*}
	\begin{aligned}
		\left\|\mathcal{F}^{-1}\big[a(t,s,\xi)\varphi(2^{-k}\xi)\big](x)\right\|_{L^1_x}
		& \lesssim
		2^{-k\frac d2}\left\|\left(1+2^k|x|\right)^N\mathcal{F}^{-1}\big[a(t,s,\xi)\varphi(2^{-k}\xi)\big](x)\right\|_{L^2_x}
		\\
		& \lesssim 2^{-k\frac d2}\sum_{\substack{\alpha\in\mathbb{N}^d \\ |\alpha|\leq N}}
		\left\|2^{k|\alpha|}\partial_\xi^\alpha\big[a(t,s,\xi)\varphi(2^{-k}\xi)\big]\right\|_{L^2_\xi}
		\\
		& \lesssim 2^{-k\frac d2}\sum_{\substack{\alpha,\beta\in\mathbb{N}^d \\ |\alpha|+|\beta|\leq N}}
		\left\|2^{k|\alpha|}\partial_\xi^\alpha a(t,s,\xi)(\partial^\beta \varphi)(2^{-k}\xi)\right\|_{L^2_\xi},
	\end{aligned}
\end{equation*}
where we have used Plancherel's theorem and $N$ is any integer larger than $\frac d2$, we see that \eqref{multiplier:1} and \eqref{multiplier:2} both hold as soon as $a(t,s,\xi)$ is sufficiently differentiable in $\xi$ (except possibly at the origin $\xi=0$) and satisfies the estimate
\begin{equation}\label{multiplier:3}
	\left\||\xi|^{|\alpha|}\partial_\xi^\alpha a(t,s,\xi)\right\|_{L^\infty_{t,s,\xi}}
	<\infty,
\end{equation}
for every multi-index $\alpha\in\mathbb{N}^d$ with $|\alpha|\leq \left[\frac d2\right]+1$.
Observe that the above criterion establishes the boundedness of $a(t,s,D)$ over $\dot B^0_{p,\infty}(\mathbb{R}^d)$, uniformly in $t$ and $s$, for all values of $1\leq p\leq \infty$, including the endpoints.
Later on, we will be making use of \eqref{multiplier:3} to show the boundedness of multipliers.

We return now to the smoothing estimates for the heat equation with a nontrivial source term $f$. The next result provides a large array of such estimates in the classical case $\alpha=0$.

\begin{prop}\label{heat:4}
	Let $\sigma\in\mathbb{R}$, $1<r<m<\infty$ and $p\in [1,\infty]$. If $f$ belongs to $L^r\big([0,T);\dot B_{p,\infty}^{\sigma+\frac 2r}\big)$ and $w_0=0$, then the solution of the heat equation \eqref{heat:1}, with $\alpha=0$, satisfies
	\begin{equation*}
		\left\|\int_0^te^{(t-s)\Delta}a(t,s,D)f(s)ds \right\|_{ L^m \big([0,T),\dot B^{\sigma+2+\frac 2m}_{p,1} \big)}
		\lesssim
		\norm{a}_{M_p} \|f \|_{ L^r\big([0,T),\dot B^{\sigma+\frac 2r}_{p,\infty}\big)},
	\end{equation*}
	for any Fourier multiplier $a(t,s,D)$.
\end{prop}

\begin{rem}
	We refer to \cite[Lemma 2]{a19} for a complete justification of the preceding proposition in the case $a(t,s,D)=\mathrm{Id}$.
	A straightforward adaptation of this proof readily extends the result to nontrivial multipliers $a(t,s,D)$.
\end{rem}

\begin{rem}
	The endpoint case $r=m$ above corresponds formally to a maximal gain of two derivatives on the solution of the heat equation. However, the method of proof of this result relies on the Hardy--Littlewood--Sobolev inequality, which typically falls short for endpoint settings. It is therefore not possible to extend the proof of \cite[Lemma 2]{a19} to the case $r=m$.
\end{rem}

The next result generalizes Proposition \ref{heat:4} to incorporate the action of a damping term.

\begin{prop}\label{damped_heat:1}
	Let $\sigma\in\mathbb{R}$, $p\in [1,\infty]$ and
	\begin{equation*}
		1\leq r\leq m\leq\infty,
		\quad
		0<\theta< 1+\frac 1m-\frac 1r\leq 1,
	\end{equation*}
	or
	\begin{equation*}
		1<r<m<\infty,
		\quad
		0<\theta= 1+\frac 1m-\frac 1r<1.
	\end{equation*}
	Then, for any $\alpha\geq 0$, one has the estimate
	\begin{equation}\label{damped_heat:3}
		\begin{aligned}
			\left\|\int_0^te^{-(t-s)(\alpha-\Delta)}a(t,s,D)f(s)ds \right\|_{ L^m \big([0,T),\dot B^{\sigma+2\theta}_{p,1} \big)}
			\hspace{-20mm}&
			\\
			&\lesssim
			\left(\frac{T}{1+\alpha T}\right)^{1+\frac 1m-\frac 1r-\theta}
			\norm{a}_{M_p}\|f \|_{ L^r\big([0,T),\dot B^{\sigma}_{p,\infty}\big)},
		\end{aligned}
	\end{equation}
	for any $f$ in $L^r\big([0,T);\dot B_{p,\infty}^{\sigma}\big)$ and any Fourier multiplier $a(t,s,D)$.
\end{prop}

\begin{rem}
	We emphasize that any implicit constant involved in the estimate of Proposition \ref{damped_heat:1} is independent of $T$ and $\alpha$. Moreover, it is permitted to set $T=\infty$ and $\alpha>0$ therein, in order to deduce a global estimate.
\end{rem}

\begin{rem}
	Observe that, choosing any $1\leq r\leq m\leq\infty$, $1\leq p,q\leq\infty$ and $\sigma\in\mathbb{R}$, one has the simple estimate
	\begin{equation*}
		\begin{aligned}
			\left\|\int_0^te^{-(t-s)(\alpha-\Delta)}a(t,s,D)f(s)ds \right\|_{ L^m \big([0,T),\dot B^{\sigma}_{p,q} \big)}
			&\lesssim \norm{a}_{M_p}
			\left\|\int_0^te^{-\alpha(t-s)}\norm{f(s)}_{\dot B^{\sigma}_{p,q}}ds \right\|_{ L^m ([0,T))}
			\\
			&\lesssim
			\left(\frac{T}{1+\alpha T}\right)^{1+\frac 1m-\frac 1r}
			\norm{a}_{M_p}\|f \|_{ L^r\big([0,T),\dot B^{\sigma}_{p,q}\big)},
		\end{aligned}
	\end{equation*}
	for all $\alpha\geq 0$,
	which corresponds to the case $\theta=0$ in the previous proposition.
\end{rem}

\begin{proof}[Proof in the case $1<r\leq m<\infty$.]
	First of all, notice that the case
	\begin{equation*}
		1<r<m<\infty,
		\quad
		0<\theta= 1+\frac 1m-\frac 1r<1
	\end{equation*}
	follows from a direct application of Proposition \ref{heat:4}, by absorbing the damping term $e^{-\alpha (t-s)}$ into the multiplier $a(t,s,D)$.
	
	In order to treat the remaining case
	\begin{equation*}
		1< r\leq m<\infty,
		\quad
		0<\theta< 1+\frac 1m-\frac 1r\leq 1,
	\end{equation*}
	we introduce auxiliary parameters
	\begin{equation*}
		1<r_0< r\leq m< m_0<\infty
	\end{equation*}
	such that
	\begin{equation*}
		\theta= 1+\frac 1{m_0}-\frac 1{r_0}.
	\end{equation*}
	In particular, in view of Proposition \ref{heat:4}, we have that
	\begin{equation*}
		\left\|\int_0^te^{(t-s)\Delta}a(t,s,D)f(s)ds \right\|_{ L^{m_0} \big([0,T),\dot B^{\sigma+2\theta}_{p,1} \big)}
		\lesssim
		\norm{a}_{M_p} \|f \|_{ L^{r_0}\big([0,T),\dot B^{\sigma}_{p,\infty}\big)}.
	\end{equation*}
	Then, an application of the Damping Lemma (Lemma \ref{damping:1}) implies, for any $\alpha\geq 0$, that \eqref{damped_heat:3} holds true, thereby concluding the proof.
\end{proof}

For the sake of completeness, since the preceding proof fails to treat the cases $r=1$ and $m=\infty$, we provide now an alternative justification of Proposition \ref{damped_heat:1}, based on the proof of Lemma 2 from \cite{a19}, which works in full generality.

\begin{proof}[General proof.]
	Following \cite{a19}, we begin by using \eqref{heat:2} and \eqref{multiplier:2}
	to deduce the existence of an independent constant $C_*>0$ such that
	\begin{equation*}
		\left\|\Delta_k \int_0^te^{-(t-s)(\alpha-\Delta)}a(t,s,D)f(s)ds\right\|_{L^p} \lesssim \int_0^t e^{-(t-s)(\alpha+C_*2^{2k})}\|\Delta_k f(s)\|_{L^p} ds .
	\end{equation*}
	For simplicity, we omit the norm $\norm{a}_{M_p}$ which we absorb in the implicit constants.
	It then follows that
	\begin{equation*}
		\begin{aligned}
			\left\|\int_0^te^{-(t-s)(\alpha-\Delta)}a(t,s,D)f(s)ds\right\|_{\dot B^{\sigma+2\theta}_{p,1} }
			& \lesssim
			\int_0^t \sum_{k\in\mathbb{Z}}e^{-(t-s)(\alpha+C_*2^{2k})}2^{k(\sigma+2\theta) }\|\Delta_k f(s)\|_{L^p} ds
			\\
			& \lesssim
			\int_0^T |t-s|^{-\theta} e^{-\alpha(t-s)}
			\|f(s)\|_{\dot B^{\sigma}_{p,\infty}} ds,
		\end{aligned}
	\end{equation*}
	where we have used \eqref{kernel:1}, with the assumption that $\theta>0$, to deduce that
	\begin{equation*}
		\sum_{k\in\mathbb{Z}}2^{2k\theta} e^{-C_*(t-s)2^{2k}}\lesssim |t-s|^{-\theta}.
	\end{equation*}
	
	Next, if $\theta=1+\frac 1m-\frac 1r$, by virtue of the Hardy--Littlewood--Sobolev inequality, which holds because $0<\theta<1$ and $1<m,r<\infty$, we infer that
	\begin{equation*}
		\left\|\int_0^te^{-(t-s)(\alpha-\Delta)}a(t,s,D)f(s)ds\right\|_{L^m \dot B^{\sigma+2\theta}_{p,1} }
		\lesssim
		\left\|\int_0^T |t-s|^{-\theta}
		\|f(s)\|_{\dot B^{\sigma}_{p,\infty}} ds\right\|_{L^m}
		\lesssim \|f\|_{L^r\dot B^{\sigma}_{p,\infty}}.
	\end{equation*}
	Similarly, if $0<\theta<1+\frac 1m-\frac 1r\leq 1$, we deduce from Young's convolution inequality that
	\begin{equation*}
		\begin{aligned}
			\left\|\int_0^te^{-(t-s)(\alpha-\Delta)}a(t,s,D)f(s)ds\right\|_{L^m \dot B^{\sigma+2\theta}_{p,1} }
			& \lesssim
			\left\|\int_0^T |t-s|^{-\theta} e^{-\alpha(t-s)}
			\|f(s)\|_{\dot B^{\sigma}_{p,\infty}} ds\right\|_{L^m}
			\\
			& \lesssim
			\left(\int_0^T (t^{-\theta} e^{-\alpha t})^{(1+\frac 1m-\frac 1r)^{-1}} dt\right)^{1+\frac 1m-\frac 1r}
			\|f\|_{L^r\dot B^{\sigma}_{p,\infty}}
			\\
			& \lesssim
			\left(\frac{T}{1+\alpha T}\right)^{1+\frac 1m-\frac 1r-\theta}
			\|f\|_{L^r\dot B^{\sigma}_{p,\infty}},
		\end{aligned}
	\end{equation*}
	which concludes the proof of the proposition.
\end{proof}

The shortcomings of Propositions \ref{heat:4} and \ref{damped_heat:1} in the case $r=m$, with $\theta=1$, naturally bring the question of the maximal regularity of the Laplacian in Banach spaces.

For a given Banach space $X$ such that the Laplacian operator $\Delta$ is defined on a dense subspace of $X$, we say that the Laplacian (or another elliptic operator) has maximal $L^p$-regularity on $[0,T)$, for some $1<p<\infty$, if the solution \eqref{heat:3} of the heat equation (without damping, i.e., $\alpha=0$) for a null initial data, i.e., $w_0=0$, is differentiable almost everywhere in $t$, takes values almost everywhere in the domain of $\Delta$ and satisfies the estimate
\begin{equation*}
	\|\partial_tw\|_{L^p([0,T);X)}+\|\Delta w\|_{L^p([0,T);X)}
	\leq C_p \|f\|_{L^p([0,T);X)},
\end{equation*}
for any source term $f\in L^p([0,T);X)$. We refer to \cite{kw04} for an introduction to the theory of maximal $L^p$-regularity for parabolic equations.

The next important result, extracted from \cite{ag20}, establishes the maximal regularity of the Laplacian in all homogeneous Besov spaces $\dot B^{\sigma}_{p,q}$.
In particular, this result provides the basis which will allow us (in Section \ref{section:low:freq}, for instance) to obtain stronger estimates, with sharp gains of parabolic regularity, by avoiding the use of Chemin--Lerner spaces.

\begin{prop}[\cite{ag20}]\label{heat:5}
	Let $\sigma\in\mathbb{R}$, $p,q\in [1,\infty]$ and $r\in (1,\infty)$. If $f$ belongs to $L^r\big([0,T);\dot B_{p,q}^{\sigma}\big)$ and $w_0=0$, then the solution of the heat equation \eqref{heat:1}, with $\alpha\geq 0$, satisfies
	\begin{equation*}
		\left\|\int_0^te^{-(t-s)(\alpha-\Delta)}a(t,s,D)f(s)ds \right\|_{ L^r \big([0,T),\dot B^{\sigma+2}_{p,q} \big)}
		\lesssim
		\norm{a}_{M_p}\|f \|_{ L^r\big([0,T),\dot B^{\sigma}_{p,q}\big)},
	\end{equation*}
	for any Fourier multiplier $a(t,s,D)$.
	The result remains valid if $r=q=1$ or $r=q=\infty$.
\end{prop}

\begin{rem}
	Again, it is to be emphasized that any implicit constant involved in the above estimate is independent of $T$ and $\alpha$. In particular, one can set $T=\infty$ therein.
\end{rem}

\begin{rem}
	We refer to \cite[Proposition 3.1]{ag20} or \cite[Lemma 3]{a19} for a proof of Proposition \ref{heat:5} in the case $a(t,s,D)=\mathrm{Id}$ and $\alpha=0$. The original proof from \cite{ag20} deals first with the case $q=1$ and then relies on an interpolation argument. The proof from \cite{a19}, however, offers a self-contained approach which avoids interpolation altogether.
\end{rem}

\begin{rem}
	In fact, the original statements of Proposition 3.1 in \cite{ag20} and Lemma 3 in \cite{a19} only cover the range of parameters $1\leq q\leq r<\infty$. Nevertheless, it is readily seen that the corresponding proofs can be used \emph{mutatis mutandis} to show identical bounds on the adjoint operator, which is defined by
	\begin{equation*}
		\int_t^Te^{-(s-t)(\alpha-\Delta)}\overline{a(s,t,D)}g(s)ds.
	\end{equation*}
	It then follows from a standard duality argument that these results hold for values $1< r\leq q\leq \infty$, as well.
\end{rem}

For the sake of completeness and clarity, we provide now a full justification of Proposition \ref{heat:5}, based on a straightforward adaptation of the proof of \cite[Lemma 3]{a19}.

\begin{proof}
	First, noticing that the damping term $e^{-\alpha(t-s)}$ can be absorbed into the bounded multiplier $a(t,s,D)$, we assume, without loss of generality, that $\alpha=0$. Then, we follow the proof of Lemma 3 from \cite{a19}.
	
	We start by considering that $1\leq q\leq r<\infty$.
	By duality, it is enough to prove that, if $g$ is a nonnegative function in $L^{b'}([0,T))$ with $b=\frac rq\geq 1$ and $\frac 1b+\frac 1{b'} = 1$, then
	\begin{equation*}
		\int_0^T g(t) \left\|\int_0^te^{(t-s)\Delta}a(t,s,D)f(s)ds\right\|_{\dot B^{\sigma+2}_{p,q}}^q dt
		\lesssim \norm{a}_{M_p}^q
		\|f \|_{L^r([0,T),\dot B^{\sigma}_{p,q})}^q
		\|g\|_{L^{b'}([0,T))}.
	\end{equation*}
	To this end, using \eqref{heat:2} and \eqref{multiplier:2}, we deduce the existence of a constant $C_*>0$ such that
	\begin{equation*}
		\begin{aligned}
			\int_0^T g(t) \left\|\int_0^te^{(t-s)\Delta}a(t,s,D)f(s)ds\right\|_{\dot B^{\sigma+2}_{p,q}}^q dt
			\hspace{-40mm}&
			\\
			& = \sum_{k \in {\mathbb Z}}
			\int_0^T g(t) \left\|\int_0^te^{(t-s)\Delta}a(t,s,D)\Delta_kf(s)ds\right\|_{L^p}^q2^{k(\sigma+2)q} dt
			\\
			& \lesssim \norm{a}_{M_p}^q \sum_{k \in {\mathbb Z}}
			\int_0^T g(t) \left(\int_0^te^{-C_*(t-s)2^{2k}}\|\Delta_kf(s)\|_{L^p}ds\right)^q2^{k(\sigma+2)q} dt
			\\
			& \lesssim \norm{a}_{M_p}^q \sum_{k \in {\mathbb Z}}
			\int_0^T g(t) \int_0^te^{-C_*(t-s)2^{2k}}\|\Delta_kf(s)\|_{L^p}^qds\ 2^{k(\sigma q+2)} dt.
		\end{aligned}
	\end{equation*}
	For simplicity, we omit the norm $\norm{a}_{M_p}$ in the remaining estimates.
	
	Next, we define a maximal operator by
		\begin{equation*}
			Mg(s)=\sup_{\rho>0} \rho\int_0^T e^{-\rho|s-t|}|g(t)| dt.
		\end{equation*}
	Classical results from harmonic analysis (see \cite[Theorems 2.1.6 and 2.1.10]{g14}) establish that $M$ is bounded over $L^c\left([0,T)\right)$, for any $1<c\leq\infty$. One can then write that
	\begin{equation*}
		\begin{aligned}
			\int_0^T g(t) \left\|\int_0^te^{(t-s)\Delta}a(t,s,D)f(s)ds\right\|_{\dot B^{\sigma+2}_{p,q}}^q dt
			\hspace{-40mm}&
			\\
			& \lesssim \sum_{k \in {\mathbb Z}}
			\int_0^T \left[2^{2k}\int_s^T g(t) e^{-C_*(t-s)2^{2k}}dt\right] \|\Delta_kf(s)\|_{L^p}^q2^{k\sigma q}ds
			\\
			&\lesssim \sum_{k \in {\mathbb Z}}
			\int_0^T Mg(s) \|\Delta_kf(s)\|_{L^p}^q 2^{k\sigma q}ds
			=\int_0^T Mg(s) \|f(s)\|_{\dot B^\sigma_{p,q}}^qds.
		\end{aligned}
	\end{equation*}
	Therefore, by the boundedness properties of $Mg$ and H\"older's inequality, we conclude that
	\begin{equation*}
		\begin{aligned}
			\int_0^T g(t) \left\|\int_0^te^{(t-s)\Delta}a(t,s,D)f(s)ds\right\|_{\dot B^{\sigma+2}_{p,q}}^q dt
			& \lesssim \|Mg\|_{L^{b'}}
			\|f\|_{L^r\dot B^{\sigma}_{p,q}}^q \\
			& \lesssim \|g\|_{L^{b'}}
			\|f\|_{L^r\dot B^{\sigma}_{p,q}}^q,
		\end{aligned}
	\end{equation*}
	which completes the proof of the proposition in the case $1\leq q\leq r<\infty$.
	
	Now, observe that the exact same proof applies to the adjoint operator
	\begin{equation*}
		\int_t^Te^{(s-t)\Delta}\overline{a(s,t,D)}f(s)ds,
	\end{equation*}
	thereby leading to the estimate
	\begin{equation*}
		\left\|\int_t^Te^{(s-t)\Delta}\overline{a(s,t,D)}f(s)ds \right\|_{ L^r \big([0,T),\dot B^{\sigma+2}_{p,q} \big)}
		\lesssim
		\|f \|_{ L^r\big([0,T),\dot B^{\sigma}_{p,q}\big)},
	\end{equation*}
	whenever $1\leq q\leq r<\infty$. Then, a standard duality argument establishes that
	\begin{equation*}
		\left\|\int_0^te^{(t-s)\Delta}a(t,s,D)f(s)ds \right\|_{ L^{r'} \big([0,T),\dot B^{-\sigma}_{p',q'} \big)}
		\lesssim
		\|f \|_{ L^{r'}\big([0,T),\dot B^{-(\sigma+2)}_{p',q'}\big)},
	\end{equation*}
	for parameter values in the range $1<r'\leq q'\leq \infty$. Therefore, replacing $p$, $q$, $r$ and $\sigma$ by $p'$, $q'$, $r'$ and $-(\sigma+2)$, respectively, shows the proposition in the case $1<r\leq q\leq \infty$, which concludes the proof.
\end{proof}

\subsection{Damped Strichartz estimates}\label{section:damped:strichartz:1}

We focus now on the interaction between damping and dispersion. More precisely, we are going to explore how the Damping Lemma (Lemma \ref{damping:1}) applies to Strichartz estimates. To that end, we first recall the general result on Strichartz estimates for abstract semigroups from \cite{kt98}. We also refer to \cite[Chapter 8]{bcd11} for a comprehensive exposition of Strichartz estimates.

\begin{prop}[\cite{kt98}]\label{strichartz:1}
	Let $H$ be a Hilbert space and $(X,dx)$ be a measure space. For each $t\in [0,T)$, with $T>0$, let $U(t):H\to L^2(X)$ be an operator such that
	\begin{equation*}
		U(t)\in L^\infty\big([0,T);\mathcal{L}(H,L^2(X))\big)
	\end{equation*}
	and, for some  $\sigma>0$,
	\begin{equation*}
		\norm{U(t)U(s)^*g}_{L^\infty(X)}
		\lesssim \frac 1{|t-s|^\sigma}\norm{g}_{L^1(X)},
	\end{equation*}
	for all $t,s\in[0,T)$, with $t\neq s$, and all $g\in L^1(X)\cap L^2(X)$.
	
	Then, the estimate
	\begin{equation*}
		\norm{U(t)f}_{L^q_tL^r_x}\lesssim\norm{f}_H
	\end{equation*}
	and its dual version
	\begin{equation*}
		\norm{\int_0^TU(t)^*g(t)dt}_H\lesssim\norm{g}_{L^{q'}_tL^{r'}_x}
	\end{equation*}
	hold for any exponent pair $(q,r)\in [2,\infty]^2$ which is admissible in the sense that
	\begin{equation*}
		\frac 1q+\frac\sigma r=\frac\sigma 2
		\quad\text{and}\quad
		(q,r,\sigma)\neq (2,\infty,1).
	\end{equation*}
	Furthermore, if $(\tilde q,\tilde r)\in [2,\infty]^2$ is also an admissible exponent pair, then the estimate
	\begin{equation*}
		\norm{\int_0^T\chi(t,s)U(t)U(s)^* g(s)ds}_{L^q_tL^r_x}\lesssim\norm{\chi}_{L^\infty}\norm{g}_{L^{\tilde q'}_tL^{\tilde r'}_x}
	\end{equation*}
	holds for any $\chi(t,s)\in L^\infty([0,T)^2;\mathbb{R})$.
\end{prop}

\begin{rem}
	In fact, the statement of the result from \cite{kt98} only considers the function $\chi(t,s)=\mathds{1}_{\{s\leq t\}}$. However, a straightforward alteration of the proof from \cite{kt98} easily shows that the result actually holds for all $\chi(t,s)\in L^\infty([0,T)^2;\mathbb{R})$. A detailed proof valid for all $\chi(t,s)$ can also be found in Section 8.2 of \cite{bcd11}.
\end{rem}

By combining the Damping Lemma with the preceding proposition, we obtain the damped Strichartz estimates, which are stated in precise terms in the next result.

\begin{prop}[Damped Strichartz Estimates]\label{strichartz:2}
	Let $H$ be a Hilbert space and $(X,dx)$ be a measure space. For each $t\in [0,T)$, with $T>0$, let $U(t):H\to L^2(X)$ be an operator such that
	\begin{equation*}
		U(t)\in L^\infty\big([0,T);\mathcal{L}(H,L^2(X))\big)
	\end{equation*}
	and, for some  $\sigma>0$,
	\begin{equation*}
		\norm{U(t)U(s)^*g}_{L^\infty(X)}
		\lesssim \frac 1{|t-s|^\sigma}\norm{g}_{L^1(X)},
	\end{equation*}
	for all $t,s\in[0,T)$, with $t\neq s$, and all $g\in L^1(X)\cap L^2(X)$.
	
	Then, for any $\alpha\geq 0$, the estimate
	\begin{equation*}
		\norm{e^{-\alpha t}U(t)f}_{L^q_tL^r_x}\lesssim \left(\frac T{1+\alpha T}\right)^{\frac 1q+\frac\sigma r-\frac \sigma 2} \norm{f}_H
	\end{equation*}
	and its dual version
	\begin{equation*}
		\norm{\int_0^Te^{-\alpha t}U(t)^*g(t)dt}_H\lesssim \left(\frac T{1+\alpha T}\right)^{\frac 1q+\frac\sigma r-\frac \sigma 2} \norm{g}_{L^{q'}_tL^{r'}_x}
	\end{equation*}
	hold for any exponent pair $(q,r)\in [1,\infty]\times[2,\infty]$ which is admissible in the sense that
	\begin{equation*}
		\frac 1q+\frac\sigma r\geq \frac\sigma 2,
		\quad
		\frac 12+\frac\sigma r\geq \frac\sigma 2
		\quad\text{and}\quad
		(r,\sigma)\neq (\infty,1).
	\end{equation*}
	Furthermore, if $(\tilde q,\tilde r)\in [1,\infty]\times[2,\infty]$ is also an admissible exponent pair such that
	\begin{equation*}
		\frac 1q+\frac 1{\tilde q}\leq 1,
	\end{equation*}
	then the estimate
	\begin{equation*}
		\norm{\int_0^Te^{-\alpha |t-s|}\chi(t,s)U(t)U(s)^* g(s)ds}_{L^{q}_tL^r_x}\lesssim
		\left(\frac T{1+\alpha T}\right)^{\frac 1q+\frac 1{\tilde q}+\sigma\left(\frac 1r+\frac 1{\tilde r}\right)-\sigma}
		\norm{\chi}_{L^\infty}\norm{g}_{L^{\tilde q'}_tL^{\tilde r'}_x},
	\end{equation*}
	holds for any $\chi(t,s)\in L^\infty([0,T)^2;\mathbb{R})$.
\end{prop}

\begin{proof}
	First of all, observe that all hypotheses of Proposition \ref{strichartz:1} are satisfied by $U(t)$. Then, introducing the parameter $q_0\in [2,\infty]$ by setting
	\begin{equation}\label{admissible:1}
		\frac 1{q_0}=\sigma\left(\frac 1 2-\frac 1r\right),
	\end{equation}
	we see that the exponent pair $(q_0,r)\in [2,\infty]^2$ is admissible for Proposition \ref{strichartz:1}. Therefore, it follows from Proposition \ref{strichartz:1}, with an application of H\"older's inequality, that
	\begin{equation*}
		\norm{e^{-\alpha t}U(t)f}_{L^q_tL^r_x}\leq \norm{e^{-\alpha t}}_{L^{(\frac 1q-\frac 1{q_0})^{-1}}([0,T))}\norm{U(t)f}_{L^{q_0}_tL^r_x}
		\lesssim \left(\frac T{1+\alpha T}\right)^{\frac 1q-\frac 1{q_0}} \norm{f}_H,
	\end{equation*}
	which establishes the first estimate of the proposition. The second estimate then ensues from a dual reformulation of the first estimate.
	
	There only remains to justify the validity of the third estimate. To that end, employing \eqref{admissible:1}, we introduce auxiliary parameters $q_0,\tilde q_0\in [2,\infty]$ so that the exponent pairs $(q_0,r),(\tilde q_0,\tilde r)\in [2,\infty]^2$ are admissible for Proposition \ref{strichartz:1}. In particular, it follows that
	\begin{equation*}
		\norm{\int_0^T\chi(t,s)U(t)U(s)^* g(s)ds}_{L^{q_0}_tL^r_x}\lesssim\norm{\chi}_{L^\infty}\norm{g}_{L^{\tilde q_0'}_tL^{\tilde r'}_x},
	\end{equation*}
	for any $\chi(t,s)\in L^\infty([0,T)^2;\mathbb{R})$. Therefore, noticing that $\tilde q_0'\leq \tilde q'\leq q\leq q_0$, we conclude from an application of Lemma \ref{damping:1} that
	\begin{equation*}
		\norm{\int_0^Te^{-\alpha |t-s|}\chi(t,s)U(t)U(s)^* g(s)ds}_{L^{q}_tL^r_x}\lesssim \left(\frac T{1+\alpha T}\right)^{\frac 1q-\frac 1{q_0}+\frac 1{\tilde q_0'}-\frac 1{\tilde q'}}
		\norm{\chi}_{L^\infty}\norm{g}_{L^{\tilde q'}_tL^{\tilde r'}_x},
	\end{equation*}
	which completes the proof.
\end{proof}

\begin{rem}
	We do not make any claim of optimality of Proposition \ref{strichartz:2}. It would be interesting, though, to test the sharpness of the admissibility criteria for the exponent pairs $(q,r)$ and $(\tilde q,\tilde r)$ in connection with the sensitivity in $T$ and $\alpha$ of the estimates.
\end{rem}

We proceed now to specific formulations of the damped Strichartz estimates for the Schr\"odinger and wave equations, as well as for Maxwell's system.

\begin{cor}[Damped Schr\"odinger equation]
	Let $d\geq 1$ and consider a solution $u(t,x)$ of the damped Schr\"odinger equation
	\begin{equation*}
		\left\{
			\begin{aligned}
				(\partial_t +\alpha - i\Delta) u(t,x)&= F(t,x)
				\\
				u(0,x)&=f(x)
			\end{aligned}
		\right.
	\end{equation*}
	with $\alpha\geq 0$, $t\in [0,T)$ and $x\in\mathbb{R}^d$.
	
	For any exponent pairs $(q,r),(\tilde q,\tilde r)\in [1,\infty]\times[2,\infty]$ which are admissible in the sense that
	\begin{equation*}
		\frac 2q+\frac d r\geq \frac d 2,
		\quad
		1+\frac d r\geq \frac d 2
		\quad\text{and}\quad
		(r,d)\neq (\infty,2)
	\end{equation*}
	and similarly for $(\tilde q,\tilde r)$, and such that
	\begin{equation*}
		\frac 1q+\frac 1{\tilde q}\leq 1,
	\end{equation*}
	one has the estimate
	\begin{equation*}
		\norm{u}_{L^q_tL^r_x}
		\lesssim \left(\frac T{1+\alpha T}\right)^{\frac 1q+\frac d2\left(\frac 1r-\frac 1 2\right)}\norm{f}_{L^2_x}
		+\left(\frac T{1+\alpha T}\right)^{\frac 1q+\frac 1{\tilde q}+\frac d2\left(\frac 1r+\frac 1{\tilde r}-1\right)}
		\norm{F}_{L^{\tilde q'}_tL^{\tilde r'}_x}.
	\end{equation*}
\end{cor}

\begin{proof}
	The solution $u(t,x)$ can be expressed by Duhamel's representation formula as
	\begin{equation*}
		\begin{aligned}
			u(t)&=e^{-\alpha t}U(t)f+\int_0^te^{-\alpha(t-s)}U(t)U(s)^*F(s)ds
			\\
			&=e^{-\alpha t}U(t)f+\int_0^te^{-\alpha(t-s)}U(t-s)F(s)ds,
		\end{aligned}
	\end{equation*}
	where
	\begin{equation*}
		U(t)=e^{it\Delta}
		\quad\text{and}\quad
		U(t)^*=e^{-it\Delta}.
	\end{equation*}
	In particular, one has the explicit formula (see Section 8.1.2 in \cite{bcd11})
	\begin{equation*}
		U(t)f(x)=\frac 1{(4\pi it)^{\frac d2}}\int_{\mathbb{R}^d}e^{-\frac{|x-y|^2}{4it}}f(y)dy,
	\end{equation*}
	which readily implies that $U(t)$ satisfies all hypotheses of Proposition \ref{strichartz:2} with $H=L^2(\mathbb{R}^d)$, $X=\mathbb{R}^d$ and $\sigma=\frac d2$.
	Therefore, we conclude that the corollary follows from a direct application of the damped Strichartz estimates of Proposition \ref{strichartz:2}.
\end{proof}

\begin{cor}[Damped half wave equation]\label{cor:halfwave}
	Let $d\geq 2$ and consider a solution $u(t,x)$ of the damped half wave equation
	\begin{equation*}
		\left\{
			\begin{aligned}
				(\partial_t +\alpha \mp i|D|) u(t,x)&= F(t,x)
				\\
				u(0,x)&=f(x)
			\end{aligned}
		\right.
	\end{equation*}
	with $\alpha\geq 0$, $t\in [0,T)$ and $x\in\mathbb{R}^d$.
	
	For any exponent pairs $(q,r),(\tilde q,\tilde r)\in [1,\infty]\times[2,\infty]$ which are admissible in the sense that
	\begin{equation*}
		\frac 2q+\frac {d-1} r\geq \frac {d-1} 2,
		\quad
		1+\frac {d-1} r\geq \frac {d-1} 2
		\quad\text{and}\quad
		(r,d)\neq (\infty,3)
	\end{equation*}
	and similarly for $(\tilde q,\tilde r)$, and such that
	\begin{equation*}
		\frac 1q+\frac 1{\tilde q}\leq 1,
	\end{equation*}
	one has the estimate
	\begin{equation*}
		\begin{aligned}
			2^{-j\frac {d+1}2\left(\frac 12-\frac 1r\right)}\norm{\Delta_j u}_{L^q_tL^r_x}
			&\lesssim
			\left(\frac {T}{1+\alpha T}\right)^{\frac 1q+\frac {d-1}2\left(\frac 1r-\frac 1 2\right)}\norm{\Delta_j f}_{L^2_x}
			\\
			&\quad+\left(\frac {T}{1+\alpha T}\right)^{\frac 1q+\frac 1{\tilde q}+\frac {d-1}2\left(\frac 1r+\frac 1{\tilde r}-1\right)}
			2^{j\frac{d+1}2\left(\frac 12-\frac 1{\tilde r}\right)}\norm{\Delta_j F}_{L^{\tilde q'}_tL^{\tilde r'}_x},
		\end{aligned}
	\end{equation*}
	for all $j\in\mathbb{Z}$.
\end{cor}

\begin{rem}
	If $\tilde q'\leq p\leq q$, then, further multiplying the preceding estimate by $2^{j\sigma}$, for some $\sigma\in\mathbb{R}$, and then summing over $j\in\mathbb{Z}$ in the $\ell^p$-norm leads to
	\begin{equation*}
		\begin{aligned}
			\norm{u}_{L^q_t\dot B^{\sigma -\frac {d+1}2\left(\frac 12-\frac 1r\right)}_{r,p,x}}
			&\lesssim
			\left(\frac {T}{1+\alpha T}\right)^{\frac 1q+\frac {d-1}2\left(\frac 1r-\frac 1 2\right)}\norm{f}_{\dot B^{\sigma}_{2,p,x}}
			\\
			&\quad+\left(\frac {T}{1+\alpha T}\right)^{\frac 1q+\frac 1{\tilde q}+\frac {d-1}2\left(\frac 1r+\frac 1{\tilde r}-1\right)}
			\norm{F}_{L^{\tilde q'}_t\dot B^{\sigma +\frac {d+1}2\left(\frac 12-\frac 1{\tilde r}\right)}_{\tilde r',p,x}}.
		\end{aligned}
	\end{equation*}
\end{rem}

\begin{proof}
	The solution $u(t,x)$ can be expressed by Duhamel's representation formula as
	\begin{equation*}
		u(t)=e^{-\alpha t}e^{\pm it|D|}f+\int_0^te^{-\alpha(t-s)}e^{\pm i(t-s)|D|}F(s)ds.
	\end{equation*}
	For each $j\in\mathbb{Z}$, we introduce now the flow
	\begin{equation*}
		U_j(t)f(x)\bydef e^{\pm it|D|}\psi\left(\frac{D}{2^j}\right)f(x),
	\end{equation*}
	where $\psi(\xi)$ is a smooth compactly supported function such that $0\notin\operatorname{supp}\psi$ and $\psi\equiv 1$ on $\{\frac 12\leq|\xi|\leq 2\}$. In particular, if $\Delta_j$ is the Littlewood--Paley frequency cutoff operator, defined in Appendix \ref{besov:1}, which localizes frequencies to $\{2^{j-1}\leq|\xi|\leq 2^{j+1}\}$, one has the representation
	\begin{equation*}
		\begin{aligned}
			\Delta_j u(t)&=e^{-\alpha t}U_j(t)\Delta_jf+\int_0^te^{-\alpha(t-s)}U_j(t-s)\Delta_jF(s)ds
			\\
			&=e^{-\alpha t}U_j(t)\Delta_jf+\int_0^te^{-\alpha(t-s)}U_{j}(t)U_j(s)^*\Delta_jF(s)ds.
		\end{aligned}
	\end{equation*}
	We are now going to apply Proposition \ref{strichartz:2} to the operator $U_0(t)$ in order to control $\Delta_0 u(t)$.
	
	Classical results, based on the Stationary Phase Method, establish that
	\begin{equation*}
		\norm{U_{0}(t)U_{0}(s)^*f}_{L^\infty(\mathbb{R}^d)}
		\lesssim \frac 1{|t-s|^\frac{d-1}{2}}\norm{f}_{L^1(\mathbb{R}^d)},
	\end{equation*}
	for all $t\neq s$ and $f\in L^1(\mathbb{R}^d)$. (Proposition 8.15 from \cite{bcd11} contains a precise justification of the preceding dispersive estimate and we further refer to Section 8.1.3 from \cite{bcd11} for more details on the Stationary Phase Method.)
	It therefore follows that $U_0(t)$ satisfies all hypotheses of Proposition \ref{strichartz:2} with $H=L^2(\mathbb{R}^d)$, $X=\mathbb{R}^d$ and $\sigma=\frac {d-1}2$.
	Hence, we conclude that
	\begin{equation}\label{strichartz:3}
		\norm{\Delta_0 u}_{L^q_tL^r_x}
		\lesssim \left(\frac T{1+\alpha T}\right)^{\frac 1q+\frac {d-1}2\left(\frac 1r-\frac 1 2\right)}\norm{\Delta_0 f}_{L^2_x}
		+\left(\frac T{1+\alpha T}\right)^{\frac 1q+\frac 1{\tilde q}+\frac {d-1}2\left(\frac 1r+\frac 1{\tilde r}-1\right)}
		\norm{\Delta_0 F}_{L^{\tilde q'}_tL^{\tilde r'}_x},
	\end{equation}
	for all admissible exponent pairs.
	
	In order to recover an estimate for all components $\Delta_j u$, where $j\in\mathbb{Z}$, we conduct a simple scaling argument by introducing
	\begin{equation*}
		u_j(t,x)\bydef u\left(\frac t{2^j},\frac x{2^j}\right),
		\quad
		F_j(t,x)\bydef \frac 1{2^j}F\left(\frac t{2^j},\frac x{2^j}\right),
		\quad
		f_j(x)\bydef f \left(\frac x{2^j}\right).
	\end{equation*}
	Noticing that
	\begin{equation*}
		\Delta_0 u_j(t,x)=(\Delta_j u)\left(\frac t{2^j},\frac x{2^j}\right)
		\quad
		\Delta_0 F_j(t,x)= \frac 1{2^j}(\Delta_j F)\left(\frac t{2^j},\frac x{2^j}\right),
		\quad
		\Delta_0 f_j(x)= (\Delta_j f) \left(\frac x{2^j}\right)
	\end{equation*}
	and that $u_j(t,x)$ solves
	\begin{equation*}
		\left\{
			\begin{aligned}
				\left(\partial_t +2^{-j}\alpha \mp i|D|\right) u_j(t,x)&= F_j(t,x)
				\\
				u_j(0,x)&=f_j(x)
			\end{aligned}
		\right.
	\end{equation*}
	on $[0,2^j T)$, we obtain, applying \eqref{strichartz:3} to $u_j$, that
	\begin{equation*}
		\begin{aligned}
			2^{j\left(\frac 1q+\frac dr\right)}\norm{\Delta_j u}_{L^q_tL^r_x}
			&=\norm{\Delta_0 u_j}_{L^q_tL^r_x}
			\\
			&\lesssim \left(\frac {2^jT}{1+\alpha T}\right)^{\frac 1q+\frac {d-1}2\left(\frac 1r-\frac 1 2\right)}\norm{\Delta_0f_j}_{L^2_x}
			\\
			&\quad+\left(\frac {2^jT}{1+\alpha T}\right)^{\frac 1q+\frac 1{\tilde q}+\frac {d-1}2\left(\frac 1r+\frac 1{\tilde r}-1\right)}
			\norm{\Delta_0F_j}_{L^{\tilde q'}_tL^{\tilde r'}_x}
			\\
			&= \left(\frac {2^jT}{1+\alpha T}\right)^{\frac 1q+\frac {d-1}2\left(\frac 1r-\frac 1 2\right)}2^{j\frac d2}\norm{\Delta_j f}_{L^2_x}
			\\
			&\quad+\left(\frac {2^jT}{1+\alpha T}\right)^{\frac 1q+\frac 1{\tilde q}+\frac {d-1}2\left(\frac 1r+\frac 1{\tilde r}-1\right)}
			2^{j\left(-\frac 1{\tilde q}+d\left(1-\frac 1{\tilde r}\right)\right)}\norm{\Delta_j F}_{L^{\tilde q'}_tL^{\tilde r'}_x}.
		\end{aligned}
	\end{equation*}
	Finally, reorganizing the terms above, we deduce that
	\begin{equation*}
		\begin{aligned}
			\norm{\Delta_j u}_{L^q_tL^r_x}
			&\lesssim
			\left(\frac {T}{1+\alpha T}\right)^{\frac 1q+\frac {d-1}2\left(\frac 1r-\frac 1 2\right)}
			2^{j\frac {d+1}2\left(\frac 12-\frac 1r\right)}\norm{\Delta_jf}_{L^2_x}
			\\
			&\quad+\left(\frac {T}{1+\alpha T}\right)^{\frac 1q+\frac 1{\tilde q}+\frac {d-1}2\left(\frac 1r+\frac 1{\tilde r}-1\right)}
			2^{j\frac{d+1}2\left(1-\frac 1r-\frac 1{\tilde r}\right)}\norm{\Delta_jF}_{L^{\tilde q'}_tL^{\tilde r'}_x},
		\end{aligned}
	\end{equation*}
	which concludes the proof.
\end{proof}

\begin{cor}[Damped wave equation]\label{cor:wave}
	Let $d\geq 2$ and consider a solution $u(t,x)$ of the damped wave equation
	\begin{equation*}
		\left\{
			\begin{aligned}
				(\partial_t^2 +\alpha\partial_t -\Delta) u(t,x)&= F(t,x)
				\\
				u(0,x)&=f(x)
				\\
				\partial_t u(0,x)&=g(x)
			\end{aligned}
		\right.
	\end{equation*}
	with $\alpha\geq 0$, $t\in [0,T)$ and $x\in\mathbb{R}^d$.
	
	For any exponent pairs $(q,r),(\tilde q,\tilde r)\in [1,\infty]\times[2,\infty]$ which are admissible in the sense that
	\begin{equation*}
		\frac 2q+\frac {d-1} r\geq \frac {d-1} 2,
		\quad
		1+\frac {d-1} r\geq \frac {d-1} 2
		\quad\text{and}\quad
		(r,d)\neq (\infty,3)
	\end{equation*}
	and similarly for $(\tilde q,\tilde r)$, and such that
	\begin{equation*}
		\frac 1q+\frac 1{\tilde q}\leq 1,
	\end{equation*}
	one has the high-frequency estimate
	\begin{equation*}
		\begin{aligned}
			2^{-j\frac {d+1}2\left(\frac 12-\frac 1r\right)}\norm{\Delta_j (\partial_t u,\nabla u)}_{L^q_tL^r_x}
			&\lesssim
			\left(\frac{T}{1+\alpha T}\right)^{\frac 1q+\frac {d-1}2\left(\frac 1r-\frac 1 2\right)}
			\norm{\Delta_j (g,\nabla f)}_{L^2_x}
			\\
			&\quad+\left(\frac{T}{1+\alpha T}\right)^{\frac 1q+\frac 1{\tilde q}+\frac {d-1}2\left(\frac 1r+\frac 1{\tilde r}-1\right)}
			2^{j\frac{d+1}2\left(\frac 12-\frac 1{\tilde r}\right)}\norm{\Delta_j F}_{L_t^{\tilde q'}L^{\tilde r'}_x},
		\end{aligned}
	\end{equation*}
	for all $j\in\mathbb{Z}$ with $2^j\geq \alpha$, and the low-frequency estimates
	\begin{equation*}
		\begin{aligned}
			2^{-jd\left(\frac 12-\frac 1r\right)}\norm{\Delta_j \partial_t u}_{L^q_tL^r_x}
			&\lesssim
			\left(\frac{T}{1+\alpha T}\right)^{\frac 1q}\norm{\Delta_j g}_{L^2_x}
			+\frac 1\alpha\left(\frac{\alpha 2^{2j}T}{\alpha+ 2^{2j}T}\right)^\frac 1q2^{j\left(1-\frac 2q\right)}\norm{\Delta_j \nabla f}_{L^2_x}
			\\
			&\quad+\left(\frac{T}{1+\alpha T}\right)^{\frac 1q+\frac 1{\tilde q}}
			2^{jd\left(\frac 12-\frac 1{\tilde r}\right)}\norm{\Delta_j F}_{L_t^{\tilde q'}L^{\tilde r'}_x}
		\end{aligned}
	\end{equation*}
	and
	\begin{equation*}
		\begin{aligned}
			2^{-j\left(d(\frac 12-\frac 1r)-\frac 2q\right)}\norm{\Delta_j \nabla u}_{L^q_tL^r_x}
			&\lesssim
			\frac 1\alpha\left(\frac{\alpha 2^{2j}T}{\alpha+ 2^{2j}T}\right)^\frac 1q2^{j}
			\norm{\Delta_j g}_{L^2_x}
			+\left(\frac{\alpha 2^{2j}T}{\alpha+ 2^{2j}T}\right)^\frac 1q\norm{\Delta_j \nabla f}_{L^2_x}
			\\
			&\quad+\frac 1\alpha\left(\frac{\alpha 2^{2j}T}{\alpha+ 2^{2j}T}\right)^{\frac 1q+\frac 1{\tilde q}}
			2^{j\left(1+d(\frac 12-\frac 1{\tilde r})-\frac 2{\tilde q}\right)}\norm{\Delta_j F}_{L_t^{\tilde q'}L^{\tilde r'}_x},
		\end{aligned}
	\end{equation*}
	for all $j\in\mathbb{Z}$ with $2^j\leq \alpha$.
\end{cor}

\begin{rem}
	Summing the preceding inequalities in $j$ easily leads to damped Strichartz estimates in Besov spaces with summability in $\ell^p$, for any $\tilde q'\leq p\leq q$. However, due to the dichotomy of the statement of Corollary \ref{cor:wave} into high and low frequencies, the resulting estimates cannot be stated with homogeneous Besov spaces in a unified format, which is natural because the damped wave equation does not enjoy any scaling invariance. Observe, though, that the high- and low-frequency estimates match in the borderline case $\alpha=2^j$, with $r=\tilde r=2$.
\end{rem}

\begin{rem}
	Corollary \ref{cor:wave} is optimal in the loose sense that it provides a similar result as Corollary \ref{cor:halfwave} for high frequencies. Moreover, it recovers the optimal Strichartz estimates for the classical wave equation in the limit $\alpha\to 0$. Corollary \ref{cor:wave} also displays optimality in its control of low frequencies. Indeed, let us consider a solution $u_c(t,x)$, for each $c>0$, of the damped wave equation
	\begin{equation*}
		\left\{
			\begin{aligned}
				(c^{-2}\partial_t^2 +\alpha\partial_t -\Delta) u_c(t,x)&= F(t,x)
				\\
				u_c(0,x)&=f(x)
				\\
				\partial_t u_c(0,x)&=g(x)
			\end{aligned}
		\right.
	\end{equation*}
	on $t\in [0,T)$. In particular, it follows that $\tilde u_c(t,x)\bydef u_c(c^{-1} t,x)$ solves
	\begin{equation*}
		\left\{
			\begin{aligned}
				(\partial_t^2 +c\alpha\partial_t -\Delta) \tilde u_c(t,x)&= F(c^{-1} t,x)
				\\
				\tilde u_c(0,x)&=f(x)
				\\
				\partial_t \tilde u_c(0,x)&=c^{-1} g(x)
			\end{aligned}
		\right.
	\end{equation*}
	on $t\in [0,cT)$. Therefore, applying Corollary \ref{cor:wave} to $\tilde u_c$, with $r=\tilde r=2$, yields the low-frequency estimates
	\begin{equation*}
		\begin{aligned}
			\norm{\Delta_j \partial_t u_c}_{L^q_tL^2_x}
			&\lesssim
			\frac 1{c^2}\left(\frac{T}{1+c^2\alpha T}\right)^{\frac 1q}\norm{\Delta_j g}_{L^2_x}
			+\frac 1{\alpha}\left(\frac{\alpha 2^{2j}T}{\alpha+ 2^{2j}T}\right)^\frac 1q2^{j\left(1-\frac 2q\right)}\norm{\Delta_j \nabla f}_{L^2_x}
			\\
			&\quad+c^2\left(\frac{T}{1+c^2\alpha T}\right)^{\frac 1q+\frac 1{\tilde q}}
			\norm{\Delta_j F}_{L_t^{\tilde q'}L^{2}_x}
		\end{aligned}
	\end{equation*}
	and
	\begin{equation*}
		\begin{aligned}
			2^{j\frac 2q}\norm{\Delta_j \nabla u_c}_{L^q_tL^2_x}
			&\lesssim
			\frac 1{c^2\alpha}\left(\frac{\alpha 2^{2j}T}{\alpha+ 2^{2j}T}\right)^\frac 1q2^{j}
			\norm{\Delta_j g}_{L^2_x}
			+\left(\frac{\alpha 2^{2j}T}{\alpha+ 2^{2j}T}\right)^\frac 1q\norm{\Delta_j \nabla f}_{L^2_x}
			\\
			&\quad+\frac 1{\alpha}\left(\frac{\alpha 2^{2j}T}{\alpha+ 2^{2j}T}\right)^{\frac 1q+\frac 1{\tilde q}}
			2^{j\left(1-\frac 2{\tilde q}\right)}\norm{\Delta_j F}_{L_t^{\tilde q'}L^{2}_x},
		\end{aligned}
	\end{equation*}
	for all $j\in\mathbb{Z}$ with $2^j\leq c\alpha$. Finally, letting $c$ tend to infinity and denoting the limit of $u_c$ (in the sense of distributions) by $u$, we obtain the estimates
	\begin{equation*}
		\norm{\Delta_j \partial_t u}_{L^q_tL^2_x}
		\lesssim
		\frac 1{\alpha}\left(\frac{\alpha 2^{2j}T}{\alpha+ 2^{2j}T}\right)^\frac 1q
		2^{j2\left(1-\frac 1q\right)}\norm{\Delta_j f}_{L^2_x}
		+\frac 1\alpha\norm{\Delta_j F}_{L_t^{q}L^{2}_x}
	\end{equation*}
	and
	\begin{equation*}
		2^{j\frac 2q}\norm{\Delta_j u}_{L^q_tL^2_x}
		\lesssim
		\left(\frac{\alpha 2^{2j}T}{\alpha+ 2^{2j}T}\right)^\frac 1q\norm{\Delta_j f}_{L^2_x}
		+\frac 1{\alpha}\left(\frac{\alpha 2^{2j}T}{\alpha+ 2^{2j}T}\right)^{\frac 1q+\frac 1{\tilde q}}
		2^{-j\frac 2{\tilde q}}\norm{\Delta_j F}_{L_t^{\tilde q'}L^{2}_x},
	\end{equation*}
	for all $j\in\mathbb{Z}$ and every $q,\tilde q\in[1,\infty]$ such that $\frac 1q+\frac 1{\tilde q}\leq 1$, which are optimal parabolic estimates for the heat equation
	\begin{equation*}
		(\alpha\partial_t-\Delta)u=F
	\end{equation*}
	with initial data $u(0,x)=f(x)$.
\end{rem}

\begin{rem}
	Other attempts at establishing Strichartz estimates for the damped wave equation can be found in \cite{i19} and \cite{iw19}. However, the results obtained therein are suboptimal. Indeed, Corollary \ref{cor:wave} supersedes the results from \cite{i19} and \cite{iw19} in both the breadth of the range of integrability parameters that it handles, and the sharpness of regularity gains that it produces.
\end{rem}

\begin{proof}
	We begin by introducing
	\begin{equation*}
		e(t,x)\bydef \partial_t u(t,x)
		\quad\text{and}\quad
		b(t,x)\bydef i|D|u(t,x).
	\end{equation*}
	It then follows that $(e,b)$ is a solution of the system
	\begin{equation*}
		\left\{
		\begin{aligned}
			&\partial_t e-i|D|b +\alpha e = F,
			\\
			&\partial_t b-i|D|e =0,
		\end{aligned}
		\right.
	\end{equation*}
	with initial data $\big(e(0,x),b(0,x)\big)=\big(g(x),i|D|f(x)\big)$. This system is reminiscent of Maxwell's equations, which are studied in Section \ref{section:perfect fluid}, and it can be recast as
	\begin{equation*}
		\partial_t
		\begin{pmatrix}
			e\\b
		\end{pmatrix}
		=
		\mathcal{L}
		\begin{pmatrix}
			e\\b
		\end{pmatrix}
		+
		\begin{pmatrix}
			F\\0
		\end{pmatrix},
	\end{equation*}
	where
	\begin{equation*}
		\mathcal{L}=\mathcal{L}(D)\bydef
		\begin{pmatrix}
			-\alpha & i|D|
			\\
			i|D| & 0
		\end{pmatrix}.
	\end{equation*}
	
	Now, a straightforward computation shows that $\mathcal{L}(\xi)$, where $\xi\in\mathbb{R}^d$, has the eigenvalues
	\begin{equation}\label{eigenvalues:def}
		\lambda_\pm(\xi)=-\frac\alpha 2\pm\sqrt{\frac{\alpha^2}4-|\xi|^2}
	\end{equation}
	in the complex field $\mathbb{C}$.
	Moreover, exploiting the trivial identities $\lambda_++\lambda_-=-\alpha$ and $\lambda_+\lambda_-=|\xi|^2$, one can readily verify that
	\begin{equation*}
		\begin{pmatrix}
			e\\b
		\end{pmatrix}
		=P_+
		\begin{pmatrix}
			e\\b
		\end{pmatrix}
		+P_-
		\begin{pmatrix}
			e\\b
		\end{pmatrix}
	\end{equation*}
	provides an eigenvector decomposition, where
	\begin{equation*}
		\begin{aligned}
			P_+
			\begin{pmatrix}
				e\\b
			\end{pmatrix}
			&=
			\frac 1{\lambda_+-\lambda_-}
			\begin{pmatrix}
				\lambda_+ e+i|D|b\\i|D|e-\lambda_-b
			\end{pmatrix}
			\\
			P_-
			\begin{pmatrix}
				e\\b
			\end{pmatrix}
			&=
			\frac 1{\lambda_--\lambda_+}
			\begin{pmatrix}
				\lambda_- e+i|D|b\\i|D|e-\lambda_+b
			\end{pmatrix}
		\end{aligned}
	\end{equation*}
	are the projections onto the eigenspaces associated with $\lambda_+(D)$ and $\lambda_-(D)$ (when these eigenvalues are distinct, i.e., when $|\xi|\neq\frac\alpha 2$), respectively.
	In particular, this decomposition allows us to deduce the representation formula
	\begin{equation}\label{representation:1}
		\begin{aligned}
			\begin{pmatrix}
				e\\b
			\end{pmatrix}(t)
			&=P_+
			\begin{pmatrix}
				e\\b
			\end{pmatrix}(t)
			+P_-
			\begin{pmatrix}
				e\\b
			\end{pmatrix}(t)
			\\
			&=
			\begin{pmatrix}
				\frac {e^{t\lambda_+}\lambda_+-e^{t\lambda_-}\lambda_-}{\lambda_+-\lambda_-} g-\frac {e^{t\lambda_+}-e^{t\lambda_-}}{\lambda_+-\lambda_-} |D|^2f
				\\
				\frac {e^{t\lambda_+}-e^{t\lambda_-}}{\lambda_+-\lambda_-}i|D|g-\frac {e^{t\lambda_+}\lambda_--e^{t\lambda_-}\lambda_+}{\lambda_+-\lambda_-}i|D|f
			\end{pmatrix}
			+\int_0^t
			\begin{pmatrix}
				\frac {e^{(t-s)\lambda_+}\lambda_+-e^{(t-s)\lambda_-}\lambda_-}{\lambda_+-\lambda_-} F\\\frac {e^{(t-s)\lambda_+}-e^{(t-s)\lambda_-}}{\lambda_+-\lambda_-}i|D|F
			\end{pmatrix}
			ds.
		\end{aligned}
	\end{equation}
	
	We want now to use \eqref{representation:1} to deduce an estimate on $\Delta_0 e$ and $\Delta_0 b$, which, by a scaling argument in the spirit of the proof of Corollary \ref{cor:halfwave}, will then result in a control on each dyadic component $\Delta_j e$ and $\Delta_j b$, with $j\in\mathbb{Z}$. Note, however, that the eigenvalues $\lambda_\pm(\xi)$ are of fundamentally different nature depending on the relative size of frequencies $\{\frac 12\leq|\xi|\leq 2\}$ with respect to $\alpha\geq 0$, which leads us to consider several cases.
	
	More specifically, on the one hand, when $\lambda_\pm\in\mathbb{R}$ (i.e., when $|\xi|\leq\frac\alpha 2$), we are going to employ the elementary controls
	\begin{equation}\label{eigenvalues:real}
		\begin{gathered}
			-\alpha\leq\lambda_-\leq-\frac\alpha 2\leq -\frac{2|\xi|^2}\alpha\leq\lambda_+\leq -\frac{|\xi|^2}\alpha,
			\\
			0\leq \frac {e^{t\lambda_+}-e^{t\lambda_-}}{\lambda_+-\lambda_-}
			=\frac {\int_{\lambda_-}^{\lambda_+}te^{ts}ds}{\lambda_+-\lambda_-}
			\leq te^{t\lambda_+}\leq te^{-t\frac{|\xi|^2}{\alpha}},
		\end{gathered}
	\end{equation}
	while, on the other hand, when $\lambda_\pm\in\mathbb{C}\setminus\mathbb{R}$ (i.e., when $|\xi|\geq\frac\alpha 2$), we are going to use the properties
	\begin{equation}\label{eigenvalues:complex}
		\begin{gathered}
			\left|\lambda_\pm\right|=|\xi|,\qquad\left|e^{-t\lambda_\pm}\right|=e^{-\frac\alpha 2 t},
			\\
			\left|\frac {e^{t\lambda_+}-e^{t\lambda_-}}{\lambda_+-\lambda_-}\right|
			=e^{-\frac\alpha 2t}\frac{\left|\sin\left(t\sqrt{|\xi|^2-\frac{\alpha^2}4}\right)\right|}{\sqrt{|\xi|^2-\frac{\alpha^2}4}}
			\leq te^{-\frac\alpha 2t}.
		\end{gathered}
	\end{equation}
	Considering that the dyadic operator $\Delta_0$ localizes frequencies to $\{\frac 12\leq|\xi|\leq 2\}$, we will then distinguish three cases:
	\begin{itemize}
		\item The complex case, where $0\leq \alpha\leq \frac 12$, so that $\lambda_\pm$ and $(\lambda_+-\lambda_-)^{-1}$ are complex and smooth on $\{\frac 12\leq|\xi|\leq 2\}$.
		\item The degenerate case, where $\frac 12<\alpha<5$ and the eigenvalues may be equal.
		\item The real case, where $\alpha \geq 5$, which implies that the eigenvalues are real and the damping phenomenon dominates the behavior of solutions on $\{\frac 12\leq|\xi|\leq 2\}$.
	\end{itemize}
	
	\paragraph{\bf The complex case} We begin by considering the range $0\leq \alpha\leq \frac 12$. In this setting, it is readily seen that the functions $\lambda_+(\xi)$, $\lambda_-(\xi)$ and $(\lambda_+(\xi)-\lambda_-(\xi))^{-1}$, as well as any number of their derivatives, are uniformly bounded on $\{\frac 13<|\xi|<3\}$, uniformly in $\alpha\in [0,\frac 12]$. In particular, by virtue of the criterion \eqref{multiplier:3} for the boundedness of multipliers, further introducing a smooth cutoff function $\psi(\xi)$ compactly supported inside $\{\frac 13<|\xi|<3\}$ and such that $\psi\equiv 1$ on $\{\frac 12\leq|\xi|\leq 2\}$, it follows that $\lambda_+(\xi)\psi(\xi)$, $\lambda_-(\xi)\psi(\xi)$ and $(\lambda_+(\xi)-\lambda_-(\xi))^{-1}\psi(\xi)$ are the symbols of bounded Fourier multipliers over any homogeneous Besov space. Therefore, we deduce from \eqref{representation:1} that
	\begin{equation}\label{strichartz:4}
		\begin{aligned}
			\norm{\Delta_0 (e, b)}_{L^q_tL^r_x}
			&\lesssim
			\sum_{\pm}\norm{e^{t\lambda_\pm}\Delta_0(f,g)}_{L^q_tL^r_x}
			+\norm{\int_0^te^{(t-s)\lambda_\pm}\Delta_0Fds}_{L^q_tL^r_x}
			\\
			&=\sum_{\pm}\norm{e^{-\frac\alpha 2t}e^{\pm it\delta(D)}\Delta_0(f,g)}_{L^q_tL^r_x}
			+\norm{\int_0^te^{-\frac\alpha 2(t-s)}e^{\pm i(t-s)\delta(D)}\Delta_0Fds}_{L^q_tL^r_x},
		\end{aligned}
	\end{equation}
	where we introduced the notation
	\begin{equation*}
		\delta(\xi)\bydef\sqrt{|\xi|^2-\frac{\alpha^2}4},
	\end{equation*}
	for convenience.
	
	Now, classically, the Stationary Phase Method can be used (see \cite[Proposition 8.15]{bcd11}, for instance) to show that
	\begin{equation*}
		\left|\int_{\mathbb{R}^d}e^{ix\cdot\xi}e^{\pm it|\xi|}\psi(\xi)d\xi\right|\leq \frac{C_\psi}{t^{\frac{d-1}2}},
	\end{equation*}
	for all $t>0$, where the constant $C_\psi>0$ is independent of $t$ and $x$, and $\psi$ is any smooth compactly supported function whose support does not contain the origin.
	A similar estimate holds true, uniformly in $\alpha\in[0,\frac 12]$, if one replaces $|\xi|$ by $\delta(\xi)$, and if the support of $\psi$ is disjoint from the closed ball $\{|\xi|\leq\frac\alpha 2\}$. More precisely, we claim that
	\begin{equation}\label{dispersion:1}
		\left|\int_{\mathbb{R}^d}e^{ix\cdot\xi}e^{\pm it\delta(\xi)}\psi(\xi)d\xi\right|\leq \frac{C_\psi}{t^{\frac{d-1}2}},
	\end{equation}
	whenever $\operatorname{supp}\psi\subset\{|\xi|>\frac 14\geq \frac\alpha 2\}$, where $C_\psi>0$ is independent of $t>0$, $x\in\mathbb{R}^d$ and $\alpha\in[0,\frac 12]$. For the sake of completeness, we provide a justification of \eqref{dispersion:1} in Appendix \ref{dispersion:2}.
	
	Therefore, introducing the flow
	\begin{equation*}
		U(t)f(x)\bydef e^{\pm it\delta(D)}\psi(D)f(x),
	\end{equation*}
	for some fixed compactly supported cutoff $\psi(\xi)$ such that $\operatorname{supp}\psi\subset\{|\xi|>\frac 14\geq \frac\alpha 2\}$ and $\psi\equiv 1$ on $\{\frac 12\leq\xi\leq 2\}$, we see, in view of \eqref{dispersion:1}, that $U(t)$ satisfies all hypotheses of Proposition \ref{strichartz:2} with $H=L^2(\mathbb{R}^d)$, $X=\mathbb{R}^d$ and $\sigma=\frac {d-1}2$.
	Hence, we conclude from \eqref{strichartz:4} that
	\begin{equation}\label{strichartz:5}
		\begin{aligned}
			\norm{\Delta_0 (e,b)}_{L^q_tL^r_x}
			&\lesssim \left(\frac T{1+\alpha T}\right)^{\frac 1q+\frac {d-1}2\left(\frac 1r-\frac 1 2\right)}\norm{\Delta_0 (f,g)}_{L^2_x}
			\\
			&\quad+\left(\frac T{1+\alpha T}\right)^{\frac 1q+\frac 1{\tilde q}+\frac {d-1}2\left(\frac 1r+\frac 1{\tilde r}-1\right)}
			\norm{\Delta_0 F}_{L^{\tilde q'}_tL^{\tilde r'}_x},
		\end{aligned}
	\end{equation}
	for all admissible exponent pairs, when $0\leq \alpha\leq \frac 12$.
	
	\paragraph{\bf The degenerate case} We are now looking at the range $\frac 12<\alpha<5$. This case is easily settled by the use of \eqref{eigenvalues:real} and \eqref{eigenvalues:complex}, which allows us to deduce, whenever $\frac 12\leq|\xi|\leq 2$, that
	\begin{equation*}
		\begin{aligned}
			\norm{\frac {e^{t\lambda_+}-e^{t\lambda_-}}{\lambda_+-\lambda_-}}_{L^c_t}&\leq \norm{te^{-\frac t{20}}}_{L^c_t}
			\lesssim \left(\frac{T^2}{1+T^2}\right)^\frac 1c\leq 2^\frac 1c\left(\frac T{1+T}\right)^\frac 2c
			\lesssim\left(\frac{T}{1+ T}\right)^\frac 1c,
			\\
			\norm{\frac {e^{t\lambda_+}\lambda_+-e^{t\lambda_-}\lambda_-}{\lambda_+-\lambda_-}}_{L^c_t}
			&\leq \norm{e^{t\lambda_+}}_{L^c_t}
			+\norm{\lambda_-\frac {e^{t\lambda_+}-e^{t\lambda_-}}{\lambda_+-\lambda_-}}_{L^c_t}
			\lesssim\left(\frac{T}{1+ T}\right)^\frac 1c,
			\\
			\norm{\frac {e^{t\lambda_+}\lambda_--e^{t\lambda_-}\lambda_+}{\lambda_+-\lambda_-}}_{L^c_t}
			&\leq \norm{e^{t\lambda_+}}_{L^c_t}
			+\norm{\lambda_+\frac {e^{t\lambda_+}-e^{t\lambda_-}}{\lambda_+-\lambda_-}}_{L^c_t}
			\lesssim\left(\frac{T}{1+ T}\right)^\frac 1c,
		\end{aligned}
	\end{equation*}
	for any $1\leq c\leq\infty$. Indeed, incorporating these controls into \eqref{representation:1} and recalling that the space of Fourier multipliers on $L^2(\mathbb{R}^d)$ is isomorphic to $L^\infty(\mathbb{R}^d)$ leads to
	\begin{equation}\label{strichartz:6}
		\begin{aligned}
			\norm{\Delta_0 (e,b)}_{L^q_tL^r_x}
			&\lesssim \norm{\Delta_0 (e,b)}_{L^q_tL^2_x}
			\\
			&\lesssim \left(\frac T{1+ T}\right)^{\frac 1q}\norm{\Delta_0 (f,g)}_{L^2_x}
			+\left(\frac T{1+ T}\right)^{\frac 1q+\frac 1{\tilde q}}
			\norm{\Delta_0 F}_{L^{\tilde q'}_tL^{2}_x}
			\\
			&\lesssim \left(\frac T{1+T}\right)^{\frac 1q+\frac {d-1}2\left(\frac 1r-\frac 1 2\right)}\norm{\Delta_0 (f,g)}_{L^2_x}
			\\
			&\quad+\left(\frac T{1+T}\right)^{\frac 1q+\frac 1{\tilde q}+\frac {d-1}2\left(\frac 1r+\frac 1{\tilde r}-1\right)}
			\norm{\Delta_0 F}_{L^{\tilde q'}_tL^{\tilde r'}_x},
		\end{aligned}
	\end{equation}
	for all admissible exponent pairs, when $\frac 12<\alpha<5$.
	
	\paragraph{\bf The real case} In the remaining case, we assume that $\alpha\geq 5$. In particular, when $\frac 12\leq|\xi|\leq 2$, it holds that
	\begin{equation*}
		\lambda_+-\lambda_-=\sqrt{\alpha^2-4|\xi|^2}\geq \sqrt{\alpha^2-16}\geq \frac{3\alpha}5.
	\end{equation*}
	Furthermore, employing \eqref{eigenvalues:real}, one finds that
	\begin{equation*}
		\norm{e^{t\lambda_+}}_{L^c_t}\leq \norm{e^{-\frac t{4\alpha}}}_{L^c_t}\lesssim\left(\frac{\alpha T}{\alpha+T}\right)^\frac 1c
		\quad\text{and}\quad
		\norm{e^{t\lambda_-}}_{L^c_t}\leq \norm{e^{-\frac {\alpha t}{2}}}_{L^c_t} \lesssim \left(\frac{T}{1+\alpha T}\right)^\frac 1c,
	\end{equation*}
	for any $1\leq c\leq\infty$. Therefore, we deduce from \eqref{representation:1} and \eqref{eigenvalues:real} that
	\begin{equation}\label{strichartz:7}
		\begin{aligned}
			\norm{\Delta_0 e}_{L^q_tL^r_x}&\lesssim \norm{\Delta_0 e}_{L^q_tL^2_x}
			\\
			&\lesssim
			\left(\frac{\alpha^{1-2q} T}{\alpha+ T}+\frac{T}{1+\alpha T}\right)^\frac 1q
			\norm{\Delta_0 g}_{L^2_x}
			+\left(\frac{\alpha^{1-q} T}{\alpha+ T}+\frac{\alpha^{-q}T}{1+\alpha T}\right)^\frac 1q
			\norm{\Delta_0 f}_{L^2_x}
			\\
			&\quad+
			\left(\frac{\alpha^{1-2\left(\frac 1q+\frac 1{\tilde q}\right)^{-1}} T}{\alpha+ T}+\frac{T}{1+\alpha T}\right)^{\frac 1q+\frac 1{\tilde q}}
			\norm{\Delta_0 F}_{L_t^{\tilde q'}L^{2}_x}
			\\
			&\lesssim
			\left(\frac{T}{1+\alpha T}\right)^{\frac 1q}\norm{\Delta_0 g}_{L^2_x}
			+\left(\frac{\alpha^{1-q} T}{\alpha+ T}\right)^\frac 1q\norm{\Delta_0 f}_{L^2_x}
			\\
			&\quad+\left(\frac{T}{1+\alpha T}\right)^{\frac 1q+\frac 1{\tilde q}}
			\norm{\Delta_0 F}_{L_t^{\tilde q'}L^{\tilde r'}_x}
		\end{aligned}
	\end{equation}
	and
	\begin{equation}\label{strichartz:8}
		\begin{aligned}
			\norm{\Delta_0 b}_{L^q_tL^r_x}&\lesssim \norm{\Delta_0 b}_{L^q_tL^2_x}
			\\
			&\lesssim
			\left(\frac{\alpha^{1-q} T}{\alpha+ T}+\frac{\alpha^{-q}T}{1+\alpha T}\right)^\frac 1q
			\norm{\Delta_0 g}_{L^2_x}
			+\left(\frac{\alpha T}{\alpha+ T}+\frac{\alpha^{-2q}T}{1+\alpha T}\right)^\frac 1q\norm{\Delta_0 f}_{L^2_x}
			\\
			&\quad+
			\left(\frac{\alpha^{1-\left(\frac 1q+\frac 1{\tilde q}\right)^{-1}} T}{\alpha+ T}+\frac{\alpha^{-\left(\frac 1q+\frac 1{\tilde q}\right)^{-1}}T}{1+\alpha T}\right)^{\frac 1q+\frac 1{\tilde q}}
			\norm{\Delta_0 F}_{L_t^{\tilde q'}L^{2}_x}
			\\
			&\lesssim
			\left(\frac{\alpha^{1-q} T}{\alpha+ T}\right)^\frac 1q\norm{\Delta_0 g}_{L^2_x}
			+\left(\frac{\alpha T}{\alpha+ T}\right)^\frac 1q\norm{\Delta_0 f}_{L^2_x}
			\\
			&\quad+\left(\frac{\alpha^{1-\left(\frac 1q+\frac 1{\tilde q}\right)^{-1}} T}{\alpha+ T}\right)^{\frac 1q+\frac 1{\tilde q}}
			\norm{\Delta_0 F}_{L_t^{\tilde q'}L^{\tilde r'}_x},
		\end{aligned}
	\end{equation}
	for all admissible exponent pairs, whenever $\alpha\geq 5$.
	
	\paragraph{\bf Scaling argument and conclusion of proof} We are now in a position to conclude the justification of the corollary.
	In order to deduce an estimate on $\Delta_j(e,b)$, for all $j\in\mathbb{Z}$, from \eqref{strichartz:5}, \eqref{strichartz:6}, \eqref{strichartz:7} and \eqref{strichartz:8}, we conduct now a scaling analysis in the spirit of the proof of Corollary \ref{cor:halfwave}. To that end, we introduce
	\begin{equation*}
		\begin{aligned}
			(e_j,b_j)(t,x)&\bydef (e,b)\left(\frac t{2^j},\frac x{2^j}\right),
			&
			F_j(t,x)&\bydef \frac 1{2^j}F\left(\frac t{2^j},\frac x{2^j}\right),
			\\
			f_j(x)&\bydef 2^jf \left(\frac x{2^j}\right),
			&
			g_j(x)&\bydef g \left(\frac x{2^j}\right),
		\end{aligned}
	\end{equation*}
	and observe that $(e_j,b_j)$ solves
	\begin{equation*}
		\left\{
		\begin{aligned}
			&\partial_t e_j-i|D|b_j +2^{-j}\alpha e_j = F_j,
			\\
			&\partial_t b_j-i|D|e_j =0,
		\end{aligned}
		\right.
	\end{equation*}
	on $[0,2^jT)$, with initial data $\big(e_j(0,x),b_j(0,x)\big)=\big(g_j(x),i|D|f_j(x)\big)$.
	
	Then, noticing that
	\begin{equation*}
		\begin{aligned}
			\norm{\Delta_j (e,b)}_{L^q_tL^r_x}&=2^{-j\left(\frac 1q+\frac dr\right)}\norm{\Delta_0 (e_j,b_j)}_{L^q_tL^r_x}
			\\
			\norm{\Delta_j g}_{L^2_x}&=2^{-j\frac d2}\norm{\Delta_0 g_j}_{L^2_x}
			\\
			\norm{\Delta_j f}_{L^2_x}&=2^{-j\left(1+\frac d2\right)}\norm{\Delta_0 f_j}_{L^2_x}
			\\
			\norm{\Delta_j F}_{L^{\tilde q'}_tL^{\tilde r'}_x}&=2^{j\left(\frac 1{\tilde q}-d(1-\frac 1{\tilde r})\right)}
			\norm{\Delta_0 F_j}_{L^{\tilde q'}_tL^{\tilde r'}_x}
		\end{aligned}
	\end{equation*}
	and applying \eqref{strichartz:5} and \eqref{strichartz:6} to $(e_j,b_j)$ yields the estimate
	\begin{equation*}
		\begin{aligned}
			\norm{\Delta_j (e,b)}_{L^q_tL^r_x}
			&\lesssim
			\left(\frac{T}{1+\alpha T}\right)^{\frac 1q+\frac {d-1}2\left(\frac 1r-\frac 1 2\right)}
			2^{j\frac {d+1}2\left(\frac 12-\frac 1r\right)}
			\norm{\Delta_j (g,\nabla f)}_{L^2_x}
			\\
			&\quad+\left(\frac{T}{1+\alpha T}\right)^{\frac 1q+\frac 1{\tilde q}+\frac {d-1}2\left(\frac 1r+\frac 1{\tilde r}-1\right)}
			2^{j\frac{d+1}2\left(1-\frac 1r-\frac 1{\tilde r}\right)}\norm{\Delta_j F}_{L_t^{\tilde q'}L^{\tilde r'}_x},
		\end{aligned}
	\end{equation*}
	whenever $2^j\geq \alpha$.
	
	Similarly, if $2^j\leq\alpha$, then, applying \eqref{strichartz:6}, \eqref{strichartz:7} and \eqref{strichartz:8} to $(e_j,b_j)$ leads to the controls
	\begin{equation*}
		\begin{aligned}
			\norm{\Delta_j e}_{L^q_tL^r_x}
			&\lesssim
			\left(\frac{T}{1+\alpha T}\right)^{\frac 1q}2^{jd\left(\frac 12-\frac 1r\right)}\norm{\Delta_j g}_{L^2_x}
			+\frac 1\alpha\left(\frac{\alpha T}{\alpha+ 2^{2j}T}\right)^\frac 1q2^{j\left(1+d(\frac 12-\frac 1r)\right)}\norm{\Delta_j \nabla f}_{L^2_x}
			\\
			&\quad+\left(\frac{T}{1+\alpha T}\right)^{\frac 1q+\frac 1{\tilde q}}
			2^{jd\left(1-\frac 1r-\frac 1{\tilde r}\right)}\norm{\Delta_j F}_{L_t^{\tilde q'}L^{\tilde r'}_x}
		\end{aligned}
	\end{equation*}
	and
	\begin{equation*}
		\begin{aligned}
			\norm{\Delta_j b}_{L^q_tL^r_x}
			&\lesssim
			\frac 1\alpha\left(\frac{\alpha T}{\alpha+ 2^{2j}T}\right)^\frac 1q2^{j\left(1+d(\frac 12-\frac 1r)\right)}\norm{\Delta_j g}_{L^2_x}
			+\left(\frac{\alpha T}{\alpha+ 2^{2j}T}\right)^\frac 1q2^{jd\left(\frac 12-\frac 1r\right)}\norm{\Delta_j \nabla f}_{L^2_x}
			\\
			&\quad+\frac 1\alpha\left(\frac{\alpha T}{\alpha+ 2^{2j}T}\right)^{\frac 1q+\frac 1{\tilde q}}
			2^{j\left(1+d(1-\frac 1r-\frac 1{\tilde r})\right)}\norm{\Delta_j F}_{L_t^{\tilde q'}L^{\tilde r'}_x},
		\end{aligned}
	\end{equation*}
	which concludes the proof of the corollary.
\end{proof}

\begin{cor}[Damped Maxwell equations]\label{cor:maxwell}
	Let $d=2$ or $d=3$ and consider a solution $(E,B)(t,x):[0,T)\times\mathbb{R}^d\to\mathbb{R}^6$ of the damped Maxwell system
	\begin{equation*}
		\begin{cases}
			\begin{aligned}
				\frac{1}{c} \partial_t E - \nabla \times B + \sigma c E & = G,
				\\
				\frac{1}{c} \partial_t B + \nabla \times E & = 0,
				\\
				\div B & =0,
			\end{aligned}
		\end{cases}
	\end{equation*}
	for some initial data $(E,B)(0,x)=(E_0,B_0)(x)$, where $\sigma\geq 0$ and $c>0$.
	
	For any exponent pairs $(q,r),(\tilde q,\tilde r)\in [1,\infty]\times[2,\infty]$ which are admissible in the sense that
	\begin{equation*}
		\frac 2q+\frac {d-1} r\geq \frac {d-1} 2,
		\quad
		1+\frac {d-1} r\geq \frac {d-1} 2
		\quad\text{and}\quad
		(r,d)\neq (\infty,3)
	\end{equation*}
	and similarly for $(\tilde q,\tilde r)$, and such that
	\begin{equation*}
		\frac 1q+\frac 1{\tilde q}\leq 1,
	\end{equation*}
	one has the high-frequency estimate
	\begin{equation*}
		\begin{aligned}
			2^{-j\frac {d+1}2\left(\frac 12-\frac 1r\right)}\norm{\Delta_j (PE,B)}_{L^q_t([0,T);L^r_x)}
			\hspace{-50mm}&
			\\
			&\lesssim
			c^{\frac {d-1}2\left(\frac 12-\frac 1r\right)-\frac 2q}
			\left(\frac{c^2T}{1+\sigma c^2T}\right)^{\frac 1q+\frac {d-1}2\left(\frac 1r-\frac 1 2\right)}
			\norm{\Delta_j (PE_0,B_0)}_{L^2_x}
			\\
			&\quad+c^{1+\frac {d-1}2\left(1-\frac 1r-\frac 1{\tilde r}\right)-\frac 2q-\frac 2{\tilde q}}
			\left(\frac{c^2T}{1+\sigma c^2T}\right)^{\frac 1q+\frac 1{\tilde q}+\frac {d-1}2\left(\frac 1r+\frac 1{\tilde r}-1\right)}
			2^{j\frac{d+1}2\left(\frac 12-\frac 1{\tilde r}\right)}\norm{\Delta_j PG}_{L_t^{\tilde q'}([0,T);L^{\tilde r'}_x)},
		\end{aligned}
	\end{equation*}
	for all $j\in\mathbb{Z}$ with $2^j\geq \sigma c$, and the low-frequency estimates
	\begin{equation*}
		\begin{aligned}
			2^{-jd\left(\frac 12-\frac 1r\right)}\norm{\Delta_j PE}_{L^q_t([0,T);L^r_x)}
			&\lesssim
			\left(\frac{T}{1+\sigma c^2 T}\right)^{\frac 1q}\norm{\Delta_j PE_0}_{L^2_x}
			\\
			&\quad +\frac 1{\sigma c}\left(\frac{\sigma 2^{2j}T}{\sigma + 2^{2j}T}\right)^\frac 1q2^{j\left(1-\frac 2q\right)}\norm{\Delta_j B_0}_{L^2_x}
			\\
			&\quad+c\left(\frac{T}{1+\sigma c^2 T}\right)^{\frac 1q+\frac 1{\tilde q}}
			2^{jd\left(\frac 12-\frac 1{\tilde r}\right)}\norm{\Delta_j PG}_{L_t^{\tilde q'}([0,T);L^{\tilde r'}_x)}
		\end{aligned}
	\end{equation*}
	and
	\begin{equation*}
		\begin{aligned}
			2^{-j\left(d(\frac 12-\frac 1r)-\frac 2q\right)}\norm{\Delta_j B}_{L^q_t([0,T);L^r_x)}
			&\lesssim
			\frac 1{\sigma c}\left(\frac{\sigma 2^{2j}T}{\sigma + 2^{2j}T}\right)^\frac 1q2^{j}
			\norm{\Delta_j PE_0}_{L^2_x}
			+\left(\frac{\sigma 2^{2j}T}{\sigma + 2^{2j}T}\right)^\frac 1q\norm{\Delta_j B_0}_{L^2_x}
			\\
			&\quad+\frac 1{\sigma}\left(\frac{\sigma 2^{2j}T}{\sigma+ 2^{2j}T}\right)^{\frac 1q+\frac 1{\tilde q}}
			2^{j\left(1+d(\frac 12-\frac 1{\tilde r})-\frac 2{\tilde q}\right)}\norm{\Delta_j PG}_{L_t^{\tilde q'}([0,T);L^{\tilde r'}_x)},
		\end{aligned}
	\end{equation*}
	for all $j\in\mathbb{Z}$ with $2^j\leq \sigma c$.
\end{cor}

\begin{rem}
	Corollary \ref{cor:maxwell} only provides estimates of the magnetic field $B$ and the divergence-free part of the electric field $PE$. Notice, though, that the divergent component $P^\perp E$ can also be estimated directly from Maxwell's system. Indeed, applying the projector $P^\perp$ to Amp\`ere's equation yields
	\begin{equation*}
		\partial_t P^\perp E + \sigma c^2 P^\perp E = cP^\perp G,
	\end{equation*}
	which leads to the representation formula
	\begin{equation*}
		P^\perp E(t)=e^{-\sigma c^2t}P^\perp E_0+c\int_0^te^{-\sigma c^2(t-s)}P^\perp G(s)ds.
	\end{equation*}
	A direct estimate then easily gives that
	\begin{equation*}
		\begin{aligned}
			2^{-jd\left(\frac 12-\frac 1r\right)}\norm{\Delta_j P^\perp E}_{L^q_tL^r_x}
			&\lesssim
			\norm{\Delta_j P^\perp E}_{L^q_tL^2_x}
			\\
			&\lesssim
			\left(\frac{T}{1+\sigma c^2 T}\right)^{\frac 1q}\norm{\Delta_j P^\perp E_0}_{L^2_x}
			+c\left(\frac{T}{1+\sigma c^2 T}\right)^{\frac 1q+\frac 1{\tilde q}}
			\norm{\Delta_j P^\perp G}_{L_t^{\tilde q'}L^{2}_x}
			\\
			&\lesssim
			\left(\frac{T}{1+\sigma c^2 T}\right)^{\frac 1q}\norm{\Delta_j P^\perp E_0}_{L^2_x}
			\\
			&\quad+c\left(\frac{T}{1+\sigma c^2 T}\right)^{\frac 1q+\frac 1{\tilde q}}
			2^{jd\left(\frac 12-\frac 1{\tilde r}\right)}\norm{\Delta_j P^\perp G}_{L_t^{\tilde q'}L^{\tilde r'}_x},
		\end{aligned}
	\end{equation*}
	for any $q,\tilde q\in [1,\infty]$ and $r,\tilde r\in [2,\infty]$, with $\frac 1q+\frac 1{\tilde q}\leq 1$.
\end{rem}

\begin{proof}
	Since $B(t,x)$ is a solenoidal field, we begin by introducing a vector potential $A(t,x)$, with $t\in [0,cT)$ and $x\in\mathbb{R}^d$, such that
	\begin{equation}\label{gauge:2}
		B(t,x)=\nabla\times A(ct, x).
	\end{equation}
	Faraday's equation $c^{-1}\partial_t B + \nabla \times E = 0$ then implies that $(\partial_tA)(ct,x)+E(t,x)$ must be curl-free, whereby there exists a scalar potential $\varphi(t,x)$, with $t\in [0,cT)$ and $x\in\mathbb{R}^d$, such that
	\begin{equation}\label{gauge:3}
		E(t,x)=\nabla\varphi(ct,x)-(\partial_tA)(ct,x).
	\end{equation}
	Observe that $A$ and $\varphi$ are not uniquely determined. Indeed, for any scalar valued potential $\psi(t,x)$, it is possible to apply the transformations
	\begin{equation*}\label{transformation:1}
		\begin{aligned}
			A(t,x)&\mapsto A(t,x)+\nabla\psi(t,x)
			\\
			\varphi(t,x)&\mapsto\varphi(t,x)+\partial_t\psi(t,x)
		\end{aligned}
	\end{equation*}
	to produce new potentials representing the same electromagnetic field $(E,B)$. Any particular choice of $A$ and $\varphi$ is called a gauge.
	
	Different choices of gauge lead to different insights into Maxwell's equations. It is therefore important to carefully select the properties fixing the gauge. A standard example of gauge fixing is the Coulomb gauge, which merely requires that $A$ be solenoidal, i.e., $\div A=0$. The Lorenz gauge, which imposes the condition
	\begin{equation*}
		\div A(t,x)=\partial_t\varphi(t,x)
	\end{equation*}
	is another classical example with the property that it produces decoupled wave equations on $A$ and $\varphi$ when there is no damping, i.e., $\sigma=0$.
	
	Here, we introduce a damped Lorenz gauge by selecting potentials $A$ and $\varphi$ solving
	\begin{equation}\label{gauge:1}
		\div A(t,x)=\partial_t\varphi(t,x)+\sigma c\varphi(t,x).
	\end{equation}
	Observe that is is always possible to find a damped Lorenz gauge. Indeed, starting from any other gauge $(A,\varphi)$, one can apply the transformations \eqref{transformation:1} with any solution $\psi(t,x)$ of the damped wave equation
	\begin{equation*}
		\partial_t^2\psi+\sigma c\partial_t\psi-\Delta\psi=\div A-\partial_t\varphi-\sigma c\varphi,
	\end{equation*}
	thereby producing new potentials satisfying \eqref{gauge:1}.
	
	Now, by inserting \eqref{gauge:2} and \eqref{gauge:3} into Amp\`ere's equation and then employing \eqref{gauge:1}, a straightforward calculation shows that the damped Lorenz gauge is a solution of the  damped wave system
	\begin{equation}\label{gauge:4}
		(\partial_t^2 +\sigma c\partial_t -\Delta) A(t,x)= -G(c^{-1}t,x)
	\end{equation}
	on $t\in [0,cT)$.
	
	Therefore, applying the Strichartz estimates for damped wave equations from Corollary \ref{cor:wave} to this system, we find, concerning high frequencies, that
	\begin{equation*}
		\begin{aligned}
			c^{\frac 1q}2^{-j\frac {d+1}2\left(\frac 12-\frac 1r\right)}\norm{\Delta_j (PE,B)}_{L^q_t([0,T);L^r_x)}
			\hspace{-30mm}&
			\\
			&\lesssim
			2^{-j\frac {d+1}2\left(\frac 12-\frac 1r\right)}\norm{\Delta_j (\partial_t PA,\nabla PA)}_{L^q_t([0,cT);L^r_x)}
			\\
			&\lesssim
			\left(\frac{cT}{1+\sigma c^2T}\right)^{\frac 1q+\frac {d-1}2\left(\frac 1r-\frac 1 2\right)}
			\norm{\Delta_j (PE_0,B_0)}_{L^2_x}
			\\
			&\quad+c^{1-\frac 1{\tilde q}}\left(\frac{cT}{1+\sigma c^2T}\right)^{\frac 1q+\frac 1{\tilde q}+\frac {d-1}2\left(\frac 1r+\frac 1{\tilde r}-1\right)}
			2^{j\frac{d+1}2\left(\frac 12-\frac 1{\tilde r}\right)}\norm{\Delta_j PG}_{L_t^{\tilde q'}([0,T);L^{\tilde r'}_x)},
		\end{aligned}
	\end{equation*}
	for all $j\in\mathbb{Z}$ with $2^j\geq \sigma c$.
	
	As for low frequencies, i.e., when $j\in\mathbb{Z}$ with $2^j\leq \sigma c$, we obtain similarly, from Corollary \ref{cor:wave}, that
	\begin{equation*}
		\begin{aligned}
			c^{\frac 1q}2^{-jd\left(\frac 12-\frac 1r\right)}\norm{\Delta_j PE}_{L^q_t([0,T);L^r_x)}
			\hspace{-10mm}&
			\\
			&=2^{-jd\left(\frac 12-\frac 1r\right)}\norm{\Delta_j \partial_t PA}_{L^q_t([0,cT);L^r_x)}
			\\
			&\lesssim
			\left(\frac{cT}{1+\sigma c^2 T}\right)^{\frac 1q}\norm{\Delta_j PE_0}_{L^2_x}
			+\frac 1{\sigma c}\left(\frac{\sigma c 2^{2j}T}{\sigma + 2^{2j}T}\right)^\frac 1q2^{j\left(1-\frac 2q\right)}\norm{\Delta_j B_0}_{L^2_x}
			\\
			&\quad+c^{1-\frac 1{\tilde q}}\left(\frac{cT}{1+\sigma c^2 T}\right)^{\frac 1q+\frac 1{\tilde q}}
			2^{jd\left(\frac 12-\frac 1{\tilde r}\right)}\norm{\Delta_j PG}_{L_t^{\tilde q'}([0,T);L^{\tilde r'}_x)}
		\end{aligned}
	\end{equation*}
	and
	\begin{equation*}
		\begin{aligned}
			c^{\frac 1q}2^{-j\left(d(\frac 12-\frac 1r)-\frac 2q\right)}\norm{\Delta_j B}_{L^q_t([0,T);L^r_x)}
			\hspace{-15mm}&
			\\
			&\lesssim
			2^{-j\left(d(\frac 12-\frac 1r)-\frac 2q\right)}\norm{\Delta_j \nabla PA}_{L^q_t([0,cT);L^r_x)}
			\\
			&\lesssim
			\frac 1{\sigma c}\left(\frac{\sigma c 2^{2j}T}{\sigma + 2^{2j}T}\right)^\frac 1q2^{j}
			\norm{\Delta_j PE_0}_{L^2_x}
			+\left(\frac{\sigma c 2^{2j}T}{\sigma + 2^{2j}T}\right)^\frac 1q\norm{\Delta_j B_0}_{L^2_x}
			\\
			&\quad+c^{1-\frac 1{\tilde q}}\frac 1{\sigma c}\left(\frac{\sigma c 2^{2j}T}{\sigma+ 2^{2j}T}\right)^{\frac 1q+\frac 1{\tilde q}}
			2^{j\left(1+d(\frac 12-\frac 1{\tilde r})-\frac 2{\tilde q}\right)}\norm{\Delta_j PG}_{L_t^{\tilde q'}([0,T);L^{\tilde r'}_x)},
		\end{aligned}
	\end{equation*}
	which concludes the proof of the corollary.
\end{proof}

The global low-frequency estimates from Corollaries \ref{cor:wave} and \ref{cor:maxwell} can be refined by considering the maximal regularity of the heat equation (without damping) discussed in Section \ref{section:damped:parabolic:1}. The next two results provide such low-frequency parabolic estimates for the wave equation and Maxwell's system, respectively.

\begin{prop}\label{cor:parabolic:wave}
	Let $d\geq 2$ and consider a solution $u(t,x)$ of the damped wave equation
	\begin{equation*}
		\left\{
			\begin{aligned}
				(\partial_t^2 +\alpha\partial_t -\Delta) u(t,x)&= F(t,x)
				\\
				u(0,x)&=f(x)
				\\
				\partial_t u(0,x)&=g(x)
			\end{aligned}
		\right.
	\end{equation*}
	with $\alpha> 0$, $t\in [0,T)$ and $x\in\mathbb{R}^d$.
	
	For any $\chi\in C^\infty_c(\mathbb{R}^d)$ and $\sigma\in\mathbb{R}$, one has the low-frequency estimates
	\begin{equation*}
		\begin{aligned}
			\norm{\chi(\alpha^{-1}D)\partial_t u}_{L^m_t([0,T);\dot B^{\sigma+\frac 2m}_{2,q})}&\lesssim
			\alpha^{-\frac 1m}\norm{g}_{\dot B^{\sigma+\frac 2m}_{2,q}}+\alpha^{\frac 1m-1}\norm{f}_{\dot B^{\sigma+2}_{2,m}}
			\\
			&\quad+\alpha^{-(1+\frac 1m-\frac 1r)}\norm{F}_{L_t^r([0,T);\dot B_{2,q}^{\sigma+\frac 2m})},
		\end{aligned}
	\end{equation*}
	for any $1<r\leq m<\infty$ and $1\leq q\leq\infty$, as well as
	\begin{equation*}
		\norm{\chi(\alpha^{-1}D)\nabla u}_{L^m_t([0,T);\dot B^{\sigma+\frac 2m}_{2,1})}\lesssim
		\alpha^{\frac 1m-1}\norm{g}_{\dot B^{\sigma+1}_{2,m}}+\alpha^{\frac 1m}\norm{f}_{\dot B^{\sigma+1}_{2,m}}
		+\alpha^{\frac 1m-\frac 1r}\norm{F}_{L_t^r([0,T);\dot B_{2,\infty}^{\sigma-1+\frac 2r})},
	\end{equation*}
	for any $1<r<m<\infty$, and
	\begin{equation*}
		\norm{\chi(\alpha^{-1}D)\nabla u}_{L^m_t([0,T);\dot B^{\sigma+\frac 2m}_{2,q})}\lesssim
		\alpha^{\frac 1m-1}\norm{g}_{\dot B^{\sigma+1}_{2,m}}+\alpha^{\frac 1m}\norm{f}_{\dot B^{\sigma+1}_{2,m}}
		+\norm{F}_{L_t^m([0,T);\dot B_{2,q}^{\sigma-1+\frac 2m})},
	\end{equation*}
	for any $1<m<\infty$ and $1\leq q\leq\infty$.
\end{prop}

\begin{proof}
	Following the proof of Corollary \ref{cor:wave}, we consider
	\begin{equation*}
		e(t,x)\bydef \partial_t u(t,x)
		\quad\text{and}\quad
		b(t,x)\bydef i|D|u(t,x).
	\end{equation*}
	In particular, one has the representation formula \eqref{representation:1}, which, for any given choice of $0<A<\frac 12$, can be recast as
	\begin{equation}\label{representation:2}
		\begin{aligned}
			e(t)
			&=e^{\frac A\alpha t\Delta}m_2(t,D)g
			-e^{\frac A\alpha t\Delta}m_1(t,D) |D|^2f
			+\int_0^te^{\frac A\alpha (t-s)\Delta}m_2(t-s,D)F ds
			\\
			&=\left(e^{\frac A\alpha t\Delta}|D|^2m_2^+(t,D)-e^{-At\alpha}m_2^-(t,D) \right)g
			-e^{\frac A\alpha t\Delta}m_1(t,D) |D|^2f
			\\
			&\quad+\int_0^t\left(e^{\frac A\alpha (t-s)\Delta}
			|D|^2m_2^+(t-s,D)-e^{-A (t-s)\alpha}
			m_2^-(t-s,D)\right)F ds,
			\\
			b(t)
			&=e^{\frac A\alpha t\Delta}
			m_1(t,D)i|D|g
			-e^{\frac A\alpha t\Delta}m_3(t,D)i|D|f
			+\int_0^te^{\frac A\alpha (t-s)\Delta}m_1(t-s,D)i|D|Fds,
		\end{aligned}
	\end{equation}
	with the time-dependent Fourier multipliers
	\begin{equation}\label{multipliers:1}
		\begin{aligned}
			m_1(t,\xi)&=\frac {e^{t\lambda_+}-e^{t\lambda_-}}{\lambda_+-\lambda_-} e^{A t\frac{|\xi|^2}\alpha},
			&
			m_2(t,\xi)&=\frac {e^{t\lambda_+}\lambda_+-e^{t\lambda_-}\lambda_-}{\lambda_+-\lambda_-} e^{A t\frac{|\xi|^2}\alpha},
			\\
			m_2^+(t,\xi)&=\frac {e^{t\lambda_+}\lambda_+}{|\xi|^2(\lambda_+-\lambda_-)} e^{A t\frac{|\xi|^2}\alpha},
			&
			m_2^-(t,\xi)&=\frac {e^{t\lambda_-}\lambda_-}{\lambda_+-\lambda_-} e^{A t\alpha},
			\\
			m_3(t,\xi)&=\frac {e^{t\lambda_+}\lambda_--e^{t\lambda_-}\lambda_+}{\lambda_+-\lambda_-} e^{A t\frac{|\xi|^2}\alpha},&&
		\end{aligned}
	\end{equation}
	where the eigenvalues $\lambda_+(\xi)$ and $\lambda_-(\xi)$ are defined in \eqref{eigenvalues:def}.
	
	Then, making use of the elementary controls \eqref{eigenvalues:real}, one can show that
	\begin{equation*}
		\begin{aligned}
			\norm{m_1\mathds{1}_{\{|\xi|\leq \frac\alpha 4\}}}_{L^\infty_{t,\xi}}
			&\lesssim \alpha^{-1},
			&
			\norm{m_2^+\mathds{1}_{\{|\xi|\leq \frac \alpha 4\}}}_{L^\infty_{t,\xi}}
			&\lesssim \alpha^{-2},
			\\
			\norm{m_2^-\mathds{1}_{\{|\xi|\leq \frac \alpha 4\}}}_{L^\infty_{t,\xi}}
			&\lesssim 1,
			&
			\norm{m_3\mathds{1}_{\{|\xi|\leq \alpha R\}}}_{L^\infty_{t,\xi}}
			&\lesssim 1.
		\end{aligned}
	\end{equation*}
	In particular, since the space of Fourier multipliers over $L^2(\mathbb{R}^d)$ is isomorphic to $L^\infty(\mathbb{R}^d)$, we conclude that $\mathds{1}_{\{|D|\leq \frac\alpha 4\}}m_1$, $\mathds{1}_{\{|D|\leq \frac\alpha 4\}}m_2^\pm$ and $\mathds{1}_{\{|D|\leq \frac\alpha 4\}}m_3$ are bounded in the sense that they satisfy \eqref{multiplier:1} for $p=2$.

	Therefore, applying Propositions \ref{damped:parabolic:1} and \ref{heat:4} to the representation formulas \eqref{representation:2} by suitably scaling time by $\frac A\alpha$, we obtain the estimates
	\begin{equation}\label{low:1}
		\begin{aligned}
			\norm{\mathds{1}_{\{|D|\leq \frac\alpha 4\}}e}_{L^m_t\dot B^{\sigma+\frac 2m}_{2,q}}
			&\lesssim
			\alpha^{\frac 1m-2}\norm{\mathds{1}_{\{|D|\leq \frac\alpha 4\}}g}_{\dot B^{\sigma+2}_{2,m}}
			+\alpha^{-\frac 1m}\norm{g}_{\dot B^{\sigma+\frac 2m}_{2,q}}
			+\alpha^{\frac 1m-1}\norm{f}_{\dot B^{\sigma+2}_{2,m}}
			\\
			&\quad+\alpha^{\frac 1m-\frac 1r-1}\norm{\mathds{1}_{\{|D|\leq \frac\alpha 4\}}F}_{L_t^r\dot B_{2,\infty}^{\sigma+\frac 2r}}
			+\alpha^{-(1+\frac 1m-\frac 1r)}\norm{F}_{L_t^r\dot B_{2,q}^{\sigma+\frac 2m}}
			\\
			&\lesssim
			\alpha^{-\frac 1m}\norm{g}_{\dot B^{\sigma+\frac 2m}_{2,q}}
			+\alpha^{\frac 1m-1}\norm{f}_{\dot B^{\sigma+2}_{2,m}}+\alpha^{-(1+\frac 1m-\frac 1r)}\norm{F}_{L_t^r\dot B_{2,q}^{\sigma+\frac 2m}}
		\end{aligned}
	\end{equation}
	and
	\begin{equation}\label{low:4}
		\norm{\mathds{1}_{\{|D|\leq \frac \alpha 4\}}b}_{L^m_t\dot B^{\sigma+\frac 2m}_{2,1}}\lesssim
		\alpha^{\frac 1m-1}\norm{g}_{\dot B^{\sigma+1}_{2,m}}+\alpha^{\frac 1m}\norm{f}_{\dot B^{\sigma+1}_{2,m}}
		+\alpha^{\frac 1m-\frac 1r}\norm{F}_{L_t^r\dot B_{2,\infty}^{\sigma-1+\frac 2r}},
	\end{equation}
	for any $1<r<m<\infty$ and $1\leq q\leq\infty$.

	If, instead of Proposition \ref{heat:4}, one uses Proposition \ref{heat:5}, then one arrives at the estimates
	\begin{equation}\label{low:2}
		\norm{\mathds{1}_{\{|D|\leq \frac\alpha 4\}}e}_{L^m_t\dot B^{\sigma+\frac 2m}_{2,q}}\lesssim
		\alpha^{-\frac 1m}\norm{g}_{\dot B^{\sigma+\frac 2m}_{2,q}}+\alpha^{\frac 1m-1}\norm{f}_{\dot B^{\sigma+2}_{2,m}}
		+\alpha^{-1}\norm{F}_{L_t^m\dot B_{2,q}^{\sigma+\frac 2m}}
	\end{equation}
	and
	\begin{equation}\label{low:5}
		\norm{\mathds{1}_{\{|D|\leq \frac\alpha 4\}}b}_{L^m_t\dot B^{\sigma+\frac 2m}_{2,q}}\lesssim
		\alpha^{\frac 1m-1}\norm{g}_{\dot B^{\sigma+1}_{2,m}}+\alpha^{\frac 1m}\norm{f}_{\dot B^{\sigma+1}_{2,m}}
		+\norm{F}_{L_t^m\dot B_{2,q}^{\sigma-1+\frac 2m}},
	\end{equation}
	for any $1<m<\infty$ and $1\leq q\leq\infty$.
	
	In order to handle frequencies lying in the range $\{\frac\alpha 4<|\xi|\leq\alpha R\}$, for any choice of parameter $R>1$ with $2AR^2<1$, we employ \eqref{eigenvalues:real} and \eqref{eigenvalues:complex} to deduce that the mutlipliers in \eqref{representation:2} satisfy that
	\begin{equation*}
		\begin{aligned}
			\norm{m_1\mathds{1}_{\{\frac\alpha 4<|\xi|\leq\alpha R\}}}_{L^\infty_{t,\xi}}
			&\lesssim \alpha^{-1},
			\\
			\norm{m_2\mathds{1}_{\{\frac\alpha 4<|\xi|\leq\alpha R\}}}_{L^\infty_{t,\xi}}
			&\leq \norm{\lambda_- m_1\mathds{1}_{\{\frac\alpha 4<|\xi|\leq\alpha R\}}}_{L^\infty_{t,\xi}}
			+\norm{e^{t(\lambda_++A\frac{|\xi|^2}\alpha)}\mathds{1}_{\{\frac\alpha 4<|\xi|\leq\alpha R\}}}_{L^\infty_{t,\xi}}
			\lesssim 1,
			\\
			\norm{m_3\mathds{1}_{\{\frac\alpha 4<|\xi|\leq\alpha R\}}}_{L^\infty_{t,\xi}}
			&\leq \norm{\lambda_+ m_1\mathds{1}_{\{\frac\alpha 4<|\xi|\leq\alpha R\}}}_{L^\infty_{t,\xi}}
			+\norm{e^{t(\lambda_++A\frac{|\xi|^2}\alpha)}\mathds{1}_{\{\frac\alpha 4<|\xi|\leq\alpha R\}}}_{L^\infty_{t,\xi}}
			\lesssim 1.
		\end{aligned}
	\end{equation*}
	Therefore, as previously, by the boundedness of multipliers and the fact that $\norm{e^{-a t}}_{L^p_t([0,\infty))}=a^{-\frac 1p}$, for any $a>0$ and $1\leq p\leq\infty$ (no need to use Propositions \ref{damped:parabolic:1}, \ref{heat:4} or \ref{heat:5}, here), we conclude from \eqref{representation:2} that
	\begin{equation}\label{low:3}
		\begin{aligned}
			\norm{\mathds{1}_{\{\frac\alpha 4<|D|\leq \alpha R\}}e}_{L^m_t\dot B^{\sigma+\frac 2m}_{2,q}}
			&\lesssim
			\alpha^{-\frac 1m}\norm{g}_{\dot B^{\sigma+\frac 2m}_{2,q}}
			+\alpha^{-\frac 1m-1}\norm{\mathds{1}_{\{|D|\leq \alpha R\}}f}_{\dot B^{\sigma+2+\frac 2m}_{2,q}}
			\\
			&\quad +\alpha^{-(1+\frac 1m-\frac 1r)}\norm{F}_{L_t^r\dot B_{2,q}^{\sigma+\frac 2m}}
			\\
			&\lesssim
			\alpha^{-\frac 1m}\norm{g}_{\dot B^{\sigma+\frac 2m}_{2,q}}+\alpha^{\frac 1m-1}\norm{f}_{\dot B^{\sigma+2}_{2,m}}
			+\alpha^{-(1+\frac 1m-\frac 1r)}\norm{F}_{L_t^r\dot B_{2,q}^{\sigma+\frac 2m}}
		\end{aligned}
	\end{equation}
	and
	\begin{equation}\label{low:6}
		\begin{aligned}
			\norm{\mathds{1}_{\{\frac\alpha 4<|D|\leq \alpha R\}}b}_{L^m_t\dot B^{\sigma+\frac 2m}_{2,q}}
			&\lesssim
			\alpha^{-\frac 1m-1}\norm{\mathds{1}_{\{|D|\leq \alpha R\}}g}_{\dot B^{\sigma+1+\frac 2m}_{2,q}}
			+\alpha^{-\frac 1m}\norm{\mathds{1}_{\{|D|\leq \alpha R\}}f}_{\dot B^{\sigma+1+\frac 2m}_{2,q}}
			\\
			&\quad+\alpha^{-2-\frac 1m+\frac 1r}\norm{\mathds{1}_{\{|D|\leq \alpha R\}}F}_{L_t^r\dot B_{2,q}^{\sigma+1+\frac 2m}}
			\\
			&\lesssim
			\alpha^{\frac 1m-1}\norm{g}_{\dot B^{\sigma+1}_{2,m}}
			+\alpha^{\frac 1m}\norm{f}_{\dot B^{\sigma+1}_{2,m}}
			+\alpha^{\frac 1m-\frac 1r}\norm{F}_{L_t^r\dot B_{2,\infty}^{\sigma-1+\frac 2r}},
		\end{aligned}
	\end{equation}
	for any $1\leq r\leq m<\infty$ and $1\leq q\leq \infty$.
	
	All in all, combining \eqref{low:1}, \eqref{low:2} and \eqref{low:3}, we obtain that
	\begin{equation*}
		\norm{\mathds{1}_{\{|D|\leq \alpha R\}}e}_{L^m_t\dot B^{\sigma+\frac 2m}_{2,q}}
		\lesssim
		\alpha^{-\frac 1m}\norm{g}_{\dot B^{\sigma+\frac 2m}_{2,q}}
		+\alpha^{\frac 1m-1}\norm{f}_{\dot B^{\sigma+2}_{2,m}}+\alpha^{-(1+\frac 1m-\frac 1r)}\norm{F}_{L_t^r\dot B_{2,q}^{\sigma+\frac 2m}},
	\end{equation*}
	for any $1<r\leq m<\infty$ and $1\leq q\leq\infty$. Similarly, combining \eqref{low:4}, \eqref{low:5} and \eqref{low:6}, we deduce that
	\begin{equation*}
		\norm{\mathds{1}_{\{|D|\leq \alpha R\}}b}_{L^m_t\dot B^{\sigma+\frac 2m}_{2,1}}\lesssim
		\alpha^{\frac 1m-1}\norm{g}_{\dot B^{\sigma+1}_{2,m}}+\alpha^{\frac 1m}\norm{f}_{\dot B^{\sigma+1}_{2,m}}
		+\alpha^{\frac 1m-\frac 1r}\norm{F}_{L_t^r\dot B_{2,\infty}^{\sigma-1+\frac 2r}},
	\end{equation*}
	for any $1<r<m<\infty$, and
	\begin{equation*}
		\norm{\mathds{1}_{\{|D|\leq \alpha R\}}b}_{L^m_t\dot B^{\sigma+\frac 2m}_{2,q}}\lesssim
		\alpha^{\frac 1m-1}\norm{g}_{\dot B^{\sigma+1}_{2,m}}+\alpha^{\frac 1m}\norm{f}_{\dot B^{\sigma+1}_{2,m}}
		+\norm{F}_{L_t^m\dot B_{2,q}^{\sigma-1+\frac 2m}},
	\end{equation*}
	for any $1<m<\infty$ and $1\leq q\leq\infty$.
	
	Finally, selecting $R$ sufficiently large so that $\operatorname{supp}\chi\subset\{|\xi|\leq R\}$ leads to the desired estimates on $\chi(\alpha^{-1}D)(e,b)$, thereby concluding the proof.
\end{proof}

\begin{rem}
	The preceding proof raises the question---is it possible to extend the statement of Proposition \ref{cor:parabolic:wave} from the $L^2$-setting (in space integrability) to a general $L^p$-setting, with $p\neq 2$? Such an extension would require dealing with the boundedness of the multipliers defined in \eqref{multipliers:1} over $L^p$. This is related to the boundedness of the Bochner--Riesz multiplier $(1-|\xi|^2)_+^\frac 12$, which is notoriously challenging and remains unsettled in general dimensions. We will therefore not be going in further detail on this subject.
\end{rem}

\begin{cor}\label{cor:parabolic:maxwell}
	Let $d=2$ or $d=3$ and consider a solution $(E,B)(t,x):[0,T)\times\mathbb{R}^d\to\mathbb{R}^6$ of the damped Maxwell system
	\begin{equation*}
		\begin{cases}
			\begin{aligned}
				\frac{1}{c} \partial_t E - \nabla \times B + \sigma c E & = G,
				\\
				\frac{1}{c} \partial_t B + \nabla \times E & = 0,
				\\
				\div B & =0,
			\end{aligned}
		\end{cases}
	\end{equation*}
	for some initial data $(E,B)(0,x)=(E_0,B_0)(x)$, where $\sigma> 0$ and $c>0$.
	
	For any $\chi\in C^\infty_c(\mathbb{R}^d)$ and $s\in\mathbb{R}$, one has the low-frequency estimates
	\begin{equation*}
		\begin{aligned}
			\norm{\chi(c^{-1}D)PE}_{L^m_t([0,T);\dot B^{s+\frac 2m}_{2,q})}
			&\lesssim
			c^{-\frac 2m}\norm{PE_0}_{\dot B^{s+\frac 2m}_{2,q}}+c^{-1}\norm{B_0}_{\dot B^{s+1}_{2,m}}
			\\
			&\quad+c^{-1+\frac 2r-\frac 2m}\norm{PG}_{L_t^r([0,T);\dot B_{2,q}^{s+\frac 2m})},
		\end{aligned}
	\end{equation*}
	for any $1<r\leq m<\infty$ and $1\leq q\leq \infty$, as well as
	\begin{equation*}
		\norm{\chi(c^{-1}D)B}_{L^m_t([0,T);\dot B^{s+\frac 2m}_{2,1})}
		\lesssim
		c^{-1}\norm{PE_0}_{\dot B^{s+1}_{2,m}}+\norm{B_0}_{\dot B^{s}_{2,m}}
		+\norm{PG}_{L_t^r([0,T);\dot B_{2,\infty}^{s-1+\frac 2r})},
	\end{equation*}
	for any $1<r<m<\infty$, and
	\begin{equation*}
		\norm{\chi(c^{-1}D)B}_{L^m_t([0,T);\dot B^{s+\frac 2m}_{2,q})}
		\lesssim
		c^{-1}\norm{PE_0}_{\dot B^{s+1}_{2,m}}+\norm{B_0}_{\dot B^{s}_{2,m}}
		+\norm{PG}_{L_t^m([0,T);\dot B_{2,q}^{s-1+\frac 2m})},
	\end{equation*}
	for any $1<m<\infty$ and $1\leq q\leq\infty$.
\end{cor}

\begin{proof}
	We follow the steps of the proof of Corollary \ref{cor:maxwell}, which involves first fixing a damped Lorenz gauge satisfying \eqref{gauge:2}, \eqref{gauge:3} and \eqref{gauge:1}. Then, applying Proposition \ref{cor:parabolic:wave} instead of Corollary \ref{cor:wave} to the damped wave system \eqref{gauge:4}, we find, for any $\chi\in C^\infty_c(\mathbb{R}^d)$ and $s\in\mathbb{R}$, that
	\begin{equation*}
		\begin{aligned}
			\norm{\chi(c^{-1}D)PE}_{L^m_t([0,T);\dot B^{s+\frac 2m}_{2,q})}
			&=c^{-\frac 1m}\norm{\chi(c^{-1}D)\partial_t PA}_{L^m_t([0,cT);\dot B^{s+\frac 2m}_{2,q})}
			\\
			&\lesssim
			c^{-\frac 2m}\norm{PE_0}_{\dot B^{s+\frac 2m}_{2,q}}+c^{-1}\norm{B_0}_{\dot B^{s+1}_{2,m}}
			\\
			&\quad+c^{-1+\frac 2r-\frac 2m}\norm{PG}_{L_t^r([0,T);\dot B_{2,q}^{s+\frac 2m})},
		\end{aligned}
	\end{equation*}
	for any $1<r\leq m<\infty$ and $1\leq q\leq \infty$. We also obtain that
	\begin{equation*}
		\begin{aligned}
			\norm{\chi(c^{-1}D)B}_{L^m_t([0,T);\dot B^{s+\frac 2m}_{2,1})}
			&\lesssim c^{-\frac 1m}\norm{\chi(c^{-1}D)\nabla PA}_{L^m_t([0,cT);\dot B^{s+\frac 2m}_{2,1})}
			\\
			&\lesssim
			c^{-1}\norm{PE_0}_{\dot B^{s+1}_{2,m}}+\norm{B_0}_{\dot B^{s}_{2,m}}
			+\norm{PG}_{L_t^r([0,T);\dot B_{2,\infty}^{s-1+\frac 2r})},
		\end{aligned}
	\end{equation*}
	for any $1<r<m<\infty$, whereas the limiting case $1<r=m<\infty$ yields that
	\begin{equation*}
		\begin{aligned}
			\norm{\chi(c^{-1}D)B}_{L^m_t([0,T);\dot B^{s+\frac 2m}_{2,q})}
			&\lesssim c^{-\frac 1m}\norm{\chi(c^{-1}D)\nabla PA}_{L^m_t([0,cT);\dot B^{s+\frac 2m}_{2,q})}
			\\
			&\lesssim
			c^{-1}\norm{PE_0}_{\dot B^{s+1}_{2,m}}+\norm{B_0}_{\dot B^{s}_{2,m}}
			+\norm{PG}_{L_t^m([0,T);\dot B_{2,q}^{s-1+\frac 2m})},
		\end{aligned}
	\end{equation*}
	for any $1\leq q\leq\infty$, which concludes the proof.
\end{proof}


\section{Perfect incompressible two-dimensional plasmas}\label{section:perfect fluid}

We are now going to apply the damped Strichartz estimates for Maxwell's system, established in the preceding section, to the analysis of the two-dimensional incompressible Euler--Maxwell system \eqref{EM}. The main goal of this section is to establish Theorems \ref{main:1}, \ref{main:2} and \ref{main:3}, below.

In order to conveniently state the results, recall first that we denote the initial energy by
\begin{equation*}
	\mathcal{E}_0 \bydef \norm {(u_0,E_0,B_0)}_{L^2}.
\end{equation*}
For ease of notation, we also introduce a natural frequency-decomposition of Besov and Chemin--Lerner spaces with respect to the speed of light $c>0$. More precisely, we define the Besov semi-norms
\begin{equation*}
	\left\|f\right\|_{\dot B^{s}_{p,q,<}}\bydef
	\left(
	\sum_{\substack{k\in\mathbb{Z}\\ 2^k< \sigma c}} 2^{ksq}
	\left\|\Delta_{k}f\right\|_{L^p}^q\right)^\frac{1}{q}
	\quad\text{and}\quad
	\left\|f\right\|_{\dot B^{s}_{p,q,>}}\bydef
	\left(
	\sum_{\substack{k\in\mathbb{Z}\\ 2^k\geq \sigma c}} 2^{ksq}
	\left\|\Delta_{k}f\right\|_{L^p}^q\right)^\frac{1}{q},
\end{equation*}
as well as the Chemin--Lerner semi-norms
\begin{equation*}
	\left\|f\right\|_{\widetilde L^r_t\dot B^{s}_{p,q,<}}\bydef
	\left(
	\sum_{\substack{k\in\mathbb{Z}\\ 2^k< \sigma c}} 2^{ksq}
	\left\|\Delta_{k}f\right\|_{L^r_tL^p_x}^q\right)^\frac{1}{q}
	\quad\text{and}\quad
	\left\|f\right\|_{\widetilde L^r_t\dot B^{s}_{p,q,>}}\bydef
	\left(
	\sum_{\substack{k\in\mathbb{Z}\\ 2^k\geq \sigma c}} 2^{ksq}
	\left\|\Delta_{k}f\right\|_{L^r_tL^p_x}^q\right)^\frac{1}{q},
\end{equation*}
for any $s\in\mathbb{R}$ and $0<p,q,r\leq\infty$ (with obvious modifications if $q$ is infinite), where the constant $\sigma>0$ is the electrical conductivity used in the original Euler--Maxwell system \eqref{EM}.

\begin{thm}\label{main:1}
	Let $p$ and $\varepsilon$ be any real numbers in $(2,\infty)$ and $(0,1)$, respectively. There is a constant $C_*>0$ such that, if the initial data $(u_0,E_0,B_0)$, with $\div u_0=\div E_0=\div B_0$, has the two-dimensional normal structure \eqref{structure:2dim} and belongs to $\left((H^1\cap\dot W^{1,p})\times\big(H^1\cap \dot B^\frac 74_{2,1}\big)^2\right)(\mathbb{R}^2)$ with
	\begin{equation}\label{initial:1}
		\left(\mathcal{E}_0+\norm{u_0}_{\dot H^1\cap\dot W^{1,p}}
		+\norm{(E_0,B_0)}_{\dot H^1}
		+c^{-\frac 34}\norm{(E_0,B_0)}_{\dot B^\frac 74_{2,1}}\right)
		C_*e^{C_*\mathcal{E}_0^{4+\varepsilon}}< c,
	\end{equation}
	where $c>0$ is the speed of light,
	then there is a global weak solution $(u,E,B)\in L^\infty(\mathbb{R}^+;L^2)$ to the two-dimensional Euler--Maxwell system \eqref{EM}, with the normal structure \eqref{structure:2dim}, satisfying the energy inequality \eqref{energy-inequa} and enjoying the additional regularity
	\begin{equation}\label{propagation:damping}
		\begin{gathered}
			u\in L^\infty(\mathbb{R}^+;\dot H^1\cap \dot W^{1,p}),
			\quad
			(E,B)\in L^\infty(\mathbb{R}^+;\dot H^1),
			\quad
			c^{-\frac 34}(E,B)\in \widetilde L^\infty(\mathbb{R}^+;\dot B^\frac 74_{2,1}),
			\\
			(cE,B)\in L^2(\mathbb{R}^+;\dot H^1),
			\quad
			B\in L^2(\mathbb{R}^+;\dot B^2_{2,1,<}),
			\\
			(E,B)\in \widetilde L^2(\mathbb{R}^+;\dot B^1_{\infty,1,>}),
			\quad
			c^\frac 14 E\in \widetilde L^2(\mathbb{R}^+;\dot B^\frac 74_{2,1}),
			\quad
			c^\frac 14 B\in \widetilde L^2(\mathbb{R}^+;\dot B^\frac 74_{2,1,>}).
		\end{gathered}
	\end{equation}
	It is to be emphasized that the bounds in \eqref{propagation:damping} are uniform in any set of initial data such that the left-hand side of \eqref{initial:1} remains bounded.
\end{thm}

\begin{rem}
	For any fixed initial data $(u_0,E_0,B_0)$ satisfying the requirements of Theorem \ref{main:1}, it is possible to improve the uniform controls \eqref{propagation:damping} by showing that the bound
	\begin{equation*}
		(E,B)\in \widetilde L^\infty(\mathbb{R}^+;\dot B^\frac 74_{2,1})
	\end{equation*}
	holds uniformly as $c\to\infty$. This is clarified in Section \ref{fixed:data:0}, below.
\end{rem}

\begin{rem}
	Note that we do not make any claim concerning the uniqueness of solutions produced by Theorem \ref{main:1}. However, Theorem \ref{main:2} below strengthens the statement of Theorem \ref{main:1} by achieving such uniqueness, provided the initial vorticity is bounded pointwise.
\end{rem}

\begin{rem}
	Employing the bounds \eqref{propagation:damping}, one can easily show that $(\nabla E,\nabla B)\in L^2(\mathbb{R}^+;L^\infty)$. Indeed, making use of straighforward embeddings in Besov and Chemin--Lerner spaces, we deduce that
	\begin{equation*}
		\begin{aligned}
			\norm{\nabla E}_{L^2(\mathbb{R}^+;L^\infty)}&\leq \norm{\nabla E}_{\widetilde L^2(\mathbb{R}^+;\dot B^0_{\infty,1})}
			\lesssim \norm{E}_{\widetilde L^2(\mathbb{R}^+;\dot B^1_{\infty,1,<})}+\norm{E}_{\widetilde L^2(\mathbb{R}^+;\dot B^1_{\infty,1,>})}
			\\
			&\lesssim \norm{cE}_{\widetilde L^2(\mathbb{R}^+;\dot B^0_{\infty,\infty,<})}+\norm{E}_{\widetilde L^2(\mathbb{R}^+;\dot B^1_{\infty,1,>})}
			\\
			&\lesssim \norm{cE}_{\widetilde L^2(\mathbb{R}^+;\dot B^1_{2,\infty,<})}+\norm{E}_{\widetilde L^2(\mathbb{R}^+;\dot B^1_{\infty,1,>})}
		\end{aligned}
	\end{equation*}
	and
	\begin{equation}\label{infinity:bound:1}
		\begin{aligned}
			\norm{\nabla B}_{L^2(\mathbb{R}^+;L^\infty)}&\leq \norm{\nabla B}_{L^2(\mathbb{R}^+;\dot B^0_{\infty,1})}
			\lesssim \norm{B}_{L^2(\mathbb{R}^+;\dot B^1_{\infty,1,<})}+\norm{B}_{\widetilde L^2(\mathbb{R}^+;\dot B^1_{\infty,1,>})}
			\\
			&\lesssim \norm{B}_{L^2(\mathbb{R}^+;\dot B^2_{2,1,<})}+\norm{B}_{\widetilde L^2(\mathbb{R}^+;\dot B^1_{\infty,1,>})}.
		\end{aligned}
	\end{equation}
\end{rem}

\begin{rem}
	The initial condition \eqref{initial:1} can be interpreted as a mere strengthening of the property that the velocity of the fluid cannot exceed the speed of light, i.e.,
	\begin{equation*}
		\|u\|_{L^\infty_{t,x}}\leq c.
	\end{equation*}
	On purely physical grounds, this condition seems therefore quite reasonable and is not very restrictive.
\end{rem}

The previous result only covers the case $\omega_0\in L^2\cap L^p$ in the range of parameters $p\in(2,\infty)$. The next result strengthens Theorem \ref{main:1} by assuming that $\omega_0\in L^2\cap L^\infty$.

\begin{thm}\label{main:2}
	If, in addition to all hypotheses of Theorem \ref{main:1}, for some given $p\in (2,\infty)$, one also assumes that $\omega_0\in L^\infty$ (but not necessarily $\nabla u_0\in L^\infty$), then the solution produced by Theorem \ref{main:1} satisfies the additional bound $\omega\in L^\infty(\mathbb{R}^+;L^\infty)$ and is unique in the space of all solutions $(\bar u, \bar E,\bar B)$ to the Euler--Maxwell system \eqref{EM} satisfying the bounds, locally in time,
	\begin{equation*}
		(\bar u, \bar E,\bar B)\in L^\infty_tL^2_x,
		\qquad \bar u \in L^{2}_tL^\infty_x,
		\qquad \bar j \in L^{2}_{t,x},
	\end{equation*}
	and having the same initial data.
\end{thm}

We address now the propagation of regularity in the Euler--Maxwell system \eqref{EM} with the following result.

\begin{thm}\label{main:3}
	Consider parameters $p\in (2,\infty)$, $\varepsilon\in (0,1)$, $s\in(\frac 74,2)$ and $n\in [1,\infty]$. There is a constant $C_{**}>0$ such that, if the initial data $(u_0,E_0,B_0)$, with $\div u_0=\div E_0=\div B_0$, has the two-dimensional normal structure \eqref{structure:2dim} and belongs to $\left((H^1\cap\dot W^{1,p})\times\big(H^1\cap \dot B^s_{2,n}\big)^2\right)(\mathbb{R}^2)$ with
	\begin{equation}\label{initial:3}
		\left(\mathcal{E}_0+\norm{u_0}_{\dot H^1\cap\dot W^{1,p}}
		+\norm{(E_0,B_0)}_{\dot H^1}
		+c^{1-s}\norm{(E_0,B_0)}_{\dot B^s_{2,n}}\right)
		C_{**}e^{C_{**}\mathcal{E}_0^{4+\varepsilon}}< c,
	\end{equation}
	where $c>0$ is the speed of light,
	then there is a global weak solution $(u,E,B)\in L^\infty(\mathbb{R}^+;L^2)$ to the two-dimensional Euler--Maxwell system \eqref{EM}, with the normal structure \eqref{structure:2dim}, satisfying the energy inequality \eqref{energy-inequa} and enjoying the additional regularity
	\begin{equation}\label{propagation:damping:3}
		\begin{gathered}
			u\in L^\infty(\mathbb{R}^+;\dot H^1\cap \dot W^{1,p}),
			\quad
			(E,B)\in L^\infty(\mathbb{R}^+;\dot H^1),
			\quad
			c^{1-s}(E,B)\in \widetilde L^\infty(\mathbb{R}^+;\dot B^s_{2,n}),
			\\
			(cE,B)\in L^2(\mathbb{R}^+;\dot H^1),
			\quad
			B\in L^2(\mathbb{R}^+;\dot B^2_{2,1,<}),
			\\
			c^{\frac 74-s}(E,B)\in \widetilde L^2(\mathbb{R}^+;\dot B^{s-\frac 34}_{\infty,n,>}),
			\quad
			c^{2-s} E\in \widetilde L^2(\mathbb{R}^+;\dot B^s_{2,n}),
			\quad
			c^{2-s} B\in \widetilde L^2(\mathbb{R}^+;\dot B^s_{2,n,>}).
		\end{gathered}
	\end{equation}
	It is to be emphasized that the bounds in \eqref{propagation:damping:3} are uniform in any set of initial data such that the left-hand side of \eqref{initial:3} remains bounded.
\end{thm}

\begin{rem}
	As in the case of Theorem \ref{main:1}, for any fixed initial data $(u_0,E_0,B_0)$ satisfying the requirements of Theorem \ref{main:3}, it is possible to improve the uniform controls \eqref{propagation:damping:3} by showing that the bound
	\begin{equation*}
		(E,B)\in \widetilde L^\infty(\mathbb{R}^+;\dot B^s_{2,n})
	\end{equation*}
	holds uniformly as $c\to\infty$. This is clarified in Section \ref{fixed:data:0}, below.
\end{rem}

The remainder of this section builds up to the proofs of Theorems \ref{main:1}, \ref{main:2} and \ref{main:3}, by implementing the strategy discussed in Section \ref{strategy:0}. The proofs of the theorems \emph{per se} are given in Sections \ref{section:main:1}, \ref{section:main:2} and \ref{section:main:3}, respectively.

\subsection{Dimensional analysis}
\label{section:dimensional}

Prior to discussing specific elements of the proofs of the above theorems, we provide here a dimensional analysis of the Euler--Maxwell system \eqref{EM}, which, we hope, will shed light on the initial conditions \eqref{initial:1} and \eqref{initial:3}.

Specifically, assuming that $(u,E,B)$ is a solution of \eqref{EM}, for some fixed light velocity $c>0$ and initial data $(u_0,E_0,B_0)$, we observe, defining
\begin{equation}\label{parabolic:scaling}
	u^\lambda(t,x)=\lambda u(\lambda^2 t,\lambda x),
	\qquad
	E^\lambda(t,x)=\lambda E(\lambda^2 t,\lambda x),
	\qquad
	B^\lambda(t,x)=\lambda B(\lambda^2 t,\lambda x),
\end{equation}
for any $\lambda>0$, that $(u^\lambda,E^\lambda,B^\lambda)$ also solves \eqref{EM} with a rescaled speed of light $c_\lambda=\lambda c$ (the electrical conductivity $\sigma$ remains unchanged) and for the initial data
\begin{equation*}
	u^\lambda_0(x)=\lambda u(\lambda x),
	\qquad
	E^\lambda_0(x)=\lambda E(\lambda x),
	\qquad
	B^\lambda_0(x)=\lambda B(\lambda x).
\end{equation*}
In particular, we readily compute that
\begin{equation*}
	\begin{aligned}
		&\norm {(u_0^\lambda,E_0^\lambda,B_0^\lambda)}_{L^2}+\norm{u_0^\lambda}_{\dot W^{1,p}}
		+\norm{(u_0^\lambda,E_0^\lambda,B_0^\lambda)}_{\dot H^1}
		+c_\lambda^{-\frac 34}\norm{(E_0^\lambda,B_0^\lambda)}_{\dot B^\frac 74_{2,1}}
		\\
		&\quad=
		\lambda
		\left(\lambda^{-1}\norm {(u_0,E_0,B_0)}_{L^2}
		+\lambda^{1-\frac 2p}\norm{u_0}_{\dot W^{1,p}}
		+\norm{(u_0,E_0,B_0)}_{\dot H^1}
		+c^{-\frac 34}\norm{(E_0,B_0)}_{\dot B^\frac 74_{2,1}}\right),
	\end{aligned}
\end{equation*}
which, by an optimization procedure in $\lambda$, implies that Theorem \ref{main:1} still holds if one replaces assumption \eqref{initial:1} with the weaker inequality
\begin{equation}\label{initial:5}
	\left(\mathcal{E}_0^{\frac{p-2}{2p-2}}\norm{u_0}_{\dot W^{1,p}}^{\frac{p}{2p-2}}
	+\norm{(u_0,E_0,B_0)}_{\dot H^1}
	+c^{-\frac 34}\norm{(E_0,B_0)}_{\dot B^\frac 74_{2,1}}\right)
	C_*e^{C_*\mathcal{E}_0^{4+\varepsilon}}< c,
\end{equation}
where the independent constant $C_*>0$ may take a different value.
Note that this inequality is now invariant with respect to the parabolic scaling \eqref{parabolic:scaling}. Similarly, the same procedure can be used to optimize \eqref{initial:3} and replace it with a scaling invariant assumption.

It turns out that the parabolic scaling \eqref{parabolic:scaling} is the only available invariant dilation which leaves the electrical conductivity $\sigma$ unchanged. However, if one allows $\sigma$ to be redefined according to the dilation, then other scalings become available. For example, introducing the hyperbolic scaling
\begin{equation}\label{hyperbolic:scaling}
	u^\lambda(t,x)= u(\lambda t,\lambda x),
	\qquad
	E^\lambda(t,x)= E(\lambda t,\lambda x),
	\qquad
	B^\lambda(t,x)= B(\lambda t,\lambda x),
\end{equation}
for any $\lambda>0$, we see that $(u^\lambda,E^\lambda,B^\lambda)$ now solves \eqref{EM} with a rescaled electrical conductivity $\sigma_\lambda=\lambda\sigma$ (the speed of light $c$ remains unchanged) and for the initial data
\begin{equation*}
	u^\lambda_0(x)= u(\lambda x),
	\qquad
	E^\lambda_0(x)= E(\lambda x),
	\qquad
	B^\lambda_0(x)= B(\lambda x).
\end{equation*}
In particular, by setting $\lambda=\sigma^{-1}$, it is now possible to deduce how the constants $C_*$ and $C_{**}$, in \eqref{initial:1} and \eqref{initial:3}, respectively, depend on $\sigma$. More precisely, this process allows us to show that Theorem \ref{main:1} holds if one further replaces \eqref{initial:5} by the assumption that
\begin{equation}\label{initial:6}
	\left(\mathcal{E}_0^{\frac{p-2}{2p-2}}\norm{u_0}_{\dot W^{1,p}}^{\frac{p}{2p-2}}
	+\norm{(u_0,E_0,B_0)}_{\dot H^1}
	+(\sigma c)^{-\frac 34}\norm{(E_0,B_0)}_{\dot B^\frac 74_{2,1}}\right)
	C_*e^{C_*\sigma^{4+\varepsilon}\mathcal{E}_0^{4+\varepsilon}}< c,
\end{equation}
for some constant $C_*>0$ which is now independent of the electrical conductivity $\sigma$. Observe that \eqref{initial:6} is now invariant with respect to both the parabolic scaling \eqref{parabolic:scaling} and the hyperbolic scaling \eqref{hyperbolic:scaling}. As previously noted, the same process can be applied to improve \eqref{initial:3}.

\subsection{Approximation procedure and stability}\label{approximation:0}

The proofs of Theorems \ref{main:1}, \ref{main:2} and \ref{main:3} proceed by compactness arguments. More specifically, they follow the standard procedure of first considering smooth solutions to a regularized approximation of the original system \eqref{EM}, where all formal estimates can be conducted with full rigor, and then showing the stability of the approximation as it converges towards the original system.

Such approximation procedures are absolutely classical in the field of fluid dynamics. We are therefore only going to outline an example of approximation which can be used here to conveniently establish our results. Specifically, for any integer $n\geq 1$, we consider the unique solution $(u_n,E_n,B_n)$ to the approximate Navier--Stokes--Maxwell system
\begin{equation}\label{NSM}
		\begin{cases}
			\begin{aligned}
				&\partial_t u_n +(S_nu_n) \cdot\nabla u_n - \frac{1}{n} \Delta u_n = - \nabla p_n + (S_nj_n) \times B_n, &\div u_n =0,&
				\\
				&\frac{1}{c} \partial_t E_n - \nabla \times B_n = - j_n , &\div E_n = 0,&
				\\
				&\frac{1}{c} \partial_t B_n + \nabla \times E_n  = 0 , &\div B_n =0,&
				\\
				&j_n= \sigma \big(cE_n + S_nP(u_n \times B_n)\big), &\div j_n =0,&
			\end{aligned}
		\end{cases}
	\end{equation}
	for the initial data $(u_n,E_n,B_n)_{|t=0}= S_n(u_{0},E_0,B_0)$, with the two-dimensional normal structure \eqref{structure:2dim}, where $S_n$ denotes the Fourier multiplier operator defined in Appendix \ref{besov:1} which restricts frequencies to the domain $\{|\xi|\leq 2^n\}$. The construction of the solution $(u_n,E_n,B_n)$ is a standard procedure. One can, for instance, follow and adapt the steps detailed in \cite[Section 12.2]{l16}.

Note that other approximation schemes can be employed. In particular, the dissipation term $-\frac 1n\Delta u_n$ is not essential. However, as a matter of convenience, the use of this term allows us to comfortably construct the approximate solution $(u_n,E_n,B_n)$ by relying on methods from the analysis of the incompressible Navier--Stokes system.

Observe that the corresponding energy inequality
\begin{equation*}
	\begin{aligned}
		\frac 12\left(\norm {u_n(t)}_{L^2}^2 + \norm {E_n(t)}_{L^2}^2 + \norm {B_n(t)}_{L^2}^2\right)
		\hspace{-50mm}&
		\\
		&+\int_0^t \left(\frac 1n\norm {\nabla u_n(\tau)}_{L^2}^2
		+\frac{1}{\sigma}\norm {j_n(\tau)}_{L^2}^2\right) d\tau
		\leq
		\frac 12 \norm {S_n(u_0,E_0,B_0)}_{L^2}^2\leq
		\frac 12 \norm {(u_0,E_0,B_0)}_{L^2}^2,
	\end{aligned}
\end{equation*}
for all $t\geq 0$, is now fully justified and, since the initial data is smooth, it is possible to show that $(u_n,E_n,B_n)$ remains smooth for all times, albeit not uniformly in $n$.

The above energy inequality only allows us to deduce the uniform bounds
\begin{equation*}
	(u_n,E_n,B_n)\in L^\infty_t L^2_x
	\qquad\text{and}\qquad
	j_n\in L^2_{t,x},
\end{equation*}
which are insufficient to establish the stability of the nonlinear terms
\begin{equation*}
	(S_nu_n) \cdot\nabla u_n,
	\qquad
	(S_nj_n) \times B_n
	\qquad\text{and}\qquad
	S_nP(u_n \times B_n)
\end{equation*}
in the limit $n\to\infty$. The general strategy is therefore to show that the bounds and properties stated in Theorems \ref{main:1}, \ref{main:2} and \ref{main:3} can be fully justified on the smooth system \eqref{NSM} uniformly in $n$. Such uniform bounds are then sufficient to show the strong relative compactness of $\big\{(u_n,E_n,B_n)\big\}_{n=1}^\infty$ in $L^2_{t,x,\mathrm{loc}}$, which then allows us to take the limit $n\to \infty$ (up to extraction of subsequences) and establish the asymptotic stability of \eqref{NSM}, thereby yielding suitable solutions of the original Euler--Maxwell system \eqref{EM}.

In what follows, for the sake of simplicity, keeping in mind that all computations can be fully justified on the approximate system \eqref{NSM}, we shall perform all estimates formally on the original system \eqref{EM}.
In particular, we emphasize that, even though, strictly speaking, the dissipation operator $-\frac 1n\Delta$ cannot be ignored, it is self-adjoint and therefore will not impact the energy estimates which are performed in the proofs, below.

\subsection{Paradifferential calculus and the normal structure}\label{section:paradifferential}

The use of Besov and Chemin--Lerner spaces in the analysis of nonlinear systems often requires a careful use of product estimates. Such results from paradifferential calculus are rather standard but their applicability is limited. In the following paradifferential lemma, we show how the normal structure \eqref{structure:2dim} can be exploited to extend the range of applicability of classical product estimates. This plays a central role in our analysis of \eqref{EM}.

\begin{lem}\label{paradifferential:1}
	Let $F,G:\mathbb{R}_t\times\mathbb{R}_x^2\to\mathbb{R}^3$ be solenoidal vector fields with the normal structure
	\begin{equation}\label{normal:1}
		F(t,x_1,x_2)=
		\begin{pmatrix}
			F_1(t,x_1,x_2)\\F_2(t,x_1,x_2)\\0
		\end{pmatrix}
		\qquad\text{and}\qquad
		G(t,x_1,x_2)=
		\begin{pmatrix}
			0\\0\\G_3(t,x_1,x_2)
		\end{pmatrix}.
	\end{equation}
	Further consider integrability parameters in $[1,\infty]$ such that
	\begin{equation*}
		\frac 1a=\frac 1{a_1}+\frac 1{a_2}
		\qquad\text{and}\qquad
		\frac 1c=\frac 1{c_1}+\frac 1{c_2}.
	\end{equation*}
	Then, recalling that $P=(-\Delta)^{-1}\curl\curl$ denotes Leray's projector onto solenoidal vector fields, one has the product estimate
	\begin{equation}\label{para-product:3}
		\norm{P(F\times G)}_{\widetilde L^a_t \dot B^{s+t-1}_{2,c}(\mathbb{R}^2)}
		\lesssim
		\norm{F}_{\widetilde L^{a_1}_t\dot B^s_{2,c_1}(\mathbb{R}^2)}
		\norm{G}_{\widetilde L^{a_2}_t\dot B^t_{2,c_2}(\mathbb{R}^2)},
	\end{equation}
	for any $s\in(-\infty,1)$ and $t\in (-\infty,2)$ with $s+t>0$. Furthermore, in the endpoint case $s=1$, one has that
	\begin{equation}\label{para-product:2}
		\norm{P(F\times G)}_{\widetilde L^a_t \dot B^{t}_{2,c}(\mathbb{R}^2)}
		\lesssim
		\norm{F}_{L^{a_1}_t L^\infty_x(\mathbb{R}^2) \cap \widetilde L^{a_1}_t\dot B^1_{2,\infty}(\mathbb{R}^2)}
		\norm{G}_{\widetilde L^{a_2}_t\dot B^t_{2,c}(\mathbb{R}^2)},
	\end{equation}
	for any $t\in(-1,2)$.
\end{lem}

\begin{rem}
	The significance of the preceding lemma lies in the fact that it allows us to consider parameters in the range $t\in[1,2)$. Without the normal structure \eqref{normal:1}, we would be restricted to values $t<1$.
\end{rem}

\begin{rem}
	A straightforward simplification of the proof below yields a corresponding result in Besov spaces for vector fields independent of the time variable $t$.
\end{rem}

\begin{proof}
	We are going to use the paradifferential decomposition
	\begin{equation*}
		F\times G=T_FG-T_GF+R(F,G),
	\end{equation*}
	where the paraproducts are defined by
	\begin{equation*}
		T_FG=\sum_{j\in\mathbb{Z}}S_{j-2}F\times \Delta_j G,
		\qquad
		T_GF=\sum_{j\in\mathbb{Z}}S_{j-2} G\times \Delta_j F=- \sum_{j\in\mathbb{Z}}\Delta_j F \times S_{j-2} G ,
	\end{equation*}
	and the remainder is given by
	\begin{equation*}
		R(F,G)=\sum_{\substack{j,k\in\mathbb{Z}\\|j-k|\leq 2}}\Delta_jF\times\Delta_kG.
	\end{equation*}
	In particular, by virtue of the solenoidal and normal structures of $F$ and $G$, one has that
	\begin{equation*}
		\nabla\times(F\times G)=F\times(\nabla\times G),
	\end{equation*}
	which leads to the identity
	\begin{equation*}
		\nabla\times(F\times G)=\nabla\times T_FG-T_{\nabla\times G}F+\nabla\times R(F,G).
	\end{equation*}
	We therefore conclude, by standard embeddings of Besov spaces, that
	\begin{equation*}
		\begin{aligned}
			\norm{P(F\times G)}_{\widetilde L^a_t \dot B^{s+t-1}_{2,c}}
			&\lesssim
			\norm{\nabla\times(F\times G)}_{\widetilde L^a_t \dot B^{s+t-2}_{2,c}}
			\\
			&\lesssim
			\norm{T_FG}_{\widetilde L^a_t \dot B^{s+t-1}_{2,c}}
			+\norm{T_{\nabla\times G}F}_{\widetilde L^a_t \dot B^{s+t-2}_{2,c}}
			+\norm{R(F,G)}_{\widetilde L^a_t \dot B^{s+t}_{1,c}},
		\end{aligned}
	\end{equation*}
	for every $s,t\in\mathbb{R}$.
	
	Next, employing the classical paradifferential estimates \eqref{para-product:4} and \eqref{para-product:5} presented in the appendix, and further exploiting standard embeddings of Besov spaces, we find that
	\begin{equation*}
		\begin{aligned}
			\norm{P(F\times G)}_{\widetilde L^a_t \dot B^{s+t-1}_{2,c}}
			&\lesssim
			\norm{F}_{\widetilde L^{a_1}_t\dot B^{s-1}_{\infty,c_1}}
			\norm{G}_{\widetilde L^{a_2}_t\dot B^t_{2,c_2}}
			+\norm{F}_{\widetilde L^{a_1}_t\dot B^s_{2,c_1}}
			\norm{\nabla\times G}_{\widetilde L^{a_2}_t\dot B^{t-2}_{\infty,c_2}}
			\\
			&\quad+\norm{F}_{\widetilde L^{a_1}_t\dot B^s_{2,c_1}}
			\norm{G}_{\widetilde L^{a_2}_t\dot B^t_{2,c_2}}
			\\
			&\lesssim
			\norm{F}_{\widetilde L^{a_1}_t\dot B^s_{2,c_1}}
			\norm{G}_{\widetilde L^{a_2}_t\dot B^t_{2,c_2}},
		\end{aligned}
	\end{equation*}
	for any $s<1$ and $t<2$, such that $s+t>0$. This establishes \eqref{para-product:3}.
	
	As for the endpoint case $s=1$, if, in addition to \eqref{para-product:4}, one also uses \eqref{para-product:6}, then, similar bounds lead to
	\begin{equation*}
		\begin{aligned}
			\norm{P(F\times G)}_{\widetilde L^a_t \dot B^{t}_{2,c}}
			&\lesssim
			\norm{F}_{L^{a_1}_t L^\infty_x}
			\norm{G}_{\widetilde L^{a_2}_t\dot B^t_{2,c}}
			+\norm{F}_{\widetilde L^{a_1}_t\dot B^1_{2,\infty}}
			\norm{\nabla\times G}_{\widetilde L^{a_2}_t\dot B^{t-2}_{\infty,c}}
			\\
			&\quad+\norm{F}_{\widetilde L^{a_1}_t\dot B^1_{2,\infty}}
			\norm{G}_{\widetilde L^{a_2}_t\dot B^t_{2,c}}
			\\
			&\lesssim
			\norm{F}_{L^{a_1}_t L^\infty_x \cap \widetilde L^{a_1}_t\dot B^1_{2,\infty}}
			\norm{G}_{\widetilde L^{a_2}_t\dot B^t_{2,c}},
		\end{aligned}
	\end{equation*}
	for any $-1<t<2$, thereby establishing \eqref{para-product:2} and concluding the proof of the lemma.
\end{proof}

The following \emph{ad hoc} variant of paraproduct estimate will be useful when handling vorticities which are bounded pointwise.

\begin{lem}\label{paradifferential:2}
	Let $F,G:\mathbb{R}_x^2\to\mathbb{R}^3$ be solenoidal vector fields with the normal structure \eqref{normal:1}.
	Then, one has the product estimate
	\begin{equation}\label{para-product:7}
		\norm{P(F\times G)}_{\dot B^{1}_{2,1}(\mathbb{R}^2)}
		\lesssim
		\norm{F}_{L^2(\mathbb{R}^2)}
		\norm{G}_{\dot B^1_{\infty,1}(\mathbb{R}^2)}
		+\norm{F}_{\dot H^1(\mathbb{R}^2)}
		\norm{G}_{\dot H^1(\mathbb{R}^2)}.
	\end{equation}
\end{lem}

\begin{rem}
	The above statement is phrased in terms of mere Besov spaces. As usual, a straightforward extension of the same result to Chemin--Lerner spaces also exists.
\end{rem}

\begin{proof}
	We employ the method of proof of Lemma \ref{paradifferential:1}. In particular, we obtain that
	\begin{equation*}
		\norm{P(F\times G)}_{\dot B^1_{2,1}}\lesssim
		\norm{T_FG}_{\dot B^1_{2,1}}+\norm{T_{\nabla\times G}F}_{\dot B^0_{2,1}}+\norm{R(F,G)}_{\dot B^1_{2,1}}.
	\end{equation*}
	Then, by virtue of the classical paradifferential estimates \eqref{para-product:4}, \eqref{para-product:5} and \eqref{para-product:6} , we arrive at
	\begin{equation*}
		\begin{aligned}
			\norm{P(F\times G)}_{\dot B^1_{2,1}}
			&\lesssim
			\norm{F}_{L^2}\norm{G}_{\dot B^1_{\infty,1}}+\norm{\nabla\times G}_{\dot B^{-1}_{\infty,2}}\norm{F}_{\dot B^1_{2,2}}
			+\norm{F}_{\dot B^0_{2,\infty}}\norm{G}_{\dot B^1_{\infty,1}}
			\\
			&\lesssim
			\norm{F}_{L^2}\norm{G}_{\dot B^1_{\infty,1}}+\norm{G}_{\dot B^{0}_{\infty,2}}\norm{F}_{\dot B^1_{2,2}}.
		\end{aligned}
	\end{equation*}
	Finally, an application of the two-dimensional embedding $\dot B^{1}_{2,2}\subset \dot B^{0}_{\infty,2}$ concludes the proof.
\end{proof}

\subsection{Controlling the vorticity}\label{section:vorticity}

In order to carry out our strategy, previously laid out in Section \ref{strategy:0}, we need to control the vorticity $\omega$ in $L^p_x$, with $p\geq 2$, by exploiting Yudovich's approach of the two-dimensional incompressible Euler equations \eqref{Euler:1}. The following basic lemma provides us with a simple tool to do so.

\begin{lem}\label{vorticity:0}
	Let $(u,E,B)\in C^1\big([0,T)\times\mathbb{R}^2\big)\cap L^\infty\big([0,T);H^1(\mathbb{R}^2)\big)$ be a smooth solution to \eqref{EM}, for some $T>0$, with the two-dimensional normal structure \eqref{structure:2dim}. Then, for all $t\in (0,T)$, it holds that
	\begin{equation*}
		\begin{aligned}
			\norm{\omega(t)}_{L^2_x}
			&\leq \norm{\omega(0)}_{L^2_x}
			+\norm{j}_{L^2([0,t);L^2_x)}\norm{\nabla B}_{L^2([0,t);L^\infty_x)},
			\\
			\norm{\omega(t)}_{L^p_x}
			&\lesssim \norm{\omega(0)}_{L^p_x}
			+\norm{j}_{L^2([0,t);\dot B^0_{2,\infty})}^\frac{2}{p}
			\norm{j}_{L^2([0,t);\dot B^1_{2,\infty})}^{1-\frac 2p}
			\norm{\nabla B}_{L^2([0,t);L^\infty_x)},
			\\
			\norm{\omega(t)}_{L^\infty_x}
			&\lesssim \norm{\omega(0)}_{L^\infty_x}
			+\norm{j}_{L^2([0,t);L^\infty_x)}
			\norm{\nabla B}_{L^2([0,t);L^\infty_x)}
			\\
			&\lesssim \norm{\omega(0)}_{L^\infty_x}
			+\norm{j}_{L^2([0,t);\dot B^1_{2,1})}
			\norm{\nabla B}_{L^2([0,t);L^\infty_x)},
		\end{aligned}
	\end{equation*}
	for any $2< p<\infty$.
\end{lem}

\begin{proof}
	By slight abuse of language, we assume here that the vorticity $\omega$ is defined as the scalar function $\omega=\partial_1u_2-\partial_2u_1$ and that
	\begin{equation*}
		u=
		\begin{pmatrix}
			u_1\\u_2
		\end{pmatrix}
		\qquad\text{and}\qquad
		j=
		\begin{pmatrix}
			j_1\\j_2
		\end{pmatrix}.
	\end{equation*}
	In particular, a straightforward computation shows that the transport equation \eqref{transport:1} can then be recast as
	\begin{equation}\label{transport:2}
		\partial_t \omega +u \cdot\nabla \omega = -j \cdot\nabla b,
	\end{equation}
	where $b(t,x)$ is the third component of the magnetic field $B(t,x)$ as denoted in \eqref{structure:2dim}.
	
	Next, supposing, for simplicity, that $p$ is finite and introducing the test function $\varphi(x)=e^{-|x|^2}$, we multiply the above vorticity transport equation by $p\omega|\omega|^{p-2}\varphi(\varepsilon x)$, where $0<\varepsilon<1$, and then integrate in space. This yields, since $u$ is divergence-free, that
	\begin{equation*}
		\begin{aligned}
			\frac{d}{dt}\int_{\mathbb{R}^2}|\omega|^p\varphi(\varepsilon x)dx
			&\leq p\int_{\mathbb{R}^2}|j||\nabla b||\omega|^{p-1}\varphi(\varepsilon x)dx
			+\varepsilon\int_{\mathbb{R}^2}|\omega|^pu\cdot\nabla\varphi(\varepsilon x)dx
			\\
			&\leq p\norm{j}_{L^p_x}\norm{\nabla b}_{L^\infty_x}
			\left(\int_{\mathbb{R}^2}|\omega|^p\varphi(\varepsilon x)dx\right)^{\frac{p-1}p}
			+\varepsilon\norm{\omega}_{L^\infty_x}^{p-1}\norm{\omega}_{L^2_x}\norm{u}_{L^2_x}\norm{\nabla\varphi}_{L^\infty_x},
		\end{aligned}
	\end{equation*}
	whereby, if $\omega$ is nontrivial,
	\begin{equation*}
		\frac{d}{dt}\left(\int_{\mathbb{R}^2}|\omega|^p\varphi(\varepsilon x)dx\right)^\frac 1p
		\leq \norm{j}_{L^p_x}\norm{\nabla b}_{L^\infty_x}
		+\frac{\varepsilon\norm{\omega}_{L^\infty_x}^{p-1}\norm{\omega}_{L^2_x}\norm{u}_{L^2_x}\norm{\nabla\varphi}_{L^\infty_x}}
		{p\left(\int_{\mathbb{R}^2}|\omega|^p\varphi(\varepsilon x)dx\right)^{\frac{p-1}p}}.
	\end{equation*}
	Finally, further integrating in time and letting $\varepsilon$ tend to zero, we conclude that
	\begin{equation*}
		\norm{\omega(t)}_{L^p_x}
		\leq \norm{\omega(0)}_{L^p_x}
		+\norm{j}_{L^2([0,t);L^p_x)}\norm{\nabla b}_{L^2([0,t);L^\infty_x)},
	\end{equation*}
	for any $t\in (0,T)$. A classical modification of this argument gives the same result for an infinite value of the parameter $p$.
	
	Now, if $p=2$, the proof is finished. If $2<p<\infty$, we further employ the following convexity inequality (see \cite{bl76} for details on interpolation theory)
	\begin{equation*}
		\norm{j}_{L^2([0,t);\dot H^{1-\frac 2p})}
		\lesssim
		\norm{j}_{L^2([0,t);\dot B^0_{2,\infty})}^\frac{2}{p}
		\norm{j}_{L^2([0,t);\dot B^1_{2,\infty})}^{1-\frac 2p},
	\end{equation*}
	in combination with the classical two-dimensional Sobolev embedding $\dot H^{1-\frac 2p}\subset L^p$, to conclude that
	\begin{equation*}
		\norm{\omega(t)}_{L^p_x}
		\lesssim \norm{\omega(0)}_{L^p_x}
		+\norm{j}_{L^2([0,t);\dot B^0_{2,\infty})}^\frac{2}{p}
		\norm{j}_{L^2([0,t);\dot B^1_{2,\infty})}^{1-\frac 2p}
		\norm{\nabla b}_{L^2([0,t);L^\infty_x)}.
	\end{equation*}
	Finally, the case $p=\infty$ is settled with an application of the continuous embeddings $\dot B^1_{2,1}\subset \dot B^0_{\infty,1}\subset L^\infty$, valid in two dimensions, thereby completing the proof of the lemma.
\end{proof}

\begin{rem}
	Recall that the approximation procedure presented in the Section \ref{approximation:0} relies on a viscous approximation of the Euler system. It is therefore important to emphasize here that the method of proof of the preceding lemma also applies to viscous approximations of the transport equation \eqref{transport:2}. Indeed, observe that multiplying the transport-diffusion equation
	\begin{equation*}
		\partial_t \omega +u \cdot\nabla \omega-\Delta\omega = -j \cdot\nabla b
	\end{equation*}
	by $\beta'(\omega)$, for some convex nonnegative renormalization $\beta\in C^2(\mathbb{R})$ with $\beta(0)=0$, and then integrating in space leads, at least formally, to the estimate
	\begin{equation*}
		\frac{d}{dt}\int_{\mathbb{R}^2}\beta(\omega)dx\leq
		\frac{d}{dt}\int_{\mathbb{R}^2}\beta(\omega)dx+\int_{\mathbb{R}^2}|\nabla\omega|^2\beta''(\omega)dx
		=\int_{\mathbb{R}^2}(j\cdot\nabla b)\beta'(\omega)dx.
	\end{equation*}
	This observation can be exploited to derive the estimates stated in Lemma \ref{vorticity:0} in spite of the diffusion term.
	Thus, we conclude that the approximated system \eqref{NSM} is well-suited for an application of the estimates from Lemma \ref{vorticity:0}.
\end{rem}

Recall now that the space $\dot H^1(\mathbb{R}^2)$ barely fails to embed into $L^\infty(\mathbb{R}^2)$. Thus, the bounds produced in the preceding lemma will be very useful to recover, by interpolation, an $L^\infty$-bound on the velocity field $u$, as explained in the following simple result.

\begin{lem}\label{vorticity:1}
	Let $(u,E,B)$ be a smooth solution to \eqref{EM}, with the two-dimensional normal structure \eqref{structure:2dim}. Then, for all $T>0$, it holds that
	\begin{equation*}
		\norm{u(T)}_{L^\infty_x}
		\lesssim
		\norm{u(T)}_{L^2_x}^{\frac{p-2}{2(p-1)}}
		\left(\norm{u(0)}_{\dot W^{1,p}_x}
		+\norm{j}_{L^2_t\dot B^0_{2,\infty}}^\frac{2}{p}
		\norm{j}_{L^2_t\dot B^1_{2,\infty}}^{1-\frac 2p}
		\norm{\nabla B}_{L^2_tL^\infty_x}\right)^{\frac p{2(p-1)}},
	\end{equation*}
	for any $2< p<\infty$, where all time-norms are taken over the interval $[0,T)$.
\end{lem}

\begin{proof}
	This result follows from the classical Gagliardo--Nirenberg interpolation inequality, in two space-dimensions. For the sake of convenience, we provide a brief justification of the precise inequality which is employed here.
	
	Specifically, for any $k\in\mathbb{Z}$ and $2^k\leq R<2^{k+1}$, we estimate that
	\begin{equation*}
		\begin{aligned}
			\norm{u}_{L^\infty}&\leq \norm{S_ku}_{L^\infty}+\sum_{i=k}^\infty\norm{\Delta_i u}_{L^\infty}
			\lesssim 2^k\norm{u}_{L^2}+\sum_{i=k}^\infty 2^{i(\frac 2p-1)}\norm{\nabla u}_{L^p}
			\\
			&\lesssim R\norm{u}_{L^2}+R^{-(1-\frac 2p)}\norm{\nabla u}_{L^p},
		\end{aligned}
	\end{equation*}
	which yields, upon optimization of the interpolation parameter value $R>0$, the Gagliardo--Nirenberg inequality
	\begin{equation}\label{convexity:1}
		\norm{u}_{L^\infty(\mathbb{R}^2)}\lesssim
		\norm{u}_{L^2(\mathbb{R}^2)}^{\frac{p-2}{2(p-1)}}\norm{\nabla u}_{L^p(\mathbb{R}^2)}^{\frac{p}{2(p-1)}},
	\end{equation}
	for any $2<p\leq \infty$.
	
	Then, combining this convexity inequality with the estimates from Lemma \ref{vorticity:0} and recalling the equivalence $\norm{\nabla u}_{L^p}\sim \norm{\omega}_{L^p}$, because $u$ is divergence-free, concludes the proof of the lemma.
\end{proof}

\begin{rem}
	Note that the proof of \eqref{convexity:1} can be adapted to establish, for any divergence-free field $u$, that
	\begin{equation}\label{convexity:2}
		\norm{u}_{L^\infty(\mathbb{R}^2)}\lesssim
		\norm{u}_{L^2(\mathbb{R}^2)}^{\frac{p-2}{2(p-1)}}\norm{\nabla \times u}_{L^p(\mathbb{R}^2)}^{\frac{p}{2(p-1)}},
	\end{equation}
	for all values $2<p\leq \infty$. Indeed, this follows from the observation that
	\begin{equation*}
		\norm{\Delta_i u}_{L^\infty}\lesssim 2^{-i}\norm{\Delta_i \nabla\times u}_{L^\infty}\lesssim 2^{i(\frac 2p-1)}\norm{\nabla\times u}_{L^p},
	\end{equation*}
	provided $\div u=0$.
\end{rem}

\subsection{Control of high-frequency damped electromagnetic waves}\label{section:high:freq}

The following result follows from a simple but careful combination of the damped Strichartz estimates for high electromagnetic frequencies, established in Section \ref{section:damped:strichartz:1}, with the paradifferential product estimates from Lemma \ref{paradifferential:1}.

\begin{lem}\label{high:freq:estimates}
	Let $d=2$. Assume that $(E,B)$ is a smooth solution to \eqref{Maxwell:system:*}, for some initial data $(E_0,B_0)$ and some divergence-free vector field $u$, with the normal structure \eqref{structure:2dim}.
	
	Then, for any exponents $1\leq p\leq q\leq\infty$, $2\leq r\leq\infty$ and $1\leq n\leq\infty$ which are admissible in the sense that
	\begin{equation}\label{admissible:0}
		\frac 2q+\frac {1} r\geq \frac {1} 2,
	\end{equation}
	one has the high-frequency estimate, over any time interval $[0,T)$, for any $0<\alpha<1$ and $s<2$, with $\alpha+s>0$,
	\begin{equation*}
		\begin{aligned}
			\norm{(E,B)}_{\widetilde L^q_t\dot B^{s+\alpha-\frac 74+\frac 3{2r}}_{r,n,>}}
			&\lesssim
			c^{\frac {1}2\left(\frac 12-\frac 1r\right)-\frac 2q}\norm{(E_0,B_0)}_{\dot B^{s+\alpha-1}_{2,n,>}}
			\\
			&\quad +c^{\frac {1}2\left(\frac 12-\frac 1r\right)+2\left(\frac 1p-\frac 1q\right)-1}
			\norm{u}_{L_t^{\infty}\dot B^0_{2,\infty}}^{1-\alpha}
			\norm{u}_{L_t^{\infty}\dot B^1_{2,\infty}}^\alpha
			\norm{B}_{\widetilde L_t^{p}\dot B^s_{2,\infty}}.
		\end{aligned}
	\end{equation*}
	In the endpoint cases $\alpha=1$ and $\alpha=0$, one also has the respective estimates
	\begin{equation}\label{frequencies:4}
		\begin{aligned}
			\norm{(E,B)}_{\widetilde L^q_t\dot B^{s-\frac {3}2\left(\frac 12-\frac 1r\right)}_{r,n,>}}
			&\lesssim
			c^{\frac {1}2\left(\frac 12-\frac 1r\right)-\frac 2q}
			\norm{(E_0,B_0)}_{\dot B^s_{2,n,>}}
			\\
			&\quad+c^{\frac {1}2\left(\frac 12-\frac 1r\right)+2\left(\frac 1p-\frac 1q\right)-1}
			\norm{u}_{L^\infty_{t,x} \cap L^\infty_t\dot B^1_{2,\infty}}\norm{B}_{\widetilde L_t^{p}\dot B^s_{2,n}},
		\end{aligned}
	\end{equation}
	for any $-1<s<2$, and
	\begin{equation*}
		\begin{aligned}
			\norm{(E,B)}_{\widetilde L^q_t\dot B^{s-\frac 74+\frac 3{2r}}_{r,n,>}}
			&\lesssim
			c^{\frac {1}2\left(\frac 12-\frac 1r\right)-\frac 2q}\norm{(E_0,B_0)}_{\dot B^{s-1}_{2,n,>}}
			\\
			&\quad +c^{\frac {1}2\left(\frac 12-\frac 1r\right)+2\left(\frac 1p-\frac 1q\right)-1}
			\norm{u}_{L_t^{\infty}\dot B^0_{2,\infty}}\norm{B}_{\widetilde L_t^{p}\dot B^s_{2,n}},
		\end{aligned}
	\end{equation*}
	for any $0<s<2$.
\end{lem}

\begin{proof}
	Considering Maxwell's system \eqref{Maxwell:system:*} and applying Corollary \ref{cor:maxwell}, with $\tilde r=2$ and $\tilde q=p'$, yields the high-frequency estimate
	\begin{equation*}
		2^{-j\frac {3}2\left(\frac 12-\frac 1r\right)}\norm{\Delta_j (E,B)}_{L^q_tL^r_x}
		\lesssim
		c^{\frac {1}2\left(\frac 12-\frac 1r\right)-\frac 2q}
		\left(
		\norm{\Delta_j (E_0,B_0)}_{L^2_x}
		+c^{\frac 2p-1}
		\norm{\Delta_j P(u\times B)}_{L_t^{p}L^{2}_x}
		\right),
	\end{equation*}
	for all $j\in\mathbb{Z}$ with $2^j\geq \sigma c$, where $1\leq p\leq q\leq \infty$ and $2\leq r\leq\infty$ must satisfy \eqref{admissible:0}.
	It is to be emphasized that, thanks to the damping phenomenon in \eqref{Maxwell:system:*}, all estimates here hold uniformly over any time interval $[0,T)$, where $T=\infty$ is allowed.
	
	Next, summing the preceding estimate in $j$ and utilizing the paradifferential product law \eqref{para-product:2}, we deduce that
	\begin{equation*}
		\begin{aligned}
			\norm{\mathds{1}_{\{2^j\geq\sigma c\}}2^{j\left(s-\frac {3}2\left(\frac 12-\frac 1r\right)\right)}\norm{\Delta_j (E,B)}_{L^q_tL^r_x}}_{\ell^n}
			\hspace{-45mm}&
			\\
			&\lesssim
			c^{\frac {1}2\left(\frac 12-\frac 1r\right)-\frac 2q}
			\norm{(E_0,B_0)}_{\dot B^s_{2,n,>}}
			+c^{\frac {1}2\left(\frac 12-\frac 1r\right)+2\left(\frac 1p-\frac 1q\right)-1}
			\norm{P(u\times B)}_{\widetilde L_t^{p}\dot B^s_{2,n}}
			\\
			&\lesssim
			c^{\frac {1}2\left(\frac 12-\frac 1r\right)-\frac 2q}
			\norm{(E_0,B_0)}_{\dot B^s_{2,n,>}}
			+c^{\frac {1}2\left(\frac 12-\frac 1r\right)+2\left(\frac 1p-\frac 1q\right)-1}
			\norm{u}_{L^\infty_{t,x} \cap \widetilde L^\infty_t\dot B^1_{2,\infty}}\norm{B}_{\widetilde L_t^{p}\dot B^s_{2,n}},
		\end{aligned}
	\end{equation*}
	for any $-1<s<2$. If, instead of using \eqref{para-product:2}, one employs the paradifferential product law \eqref{para-product:3}, then one arrives at the estimate
	\begin{equation*}
		\begin{aligned}
			\norm{
			\mathds{1}_{\{2^j\geq\sigma c\}}2^{j\left(s+\alpha-\frac 74+\frac 3{2r}\right)}\norm{\Delta_j (E,B)}_{L^q_tL^r_x}
			}_{\ell^n}
			\hspace{-60mm}&
			\\
			&\lesssim
			c^{\frac {1}2\left(\frac 12-\frac 1r\right)-\frac 2q}\norm{(E_0,B_0)}_{\dot B^{s+\alpha-1}_{2,n,>}}
			+c^{\frac {1}2\left(\frac 12-\frac 1r\right)+2\left(\frac 1p-\frac 1q\right)-1}
			\norm{P(u\times B)}_{\widetilde L_t^{p}\dot B^{s+\alpha-1}_{2,n}}
			\\
			&\lesssim
			c^{\frac {1}2\left(\frac 12-\frac 1r\right)-\frac 2q}\norm{(E_0,B_0)}_{\dot B^{s+\alpha-1}_{2,n,>}}
			+c^{\frac {1}2\left(\frac 12-\frac 1r\right)+2\left(\frac 1p-\frac 1q\right)-1}
			\norm{u}_{\widetilde L_t^{\infty}\dot B^\alpha_{2,1}}\norm{B}_{\widetilde L_t^{p}\dot B^s_{2,\infty}}
			\\
			&\lesssim
			c^{\frac {1}2\left(\frac 12-\frac 1r\right)-\frac 2q}\norm{(E_0,B_0)}_{\dot B^{s+\alpha-1}_{2,n,>}}
			+c^{\frac {1}2\left(\frac 12-\frac 1r\right)+2\left(\frac 1p-\frac 1q\right)-1}
			\norm{u}_{\widetilde L_t^{\infty}\dot B^0_{2,\infty}}^{1-\alpha}
			\norm{u}_{\widetilde L_t^{\infty}\dot B^1_{2,\infty}}^\alpha
			\norm{B}_{\widetilde L_t^{p}\dot B^s_{2,\infty}},
		\end{aligned}
	\end{equation*}
	which is valid for any $0<\alpha<1$ and $s<2$, with $\alpha+s>0$, where we exploited that $\widetilde L_t^{\infty}\dot B^\alpha_{2,1}$ is an interpolation space between $\widetilde L_t^{\infty}\dot B^0_{2,\infty}$ and $\widetilde L_t^{\infty}\dot B^1_{2,\infty}$ (see \cite{bl76} for details on interpolation theory).
	
	Similarly, the case $\alpha=0$ yields
	\begin{equation*}
		\begin{aligned}
			\norm{
			\mathds{1}_{\{2^j\geq\sigma c\}}2^{j\left(s-\frac 74+\frac 3{2r}\right)}\norm{\Delta_j (E,B)}_{L^q_tL^r_x}
			}_{\ell^n}
			\hspace{-45mm}&
			\\
			&\lesssim
			c^{\frac {1}2\left(\frac 12-\frac 1r\right)-\frac 2q}\norm{(E_0,B_0)}_{\dot B^{s-1}_{2,n,>}}
			+c^{\frac {1}2\left(\frac 12-\frac 1r\right)+2\left(\frac 1p-\frac 1q\right)-1}
			\norm{P(u\times B)}_{\widetilde L_t^{p}\dot B^{s-1}_{2,n}}
			\\
			&\lesssim
			c^{\frac {1}2\left(\frac 12-\frac 1r\right)-\frac 2q}\norm{(E_0,B_0)}_{\dot B^{s-1}_{2,n,>}}
			+c^{\frac {1}2\left(\frac 12-\frac 1r\right)+2\left(\frac 1p-\frac 1q\right)-1}
			\norm{u}_{\widetilde L_t^{\infty}\dot B^0_{2,\infty}}\norm{B}_{\widetilde L_t^{p}\dot B^s_{2,n}},
		\end{aligned}
	\end{equation*}
	for any $0<s<2$. This completes the justification of the high-frequency estimates.
\end{proof}

\subsection{Control of low-frequency damped electromagnetic waves}\label{section:low:freq}

The statement of Corollary \ref{cor:maxwell} has clearly emphasized how solutions to the damped Maxwell system enjoy intrinsically different properties on the distinct ranges of low and high frequencies.

In particular, the combination of Corollary \ref{cor:maxwell} with the paradifferential Lemma \ref{paradifferential:1} resulted in the nonlinear high-frequency estimates of Lemma \ref{high:freq:estimates}. A similar strategy based on the same corollary could now be employed to deduce suitable nonlinear low-frequency estimates.

However, in the next lemma, we are instead going to exploit the refined estimates established in Corollary \ref{cor:parabolic:maxwell}, which are a consequence of the maximal parabolic regularity studied in Section \ref{section:damped:parabolic:1}, to obtain a sharper control of low frequencies. This will lead to stronger statements of our main theorems.

\begin{lem}\label{low:freq:estimates}
	Let $d=2$. Assume that $(E,B)$ is a smooth solution to \eqref{Maxwell:system:*}, for some initial data $(E_0,B_0)$ and some divergence-free vector field $u$, with the normal structure \eqref{structure:2dim}.
	
	Then, for any exponents $1< p\leq q<\infty$ and $1\leq n\leq\infty$, one has the following low-frequency estimates, over any time interval $[0,T)$. For any $0<\alpha<1$ and $s<2$, with $\alpha+s>0$, it holds that
	\begin{equation}\label{frequencies:7}
		\begin{aligned}
			\norm{E}_{L^q_t\dot B^{s+\alpha-1}_{2,n,<}}
			&\lesssim
			c^{-\frac 2q}\norm{E_0}_{\dot B^{s+\alpha-1}_{2,n,<}}
			+c^{-1}\norm{B_0}_{\dot B^{s+\alpha-\frac 2q}_{2,q,<}}
			\\
			&\quad+c^{2\left(\frac 1p-\frac 1q\right)-1}
			\norm{u}_{L_t^{\infty}\dot B^0_{2,\infty}}^{1-\alpha}
			\norm{u}_{L_t^{\infty}\dot B^1_{2,\infty}}^\alpha
			\norm{B}_{L_t^{p}\dot B^s_{2,\infty}},
			\\
			\norm{B}_{L^q_t\dot B^{s+\alpha+\frac 2q-\frac 2p}_{2,1,<}}
			&\lesssim
			c^{-1}\norm{E_0}_{\dot B^{s+\alpha+1-\frac 2p}_{2,q,<}}
			+\norm{B_0}_{\dot B^{s+\alpha-\frac 2p}_{2,q,<}}
			+\norm{u}_{L_t^{\infty}\dot B^0_{2,\infty}}^{1-\alpha}
			\norm{u}_{L_t^{\infty}\dot B^1_{2,\infty}}^\alpha
			\norm{B}_{L_t^{p}\dot B^s_{2,\infty}}.
		\end{aligned}
	\end{equation}
	
	In the endpoint case $\alpha=1$, with $-1<s<2$, one also has the estimates
	\begin{equation}\label{frequencies:10}
		\norm{E}_{L^q_t\dot B^{s}_{2,n,<}}
		\lesssim
		c^{-\frac 2q}\norm{E_0}_{\dot B^{s}_{2,n,<}}
		+c^{-1}\norm{B_0}_{\dot B^{s+1-\frac 2q}_{2,q,<}}
		+c^{2\left(\frac 1p-\frac 1q\right)-1}
		\norm{u}_{L^\infty_{t,x} \cap L^\infty_t\dot B^1_{2,\infty}}\norm{B}_{ L_t^{p}\dot B^s_{2,n}},
	\end{equation}
	as well as, if $p<q$,
	\begin{equation*}
			\norm{B}_{L^q_t\dot B^{s+1+\frac 2q-\frac 2p}_{2,1,<}}
			\lesssim
			c^{-1}\norm{E_0}_{\dot B^{s+2-\frac 2p}_{2,q,<}}
			+\norm{B_0}_{\dot B^{s+1-\frac 2p}_{2,q,<}}
			+\norm{u}_{L^\infty_{t,x} \cap L^\infty_t\dot B^1_{2,\infty}}\norm{B}_{ L_t^{p}\dot B^s_{2,\infty}},
	\end{equation*}
	and, if $p=q$,
	\begin{equation}\label{frequencies:8}
		\norm{B}_{L^q_t\dot B^{s+1}_{2,n,<}}
		\lesssim
		c^{-1}\norm{E_0}_{\dot B^{s+2-\frac 2q}_{2,q,<}}
		+\norm{B_0}_{\dot B^{s+1-\frac 2q}_{2,q,<}}
		+\norm{u}_{L^\infty_{t,x} \cap L^\infty_t\dot B^1_{2,\infty}}\norm{B}_{ L_t^{q}\dot B^s_{2,n}}.
	\end{equation}
	
	Finally, in the remaining endpoint case $\alpha=0$, with $0<s<2$, it holds that
	\begin{equation*}
		\norm{E}_{L^q_t\dot B^{s-1}_{2,n,<}}
		\lesssim
		c^{-\frac 2q}\norm{E_0}_{\dot B^{s-1}_{2,n,<}}
		+c^{-1}\norm{B_0}_{\dot B^{s-\frac 2q}_{2,q,<}}
		+c^{2\left(\frac 1p-\frac 1q\right)-1}
		\norm{u}_{L_t^{\infty}\dot B^0_{2,\infty}}
		\norm{B}_{L_t^{p}\dot B^s_{2,n}},
	\end{equation*}
	as well as, if $p<q$,
	\begin{equation*}
			\norm{B}_{L^q_t\dot B^{s+\frac 2q-\frac 2p}_{2,1,<}}
			\lesssim
			c^{-1}\norm{E_0}_{\dot B^{s+1-\frac 2p}_{2,q,<}}
			+\norm{B_0}_{\dot B^{s-\frac 2p}_{2,q,<}}
			+\norm{u}_{L_t^{\infty}\dot B^0_{2,\infty}}
			\norm{B}_{L_t^{p}\dot B^s_{2,\infty}},
	\end{equation*}
	and, if $p=q$,
	\begin{equation}\label{frequencies:9}
		\norm{B}_{L^q_t\dot B^{s}_{2,n,<}}
		\lesssim
		c^{-1}\norm{E_0}_{\dot B^{s+1-\frac 2q}_{2,q,<}}
		+\norm{B_0}_{\dot B^{s-\frac 2q}_{2,q,<}}
		+\norm{u}_{L_t^{\infty}\dot B^0_{2,\infty}}
		\norm{B}_{L_t^{q}\dot B^s_{2,n}}.
	\end{equation}
\end{lem}

\begin{rem}
	The devil is in the details. In the statement of Lemma \ref{low:freq:estimates} and its corresponding proof, below, we have paid painstaking attention to the summability index of Besov spaces, occasionally referred to as the third index. Thus, in the above statement, sometimes the third index is $q$, other times it is $n$, $1$ or $\infty$. Either way, we believe that these defining values of Besov spaces are optimal, which is a crucial step to reach sharp statements of our main results.
\end{rem}

\begin{proof}
	Let us consider some fixed regularity parameters $s\in\mathbb{R}$ and $0\leq\alpha\leq 1$.
	
	On the one hand, an application of Corollary \ref{cor:maxwell} with $r=\tilde r=2$ (note that this is the best possible choice for $r$ and $\tilde r$, since all other estimates for values $2\leq r,\tilde r\leq\infty$ follow from the case $r=\tilde r=2$ by Sobolev embeddings) and $\tilde q=p'$ gives that
	\begin{equation*}
		2^{j(s+\alpha-1)}\norm{\Delta_j E}_{L^q_tL^2_x}
		\lesssim
		c^{-\frac 2q}\norm{E_0}_{\dot B^{s+\alpha-1}_{2,\infty,<}}
		+c^{-1}\norm{B_0}_{\dot B^{s+\alpha-\frac 2q}_{2,\infty,<}}
		+c^{2\left(\frac 1p-\frac 1q\right)-1} \norm{P(u\times B)}_{L_t^{p}\dot B^{s+\alpha-1}_{2,\infty}}
	\end{equation*}
	and
	\begin{equation*}
		2^{j\left(s+\alpha+\frac 2q-\frac 2p\right)}\norm{\Delta_j B}_{L^q_tL^2_x}
		\lesssim
		c^{-1}\norm{E_0}_{\dot B^{s+\alpha+1-\frac 2p}_{2,\infty,<}}
		+\norm{B_0}_{\dot B^{s+\alpha-\frac 2p}_{2,\infty,<}}
		+\norm{P(u\times B)}_{L_t^{p}\dot B^{s+\alpha-1}_{2,\infty}},
	\end{equation*}
	for all $j\in\mathbb{Z}$ with $2^j< \sigma c$, where $1\leq p\leq q\leq \infty$.
	
	On the other hand, employing Corollary \ref{cor:parabolic:maxwell} yields that
	\begin{equation*}
		\begin{aligned}
			\norm{\mathds{1}_{\{2^j<\frac{\sigma c}2\}}2^{j(s+\alpha-1)}\norm{\Delta_j E}_{L^2_x}}_{L^q_t\ell^n_j}
			&\lesssim
			c^{-\frac 2q}\norm{E_0}_{\dot B^{s+\alpha-1}_{2,n,<}}
			+c^{-1}\norm{B_0}_{\dot B^{s+\alpha-\frac 2q}_{2,q,<}}
			\\
			&\quad +c^{2\left(\frac 1p-\frac 1q\right)-1} \norm{P(u\times B)}_{L_t^{p}\dot B^{s+\alpha-1}_{2,n}},
		\end{aligned}
	\end{equation*}
	for any $1< p\leq q< \infty$ and $1\leq n\leq\infty$, as well as
	\begin{equation*}
		\begin{aligned}
			\norm{\mathds{1}_{\{2^j<\frac{\sigma c}2\}}2^{j\left(s+\alpha+\frac 2q-\frac 2p\right)}
			\norm{\Delta_j B}_{L^2_x}}_{L^q_t\ell^1_j}
			&\lesssim
			c^{-1}\norm{E_0}_{\dot B^{s+\alpha+1-\frac 2p}_{2,q,<}}
			+\norm{B_0}_{\dot B^{s+\alpha-\frac 2p}_{2,q,<}}
			\\
			&\quad+\norm{P(u\times B)}_{L_t^{p}\dot B^{s+\alpha-1}_{2,\infty}},
		\end{aligned}
	\end{equation*}
	for any $1< p< q< \infty$, and
	\begin{equation*}
		\norm{\mathds{1}_{\{2^j<\frac{\sigma c}2\}}2^{j\left(s+\alpha\right)}
		\norm{\Delta_j B}_{L^2_x}}_{L^q_t\ell^n_j}
		\lesssim
		c^{-1}\norm{E_0}_{\dot B^{s+\alpha+1-\frac 2q}_{2,q,<}}
		+\norm{B_0}_{\dot B^{s+\alpha-\frac 2q}_{2,q,<}}
		+\norm{P(u\times B)}_{L_t^{q}\dot B^{s+\alpha-1}_{2,n}},
	\end{equation*}
	for any $1< q< \infty$ and $1\leq n\leq\infty$.
	
	On the whole, combining the above estimates, we arrive at the conclusion that
	\begin{equation*}
		\norm{E}_{L^q_t\dot B^{s+\alpha-1}_{2,n,<}}
		\lesssim
		c^{-\frac 2q}\norm{E_0}_{\dot B^{s+\alpha-1}_{2,n,<}}
		+c^{-1}\norm{B_0}_{\dot B^{s+\alpha-\frac 2q}_{2,q,<}}
		+c^{2\left(\frac 1p-\frac 1q\right)-1} \norm{P(u\times B)}_{L_t^{p}\dot B^{s+\alpha-1}_{2,n}},
	\end{equation*}
	for any $1< p\leq q< \infty$ and $1\leq n\leq\infty$, as well as
	\begin{equation*}
			\norm{B}_{L^q_t\dot B^{s+\alpha+\frac 2q-\frac 2p}_{2,1,<}}
			\lesssim
			c^{-1}\norm{E_0}_{\dot B^{s+\alpha+1-\frac 2p}_{2,q,<}}
			+\norm{B_0}_{\dot B^{s+\alpha-\frac 2p}_{2,q,<}}
			+\norm{P(u\times B)}_{L_t^{p}\dot B^{s+\alpha-1}_{2,\infty}},
	\end{equation*}
	for any $1< p< q< \infty$, and
	\begin{equation*}
		\norm{B}_{L^q_t\dot B^{s+\alpha}_{2,n,<}}
		\lesssim
		c^{-1}\norm{E_0}_{\dot B^{s+\alpha+1-\frac 2q}_{2,q,<}}
		+\norm{B_0}_{\dot B^{s+\alpha-\frac 2q}_{2,q,<}}
		+\norm{P(u\times B)}_{L_t^{q}\dot B^{s+\alpha-1}_{2,n}},
	\end{equation*}
	for any $1< q< \infty$ and $1\leq n\leq\infty$.
	
	We are now going to apply the paradifferential product inequalities from Lemma \ref{paradifferential:1} to the three preceding controls. More precisely, setting $\alpha=1$, restricting the range of $s$ to $(-1,2)$ and utilizing the product law \eqref{para-product:2} to handle the nonlinear term $u\times B$, we obtain that
	\begin{equation*}
		\norm{E}_{L^q_t\dot B^{s}_{2,n,<}}
		\lesssim
		c^{-\frac 2q}\norm{E_0}_{\dot B^{s}_{2,n,<}}
		+c^{-1}\norm{B_0}_{\dot B^{s+1-\frac 2q}_{2,q,<}}
		+c^{2\left(\frac 1p-\frac 1q\right)-1}
		\norm{u}_{L^\infty_{t,x} \cap L^\infty_t\dot B^1_{2,\infty}}\norm{B}_{ L_t^{p}\dot B^s_{2,n}},
	\end{equation*}
	for any $1< p\leq q< \infty$ and $1\leq n\leq\infty$, as well as
	\begin{equation*}
			\norm{B}_{L^q_t\dot B^{s+1+\frac 2q-\frac 2p}_{2,1,<}}
			\lesssim
			c^{-1}\norm{E_0}_{\dot B^{s+2-\frac 2p}_{2,q,<}}
			+\norm{B_0}_{\dot B^{s+1-\frac 2p}_{2,q,<}}
			+\norm{u}_{L^\infty_{t,x} \cap L^\infty_t\dot B^1_{2,\infty}}\norm{B}_{ L_t^{p}\dot B^s_{2,\infty}},
	\end{equation*}
	for any $1< p< q< \infty$, and
	\begin{equation*}
		\norm{B}_{L^q_t\dot B^{s+1}_{2,n,<}}
		\lesssim
		c^{-1}\norm{E_0}_{\dot B^{s+2-\frac 2q}_{2,q,<}}
		+\norm{B_0}_{\dot B^{s+1-\frac 2q}_{2,q,<}}
		+\norm{u}_{L^\infty_{t,x} \cap L^\infty_t\dot B^1_{2,\infty}}\norm{B}_{ L_t^{q}\dot B^s_{2,n}},
	\end{equation*}
	for any $1< q< \infty$ and $1\leq n\leq\infty$.
	
	Similarly, choosing parameters $0<\alpha<1$ and $s<2$, with $\alpha+s>0$, utilizing \eqref{para-product:3} instead of \eqref{para-product:2}, and exploiting again that $L_t^{\infty}\dot B^\alpha_{2,1}$ is an interpolation space between $L_t^{\infty}\dot B^0_{2,\infty}$ and $L_t^{\infty}\dot B^1_{2,\infty}$, we find that
	\begin{equation*}
		\begin{aligned}
			\norm{E}_{L^q_t\dot B^{s+\alpha-1}_{2,n,<}}
			&\lesssim
			c^{-\frac 2q}\norm{E_0}_{\dot B^{s+\alpha-1}_{2,n,<}}
			+c^{-1}\norm{B_0}_{\dot B^{s+\alpha-\frac 2q}_{2,q,<}}
			\\
			&\quad+c^{2\left(\frac 1p-\frac 1q\right)-1}
			\norm{u}_{L_t^{\infty}\dot B^0_{2,\infty}}^{1-\alpha}
			\norm{u}_{L_t^{\infty}\dot B^1_{2,\infty}}^\alpha
			\norm{B}_{L_t^{p}\dot B^s_{2,\infty}},
		\end{aligned}
	\end{equation*}
	for any $1< p\leq q< \infty$ and $1\leq n\leq\infty$, as well as
	\begin{equation*}
			\norm{B}_{L^q_t\dot B^{s+\alpha+\frac 2q-\frac 2p}_{2,1,<}}
			\lesssim
			c^{-1}\norm{E_0}_{\dot B^{s+\alpha+1-\frac 2p}_{2,q,<}}
			+\norm{B_0}_{\dot B^{s+\alpha-\frac 2p}_{2,q,<}}
			+\norm{u}_{L_t^{\infty}\dot B^0_{2,\infty}}^{1-\alpha}
			\norm{u}_{L_t^{\infty}\dot B^1_{2,\infty}}^\alpha
			\norm{B}_{L_t^{p}\dot B^s_{2,\infty}},
	\end{equation*}
	for any $1< p\leq q< \infty$.
	
	Finally, the case $\alpha=0$, with $0<s<2$, is handled with the same product estimate \eqref{para-product:3} and results in
	\begin{equation*}
		\norm{E}_{L^q_t\dot B^{s-1}_{2,n,<}}
		\lesssim
		c^{-\frac 2q}\norm{E_0}_{\dot B^{s-1}_{2,n,<}}
		+c^{-1}\norm{B_0}_{\dot B^{s-\frac 2q}_{2,q,<}}
		+c^{2\left(\frac 1p-\frac 1q\right)-1}
		\norm{u}_{L_t^{\infty}\dot B^0_{2,\infty}}
		\norm{B}_{L_t^{p}\dot B^s_{2,n}},
	\end{equation*}
	for any $1< p\leq q< \infty$ and $1\leq n\leq\infty$, as well as
	\begin{equation*}
			\norm{B}_{L^q_t\dot B^{s+\frac 2q-\frac 2p}_{2,1,<}}
			\lesssim
			c^{-1}\norm{E_0}_{\dot B^{s+1-\frac 2p}_{2,q,<}}
			+\norm{B_0}_{\dot B^{s-\frac 2p}_{2,q,<}}
			+\norm{u}_{L_t^{\infty}\dot B^0_{2,\infty}}
			\norm{B}_{L_t^{p}\dot B^s_{2,\infty}},
	\end{equation*}
	for any $1< p< q< \infty$, and
	\begin{equation*}
		\norm{B}_{L^q_t\dot B^{s}_{2,n,<}}
		\lesssim
		c^{-1}\norm{E_0}_{\dot B^{s+1-\frac 2q}_{2,q,<}}
		+\norm{B_0}_{\dot B^{s-\frac 2q}_{2,q,<}}
		+\norm{u}_{L_t^{\infty}\dot B^0_{2,\infty}}
		\norm{B}_{L_t^{q}\dot B^s_{2,n}},
	\end{equation*}
	for any $1< q< \infty$ and $1\leq n\leq\infty$, thereby completing the justifications of all low-frequency estimates.
\end{proof}

Note that the estimates from Lemma \ref{low:freq:estimates} do not include the value $q=\infty$. Rather than providing a technical extension of the preceding proof to incorporate the value $q=\infty$, we show in the next result a simple energy estimate on Maxwell's system \eqref{Maxwell:system:*}, which corresponds to the case $q=\infty$ in Lemma \ref{low:freq:estimates}. This energy estimate allows us to propagate the $\dot H^1_x$-norm of electromagnetic fields, which will come in handy in the proof of Theorem \ref{main:1}, below.

\begin{lem}\label{prop:classical:energy}
	Let $d=2$. Assume that $(E,B)$ is a smooth solution to \eqref{Maxwell:system:*}, for some initial data $(E_0,B_0)$ and some divergence-free vector field $u$, with the normal structure \eqref{structure:2dim}.
	
	Then, one has the estimates
	\begin{equation}\label{energy:EB}
		\begin{aligned}
			\norm{(E,B)}_{L^\infty_t\dot H^1_x}+c\norm{E}_{L^2_t\dot H^1_x}
			&\lesssim \norm{(E_0,B_0)}_{\dot H^1_x}
			+\|u\|_{L^\infty_tL^2_x}\norm{\nabla B}_{L^2_tL^\infty_x},
			\\
			\norm{(E,B)}_{L^\infty_t\dot H^1_x}+c\norm{E}_{L^2_t\dot H^1_x}
			&\lesssim \norm{(E_0,B_0)}_{\dot H^1_x}
			+\norm{u}_{L_t^{\infty}\dot B^0_{2,\infty}}^{1-\alpha}
			\norm{u}_{L_t^{\infty}\dot B^1_{2,\infty}}^\alpha
			\norm{B}_{L_t^{2}\dot B^{2-\alpha}_{2,\infty}},
			\\
			\norm{(E,B)}_{L^\infty_t\dot H^1_x}+c\norm{E}_{L^2_t\dot H^1_x}
			&\lesssim \norm{(E_0,B_0)}_{\dot H^1_x}
			+\norm{u}_{L^\infty_{t,x} \cap L^\infty_t\dot B^1_{2,\infty}}\norm{B}_{ L^2_t\dot H^1_x},
		\end{aligned}
	\end{equation}
	over any time interval $[0,T)$, for any $0<\alpha<1$.
\end{lem}

\begin{proof}
	 We perform a classical energy estimate on \eqref{Maxwell:system:*}. More precisely, denoting
	\begin{equation*}
		(\widetilde{E},\widetilde{B})\bydef(\nabla \times E,\nabla \times B),
	\end{equation*}
	we observe from \eqref{Maxwell:system:*} that $(\widetilde{E},\widetilde{B})$ solves the system
	\begin{equation*}
		\begin{cases}
			\begin{aligned}
				\frac{1}{c} \partial_t \widetilde E - \nabla \times \widetilde B + \sigma c \widetilde E & =- \sigma \nabla\times ( u \times B) = \sigma (u\cdot \nabla)B-\sigma (B\cdot \nabla)u,
				\\
				\frac{1}{c} \partial_t \widetilde B + \nabla \times \widetilde E & = 0,
			\end{aligned}
		\end{cases}
	\end{equation*}
	where we employed that $u$ and $B$ are divergence-free fields. In fact, the preceding step holds in any dimension $d=2$ or $d=3$.
	
	However, restricting ourselves to the two-dimensional setting and assuming that the field $(u,E,B)$ satisfies the structure \eqref{structure:2dim} allows us to discard the term $(B\cdot\nabla)u$. To be precise, we now have that $(\widetilde{E},\widetilde{B})$ solves
	\begin{equation*}
		\begin{cases}
			\begin{aligned}
				\frac{1}{c} \partial_t \widetilde{E} - \nabla \times \widetilde{B} + \sigma c \widetilde{E} & =-\sigma u\times\widetilde{B},
				\\
				\frac{1}{c} \partial_t \widetilde{B} + \nabla \times \widetilde{E} & = 0.
			\end{aligned}
		\end{cases}
	\end{equation*}
	Thus, multiplying the first equation by $\widetilde E$, the second by $\widetilde B$ and integrating in time and space, we deduce that
	\begin{equation*}
		\begin{aligned}
			\frac 1{2c}\big\|\big(\widetilde E,\widetilde B\big)(T)\big\|_{L^2_x}^2
			+\sigma c\big\|\widetilde E\big\|_{L^2_t([0,T);L^2_x)}^2
			\hspace{-30mm}&
			\\
			&\leq
			\frac 1{2c}\big\|\big(\widetilde E,\widetilde B\big)(0)\big\|_{L^2_x}^2
			+\sigma \|u\times\widetilde B\|_{L^2_t([0,T);L^2_x)}
			\big\|\widetilde E\big\|_{L^2_t([0,T);L^2_x)}
			\\
			&\leq
			\frac 1{2c}\big\|\big(\widetilde E,\widetilde B\big)(0)\big\|_{L^2_x}^2
			+\frac{\sigma}{2c} \|u\times\widetilde B\|_{L^2_t([0,T);L^2_x)}^2
			+\frac{\sigma c}2\big\|\widetilde E\big\|_{L^2_t([0,T);L^2_x)}^2.
		\end{aligned}
	\end{equation*}
	Further employing the fact that
	\begin{equation*}
		\norm {(E,B)}_{\dot H^1_x}
		=\norm { (\nabla  E,\nabla B)}_{L^2_x}
		=\big\|(\widetilde E,\widetilde B)\big\|_{L^2_x},
	\end{equation*}
	we arrive at the estimate
	\begin{equation*}
		\norm{(E,B)(T)}_{\dot H^1_x}^2
		+\sigma c^2\norm{E}_{L^2_t([0,T);\dot H^1_x)}^2
		\leq \norm{(E_0,B_0)}_{\dot H^1_x}^2
		+\sigma\|u\times\widetilde B\|_{L^2_t([0,T);L^2_x)}^2.
	\end{equation*}
	Therefore, the proof will be completed upon controlling the nonlinear term $\|u\times\widetilde B\|_{L^2_t([0,T);L^2_x)}^2$.
	
	To this end, an elementary application of H\"older's inequality first leads to
	\begin{equation*}
		\|u\times\widetilde B\|_{L^2_{t,x}}
		\leq \|u\|_{L^\infty_tL^2_x}
		\norm{\nabla B}_{L^2_tL^\infty_x},
	\end{equation*}
	thereby completing the justification of the first estimate of the lemma.
	
	Alternatively, following the proof of Lemma \ref{low:freq:estimates}, one can use the paraproduct estimates from Lemma \ref{paradifferential:1}, again. More specifically, utilizing the product law \eqref{para-product:2}, we obtain that
	\begin{equation*}
		\|u\times\widetilde B\|_{L^2_{t,x}}=\|P(u\times B)\|_{L^2_t\dot H^1_x}
		\lesssim
		\norm{u}_{L^\infty_{t,x} \cap L^\infty_t\dot B^1_{2,\infty}}\norm{B}_{ L^2_t\dot H^1_x}.
	\end{equation*}
	Similarly, choosing a parameter $0<\alpha<1$, utilizing \eqref{para-product:3} instead of \eqref{para-product:2}, and exploiting that $L_t^{\infty}\dot B^\alpha_{2,2}$ is an interpolation space between $L_t^{\infty}\dot B^0_{2,\infty}$ and $L_t^{\infty}\dot B^1_{2,\infty}$, we finally infer that
	\begin{equation*}
		\begin{aligned}
			\|u\times\widetilde B\|_{L^2_{t,x}}=\|P(u\times B)\|_{L^2_t\dot H^1_x}
			&\lesssim
			\norm{u}_{L_t^{\infty}\dot B^\alpha_{2,2}}
			\norm{B}_{L_t^{2}\dot B^{2-\alpha}_{2,\infty}}
			\\
			&\lesssim
			\norm{u}_{L_t^{\infty}\dot B^0_{2,\infty}}^{1-\alpha}
			\norm{u}_{L_t^{\infty}\dot B^1_{2,\infty}}^\alpha
			\norm{B}_{L_t^{2}\dot B^{2-\alpha}_{2,\infty}},
		\end{aligned}
	\end{equation*}
	which concludes the proof.
\end{proof}

\subsection{Almost-parabolic estimates on the magnetic field}\label{section:almost}

In the singular limit $c\to\infty$, Maxwell's equations \eqref{Maxwell:system:*} formally converge towards the parabolic system
\begin{equation*}
	\partial_t B+(u\cdot\nabla)B-\frac 1\sigma\Delta B=(B\cdot\nabla)u,
\end{equation*}
where we employed that $u$ and $B$ are divergence-free. This holds in both dimensions $d=2$ and $d=3$. Further assuming the two-dimensional normal structure \eqref{structure:2dim}, the preceding system reduces to the simple transport-diffusion equation
\begin{equation*}
	\partial_t b+u\cdot\nabla b-\frac 1\sigma\Delta b=0,
\end{equation*}
which satisfies the energy inequality
\begin{equation}\label{heat:energy}
	\frac 12 \norm{b(T)}_{L^2_x}^2+\frac 1\sigma\norm{\nabla b}_{L^2_x}^2\leq \frac 12 \norm{b(0)}_{L^2_x}^2,
\end{equation}
for all $T>0$, at least formally.

The estimates provided by Lemma \ref{low:freq:estimates} fail to fully capture this asymptotic parabolic behavior of Maxwell's equations, because they always contain a control of the nonlinear term $u\cdot\nabla b$, whereas this term does not contribute to the energy dissipation inequality \eqref{heat:energy}.

The next result establishes a singular almost-parabolic energy estimate for Maxwell's system \eqref{Maxwell:system:*} which recovers the classical a priori estimate \eqref{heat:energy} (up to multiplicative constants) for the heat equation in the limit $c\to\infty$. This is crucial to our work, as it will allow us to construct solutions to the Euler--Maxwell system \eqref{EM} for arbitrarily large initial data as the speed of light $c$ tends to infinity.

\begin{lem}\label{prop:heatenergy}
	Let $d=2$. Assume that $(E,B)$ is a smooth solution to \eqref{Maxwell:system:*}, for some initial data $(E_0,B_0)$ and some divergence-free vector field $u$, with the normal structure \eqref{structure:2dim}.
	
	Then, one has the estimate
	\begin{equation}\label{c-L2-energy}
		\begin{aligned}
			\norm{B}_{L^\infty_tL_x^2}+\norm{\nabla B}_{L^2_tL_x^2}
			&\lesssim \norm{B_0}_{L^2_x}+c^{-1}\norm{(E_0,B_0)}_{\dot H^1_x}
			+c^{-1}\|u\|_{L^\infty_tL^2_x}\norm{\nabla B}_{L^2_tL^\infty_x},
		\end{aligned}
	\end{equation}
	over any time interval $[0,T)$.
\end{lem}

\begin{proof}
	First, we observe that taking the curl of Amp\`ere's equation in \eqref{Maxwell:system:*} and then employing Faraday's equation to substitute $\nabla\times E=-\frac 1c\partial_tB$, when necessary, leads to the system
	\begin{equation*}
		\partial_t B +(u\cdot \nabla)B -\frac 1\sigma \Delta B = (B\cdot\nabla)u+\frac 1{\sigma c}\partial_t(\nabla\times E),
	\end{equation*}
	where we have also used that $u$ and $B$ are divergence-free. This holds in any dimension $d=2$ or $d=3$.
	Then, further assuming that $(u,E,B)$ has the two-dimensional normal structure \eqref{structure:2dim} yields that
	\begin{equation*}
		\partial_t b + u\cdot \nabla b -\frac 1\sigma \Delta b = \frac 1{\sigma c}\partial_t (\curl E).
	\end{equation*}
	
	Now, the elementary observation that
	\begin{equation*}
		\int_{\mathbb{R}^2}(u\cdot\nabla b)b dx=\frac 12 \int_{\mathbb{R}^2}u\cdot\nabla(b^2)dx
		=-\frac 12 \int_{\mathbb{R}^2}(\div u)b^2dx=0,
	\end{equation*}
	which is a consequence of the incompressibility of $u$, allows us to deduce the parabolic energy estimate
	\begin{equation}\label{crucial:calculation}
		\frac 12\frac{d}{dt}\int_{\mathbb{R}^2}b^2 dx +\frac 1\sigma \int_{\mathbb{R}^2}|\nabla b|^2 dx
		=\frac 1{\sigma c}\frac{d}{dt}\int_{\mathbb{R}^2}(\curl E) bdx
		+ \frac 1{\sigma}\int_{\mathbb{R}^2}(\curl E)^2dx,
	\end{equation}
	where we used Faraday's equation, again, to substitute $\partial_t b=-c\curl E$.
	
	Then, integrating in time, we infer that
	\begin{equation*}
		\begin{aligned}
			\frac 12\norm{b(T)}_{L^2_x}^2
			+\frac 1\sigma \norm{\nabla b}_{L^2_t([0,T);L_x^2)}^2
			& \leq \frac 12\norm{b(0)}_{L^2_x}^2
			+\frac 1{\sigma c}\int_{\mathbb{R}^2}\curl E(T) b(T)dx
			\\
			&\quad -\frac 1{\sigma c}\int_{\mathbb{R}^2}\curl E(0) b(0)dx
			+\frac 1\sigma \norm{\curl E}_{L^2_t([0,T);L_x^2)}^2
			\\
			&\leq \frac 12\norm{b(0)}_{L^2_x}^2+\frac 14\left(\norm{b(0)}_{L^2_x}^2
			+\norm{b(T)}_{L^2_x}^2\right)
			\\
			&\quad
			+\frac 1{\sigma^2 c^2}\left(\norm{E(0)}_{\dot H^1_x}^2
			+\norm{E(T)}_{\dot H^1_x}^2\right)
			+\frac 1\sigma \norm{E}_{L^2_t([0,T);\dot H^1_x)}^2,
		\end{aligned}
	\end{equation*}
	which leads to
	\begin{equation}\label{almost:1}
		\begin{aligned}
			\frac 14\norm{B(T)}_{L^2_x}^2
			+\frac 1\sigma \norm{\nabla B}_{L^2_t([0,T);L_x^2)}^2
			&\leq \frac 34\norm{B_0}_{L^2_x}^2
			+\frac 1{\sigma^2 c^2}\norm{E_0}_{\dot H^1_x}^2
			\\
			&\quad +\frac 1{\sigma^2 c^2}\left(\norm{E(T)}_{\dot H^1_x}^2
			+\sigma \norm{cE}_{L^2_t([0,T);\dot H^1_x)}^2\right).
		\end{aligned}
	\end{equation}
	Finally, combining \eqref{almost:1} with the energy estimates \eqref{energy:EB} shows that \eqref{c-L2-energy} holds, thereby reaching the conclusion of the proof.
\end{proof}

\begin{rem}
	The identity \eqref{crucial:calculation} contains the crucial calculation which enables us to extract a uniform bound on $B$ in $L^2_t\dot H^1_x$, with remainder terms of order $c^{-1}$. This estimate will play an important role in the control of the low frequencies of nonlinear source terms in the Euler--Maxwell system.
\end{rem}

\subsection{Proof of Theorem \ref{main:1}}\label{section:main:1}

We proceed to the proof of our first main result---Theorem \ref{main:1}. Recall that we are taking the liberty of assuming, for simplicity, to be dealing with a smooth solution $(u,E,B)$ of \eqref{EM}, for some smooth initial data $(u_0,E_0,B_0)$, and that the justification of the theorem is only fully completed by relying on the approximation procedure laid out in Section \ref{approximation:0}.

Our proof hinges upon the preliminary lemmas established in Sections \ref{section:vorticity}, \ref{section:high:freq}, \ref{section:low:freq} and \ref{section:almost}. Accordingly, we begin by carefully gathering all the relevant estimates, on an arbitrary time interval $[0,T)$. We will then move on to construct the energy functional which will produce the uniform bounds we are seeking.

\paragraph{\bf Control of velocity field.}
The control of $u$ is obtained from Lemmas \ref{vorticity:0} and \ref{vorticity:1}. Recalling the equivalence $\norm{\nabla u}_{L^a_x}\sim \norm{\omega}_{L^a_x}$, for any value $1<a<\infty$, these lemmas provide the estimates
\begin{equation}\label{proof:1}
	\begin{aligned}
		\norm{u}_{L^\infty_t\dot H^1_x}
		&\lesssim \norm{u_0}_{\dot H^1_x}
		+\norm{j}_{L^2_{t,x}}\norm{\nabla B}_{L^2_tL^\infty_x},
		\\
		\norm{u}_{L^\infty_t\dot W^{1,p}_x}
		&\lesssim \norm{u_0}_{\dot W^{1,p}_x}
		+\norm{j}_{L^2_{t,x}}^\frac{2}{p}
		\norm{j}_{L^2_t\dot H^1_x}^{1-\frac 2p}
		\norm{\nabla B}_{L^2_tL^\infty_x},
		\\
		\norm{u}_{L^\infty_{t,x}}
		&\lesssim
		\norm{u}_{L^\infty_tL^2_x}^{\frac{p-2}{2(p-1)}}
		\norm{u}_{L^\infty_t\dot W^{1,p}_x}^{\frac{p}{2(p-1)}}
		\\
		&\lesssim
		\norm{u}_{L^\infty_tL^2_x}^{\frac{p-2}{2(p-1)}}
		\left(\norm{u_0}_{\dot W^{1,p}_x}
		+\norm{j}_{L^2_{t,x}}^\frac{2}{p}
		\norm{j}_{L^2_t\dot H^1_x}^{1-\frac 2p}
		\norm{\nabla B}_{L^2_tL^\infty_x}\right)^{\frac p{2(p-1)}},
	\end{aligned}
\end{equation}
where $2< p<\infty$ is a fixed value.

\paragraph{\bf Control of high electromagnetic frequencies.}
The control of high frequencies of $(E,B)$ is obtained from Lemma \ref{high:freq:estimates}. Specifically, setting $s=\frac 74$ in \eqref{frequencies:4} allows us to deduce that
\begin{equation*}
	\begin{aligned}
		c^{-\frac 34}\norm{(E,B)}_{\widetilde L^\infty_t\dot B^{\frac 74}_{2,1,>}}
		+c^{\frac 14}\norm{(E,B)}_{\widetilde L^2_t\dot B^{\frac 74}_{2,1,>}}
		+\norm{(E,B)}_{\widetilde L^2_t\dot B^{1}_{\infty,1,>}}
		\hspace{-50mm}&
		\\
		&\lesssim
		c^{-\frac 34}\norm{(E_0,B_0)}_{\dot B^\frac{7}{4}_{2,1,>}}
		+c^{-\frac 34}
		\norm{u}_{L^\infty_{t,x} \cap L^\infty_t\dot H^1_x}\norm{B}_{\widetilde L_t^{2}\dot B^\frac{7}{4}_{2,1}}.
	\end{aligned}
\end{equation*}
Then, further decomposing high and low frequencies in the last term above, we obtain that
\begin{equation*}
	\begin{aligned}
		\norm{B}_{\widetilde L_t^{2}\dot B^\frac{7}{4}_{2,1}}
		&\lesssim
		\norm{B}_{\widetilde L_t^{2}\dot B^\frac{7}{4}_{2,1,<}}
		+\norm{B}_{\widetilde L_t^{2}\dot B^\frac{7}{4}_{2,1,>}}
		\\
		&\lesssim
		\norm{B}_{\widetilde L_t^{2}\dot B^1_{2,\infty,<}}^\frac 14
		\norm{B}_{\widetilde L_t^{2}\dot B^2_{2,\infty,<}}^\frac 34
		+\norm{B}_{\widetilde L_t^{2}\dot B^\frac{7}{4}_{2,1,>}}
		\\
		&\lesssim
		\norm{B}_{L_t^{2}\dot H^1_x}^\frac 14
		\norm{B}_{L_t^{2}\dot B^2_{2,1,<}}^\frac 34
		+\norm{B}_{\widetilde L_t^{2}\dot B^\frac{7}{4}_{2,1,>}},
	\end{aligned}
\end{equation*}
which yields the estimate
\begin{equation}\label{proof:2}
	\begin{aligned}
		c^{-\frac 34}\norm{(E,B)}_{\widetilde L^\infty_t\dot B^{\frac 74}_{2,1,>}}
		+c^{\frac 14}\norm{(E,B)}_{\widetilde L^2_t\dot B^{\frac 74}_{2,1,>}}
		+\norm{(E,B)}_{\widetilde L^2_t\dot B^{1}_{\infty,1,>}}
		\hspace{-95mm}&
		\\
		&\lesssim
		c^{-\frac 34}\norm{(E_0,B_0)}_{\dot B^\frac{7}{4}_{2,1,>}}
		+c^{-\frac 34}
		\norm{u}_{L^\infty_{t,x} \cap L^\infty_t\dot H^1_x}
		\left(\norm{B}_{L_t^{2}\dot H^1_x}^\frac 14
		\norm{B}_{L_t^{2}\dot B^2_{2,1,<}}^\frac 34
		+\norm{B}_{\widetilde L_t^{2}\dot B^\frac{7}{4}_{2,1,>}}\right).
	\end{aligned}
\end{equation}
Observe that the choice of regularity parameter $s=\frac 74$ is dictated by the need to control $\nabla B$ in $L^2_tL^\infty_x$, as anticipated in the strategy presented in Section \ref{strategy:0}. Further notice that it is therefore crucial to be able to set a regularity parameter $s$ with a value greater than one in \eqref{frequencies:4}. This flexibility comes from the use of the normal structure \eqref{structure:2dim}, which we exploited in the product estimates established in Section \ref{section:paradifferential}.

\paragraph{\bf Control of low electromagnetic frequencies.}
The control of low frequencies of $(E,B)$ is deduced from Lemmas \ref{low:freq:estimates} and \ref{prop:classical:energy}. Here as well, the choice of parameters is dictated by the need to control $\nabla B$ in $L^2_tL^\infty_x$. Thus, since $\dot B^2_{2,1}(\mathbb{R}^2)$ is contained in $L^\infty(\mathbb{R}^2)$, one could, for instance, set $s=2$ in \eqref{frequencies:9}, which gives that
\begin{equation*}
	\norm{B}_{L^2_t\dot B^{2}_{2,1,<}}
	\lesssim
	c^{-1}\norm{E_0}_{\dot B^{2}_{2,2,<}}
	+\norm{B_0}_{\dot B^{1}_{2,2,<}}
	+\norm{u}_{L_t^{\infty}L^2_x}
	\norm{B}_{L_t^{2}\dot B^2_{2,1}}.
\end{equation*}
But, due to the coefficient $\norm{u}_{L_t^{\infty}L^2_x}$, such an estimate would eventually lead to a smallness condition, uniform in $c$, on the initial energy $\mathcal{E}_0$, which is not desirable. Therefore, instead, we employ \eqref{frequencies:7} and \eqref{frequencies:8}. More precisely, by setting $s=1$ in \eqref{frequencies:8} and $s=2-\alpha$ in \eqref{frequencies:7}, for any choice $0<\alpha<1$, we obtain the respective estimates
\begin{equation*}
	\begin{aligned}
		\norm{B}_{L^2_t\dot B^{2}_{2,1,<}}
		& \lesssim
		c^{-1}\norm{E_0}_{\dot B^{2}_{2,2,<}}
		+\norm{B_0}_{\dot B^{1}_{2,2,<}}
		+\norm{u}_{L^\infty_{t,x} \cap L^\infty_t\dot H^1_x}\norm{B}_{ L_t^{2}\dot B^1_{2,1}},
		\\
		\norm{B}_{L^2_t\dot B^{2}_{2,1,<}}
		& \lesssim
		c^{-1}\norm{E_0}_{\dot B^{2}_{2,2,<}}
		+\norm{B_0}_{\dot B^{1}_{2,2,<}}
		+\norm{u}_{L_t^{\infty}L^2_x}^{1-\alpha}
		\norm{u}_{L_t^{\infty}\dot H^1_x}^\alpha
		\norm{B}_{L_t^{2}\dot B^{1}_{2,\infty}}^\alpha\norm{B}_{L_t^{2}\dot B^{2}_{2,\infty}}^{1-\alpha}.
	\end{aligned}
\end{equation*}
In fact, it is possible to straightforwardly adapt the proofs of the above estimates to derive the more useful combined control, where we split the high and low frequencies of $B$ in the right-hand side,
\begin{equation}\label{proof:11}
	\begin{aligned}
		\norm{B}_{L^2_t\dot B^{2}_{2,1,<}}
		& \lesssim
		c^{-1}\norm{E_0}_{\dot B^{2}_{2,2,<}}
		+\norm{B_0}_{\dot B^{1}_{2,2,<}}
		+\norm{u}_{L^\infty_{t,x} \cap L^\infty_t\dot H^1_x}\norm{\mathds{1}_{\{|D|\geq \frac{\sigma c}2\}}B}_{ L_t^{2}\dot B^1_{2,1}}
		\\
		&\quad +\norm{u}_{L_t^{\infty}L^2_x}^{1-\alpha}
		\norm{u}_{L_t^{\infty}\dot H^1_x}^\alpha
		\norm{\mathds{1}_{\{|D|< \frac{\sigma c}2\}}B}_{L_t^{2}\dot B^{1}_{2,\infty}}^\alpha\norm{\mathds{1}_{\{|D|< \frac{\sigma c}2\}}B}_{L_t^{2}\dot B^{2}_{2,\infty}}^{1-\alpha}
		\\
		& \lesssim
		\norm{(E_0,B_0)}_{\dot H^1_x}
		+\norm{u}_{L^\infty_{t,x} \cap L^\infty_t\dot H^1_x}
		\left(c^{-1}\norm{B}_{ L_t^{2}\dot B^2_{2,1,<} }+
		c^{-\frac 34}\norm{B}_{ \widetilde L_t^{2}\dot B^\frac{7}{4}_{2,1,>} }
		\right)
		\\
		&\quad +\norm{u}_{L_t^{\infty}L^2_x}^{1-\alpha}
		\norm{u}_{L_t^{\infty}\dot H^1_x}^\alpha
		\norm{B}_{L_t^{2}\dot H^1_x}^\alpha
		\norm{B}_{L_t^{2}\dot B^{2}_{2,1,<}}^{1-\alpha}.
	\end{aligned}
\end{equation}
Further combining the preceding estimate with the energy estimate \eqref{energy:EB} from Lemma \ref{prop:classical:energy}, we obtain that
\begin{equation*}
	\begin{aligned}
		\norm{(E,B)}_{L^\infty_t\dot H^1_x}+c\norm{E}_{L^2_t\dot H^1_x}
		+\norm{B}_{L^2_t\dot B^{2}_{2,1,<}}
		\hspace{-40mm}&
		\\
		& \lesssim
		\norm{(E_0,B_0)}_{\dot H^1_x}
		+\norm{u}_{L^\infty_{t,x} \cap L^\infty_t\dot H^1_x}
		\left(c^{-1}\norm{B}_{ L_t^{2}\dot B^2_{2,1,<} }+
		c^{-\frac 34}\norm{B}_{ \widetilde L_t^{2}\dot B^\frac{7}{4}_{2,1,>} }
		\right)
		\\
		&\quad +\norm{u}_{L_t^{\infty}L^2_x}^{1-\alpha}
		\norm{u}_{L_t^{\infty}\dot H^1_x}^\alpha
		\norm{B}_{L_t^{2}\dot H^1_x}^\alpha
		\norm{B}_{L_t^{2}\dot B^{2}_{2,1,<}}^{1-\alpha}.
	\end{aligned}
\end{equation*}
Finally, by a classical use of Young's inequality $ab\leq \alpha a^\frac 1\alpha+(1-\alpha)b^\frac 1{1-\alpha}$, with $a,b\geq 0$, aimed at absorbing the term $\norm{B}_{L_t^{2}\dot B^{2}_{2,1,<}}^{1-\alpha}$ with the above left-hand side, we conclude that
\begin{equation}\label{proof:3}
	\begin{aligned}
		\norm{(E,B)}_{L^\infty_t\dot H^1_x}+c\norm{E}_{L^2_t\dot H^1_x}
		+\norm{B}_{L^2_t\dot B^{2}_{2,1,<}}
		\hspace{-40mm}&
		\\
		& \lesssim
		\norm{(E_0,B_0)}_{\dot H^1_x}
		+\norm{u}_{L^\infty_{t,x} \cap L^\infty_t\dot H^1_x}
		\left(c^{-1}\norm{B}_{ L_t^{2}\dot B^2_{2,1,<} }+
		c^{-\frac 34}\norm{B}_{ \widetilde L_t^{2}\dot B^\frac{7}{4}_{2,1,>} }
		\right)
		\\
		&\quad +\norm{u}_{L_t^{\infty}L^2_x}^{\frac 1\alpha-1}
		\norm{u}_{L_t^{\infty}\dot H^1_x}
		\norm{B}_{L_t^{2}\dot H^1_x}.
	\end{aligned}
\end{equation}

\paragraph{\bf Parabolic stability of magnetic field.}
The parabolic stability of the magnetic field comes as a result of the almost-parabolic estimates established in Lemma \ref{prop:heatenergy}, which we conveniently reproduce here:
\begin{equation}\label{proof:4}
	\norm{B}_{L^2_t\dot H^1_x}
	\lesssim \norm{B_0}_{L^2_x}+c^{-1}\norm{(E_0,B_0)}_{\dot H^1_x}
	+c^{-1}\|u\|_{L^\infty_tL^2_x}\norm{\nabla B}_{L^2_tL^\infty_x},
\end{equation}
This estimate will serve to control the term $\norm{B}_{L^2_t\dot H^1_x}$ in \eqref{proof:3}.

\paragraph{\bf Ohm's law estimate.}
Finally, we need to employ Ohm's law from \eqref{EM} to control the electric current $j$. More precisely, by Ohm's law and the normal structure \eqref{structure:2dim}, we have that
\begin{equation}\label{proof:5}
	\begin{aligned}
		\norm{j}_{L^2_t\dot H^1_x}&\lesssim
		c\norm{E}_{L^2_t\dot H^1_x}+\norm{P(u\times B)}_{L^2_t\dot H^1_x}
		\\
		&\lesssim
		c\norm{E}_{L^2_t\dot H^1_x}+\norm{u}_{L^\infty_t L^2_x}\norm{\nabla B}_{L^2_tL^\infty_x},
	\end{aligned}
\end{equation}
which will be used to control $\norm{j}_{L^2_t\dot H^1_x}$ in \eqref{proof:1}.

\paragraph{\bf Nonlinear energy estimate.}
We are now in a position to derive the global nonlinear energy estimate which will yield a uniform bound on solutions to \eqref{EM}. Thus, inspired by the above set of estimates, we introduce the energy $\mathcal{H}(t_1,t_2)$, with $0\leq t_1\leq t_2$, by setting
\begin{equation*}
	\begin{aligned}
		\mathcal{H}(t_1,t_2)&\bydef
		\norm{u}_{L^\infty_t\dot H^1_x}
		+\mathcal{E}_0^{\frac{p-2}{2(p-1)}}\norm{u}_{L^\infty_t\dot W^{1,p}_x}^{\frac p{2(p-1)}}
		\\
		&\quad+c^{-\frac 34}\norm{(E,B)}_{\widetilde L^\infty_t\dot B^{\frac 74}_{2,1,>}}
		+c^{\frac 14}\norm{(E,B)}_{\widetilde L^2_t\dot B^{\frac 74}_{2,1,>}}
		+\norm{(E,B)}_{\widetilde L^2_t\dot B^{1}_{\infty,1,>}}
		\\
		&\quad+\norm{(E,B)}_{L^\infty_t\dot H^1_x}+c\norm{E}_{L^2_t\dot H^1_x}+\norm{B}_{L^2_t\dot B^{2}_{2,1,<}},
	\end{aligned}
\end{equation*}
where all time-norms are taken over the interval $[t_1,t_2)$. Since $(u,E,B)$ is assumed, without any loss of generality, to be a smooth solution of \eqref{EM}, observe that $\mathcal{H}(t_1,t_2)$ is continuous on $\{0\leq t_1\leq t_2\}$. In particular, we can further define the continuous function
\begin{equation*}
	\mathcal{H}(t)\bydef \mathcal{H}(t,t)=
	\norm{u(t)}_{\dot H^1_x}+\mathcal{E}_0^{\frac{p-2}{2(p-1)}}\norm{u(t)}_{\dot W^{1,p}_x}^{\frac p{2(p-1)}}
	+c^{-\frac 34}\norm{(E,B)(t)}_{\dot B^{\frac 74}_{2,1,>}}
	+\norm{(E,B)(t)}_{\dot H^1_x},
\end{equation*}
for every $t\geq 0$. Note that
\begin{equation*}
	\mathcal{H}(t)\leq \mathcal{H}(t_1,t_2),
\end{equation*}
for all $t\in[t_1,t_2]$.

It will also be useful to assign the notation
\begin{equation*}
	\mathcal{J}(t_1,t_2)\bydef \norm{j}_{L^2_{t,x}}
\end{equation*}
to the dissipation produced by the electric current in \eqref{energy-inequa}, where the $L^2$-norm is taken over the time-interval $[t_1,t_2)$, as well. In particular, one has the uniform bound
\begin{equation}\label{dissipation:1}
	\mathcal{J}(t_1,t_2)\leq\left(\frac\sigma 2\right)^\frac 12\mathcal{E}_0
\end{equation}
by virtue of the energy inequality \eqref{energy-inequa}.

Observe now that all the above estimates, which were stated on the time-interval $[0,T)$, could equally well be written over any other finite time-interval $[t_1,t_2)$, provided one replaces the initial data $(u_0,E_0,B_0)$ by the data at time $t_1$. Thus, employing the energy inequality \eqref{energy-inequa}, the energies $\mathcal{H}(t_1,t_2)$ and $\mathcal{H}(t)$, and the embedding
\begin{equation*}
	\norm{\nabla B}_{L^2_tL^\infty_x}\leq\norm{B}_{L^2_t\dot B^1_{\infty,1}}
	\lesssim
	\norm{B}_{L^2_t\dot B^{2}_{2,1,<}}
	+\norm{B}_{\widetilde L^2_t\dot B^{1}_{\infty,1,>}}\leq\mathcal{H}(t_1,t_2),
\end{equation*}
one can write the parabolic stability estimate \eqref{proof:4} as
\begin{equation}\label{proof:9}
	\norm{B}_{L^2_t\dot H^1_x}
	\lesssim \mathcal{E}_0+c^{-1}\mathcal{H}(t_1)
	+c^{-1}\mathcal{E}_0\mathcal{H}(t_1,t_2),
\end{equation}
and the Ohm's law estimate \eqref{proof:5} as
\begin{equation}\label{proof:10}
	\norm{j}_{L^2_t\dot H^1_x}\lesssim (1+\mathcal{E}_0)\mathcal{H}(t_1,t_2),
\end{equation}
which are linear in $\mathcal{H}(t_1,t_2)$.

Then, incorporating the linear estimate \eqref{proof:10} into the velocity control \eqref{proof:1}, we obtain that
\begin{equation}\label{proof:6}
	\begin{aligned}
		\norm{u}_{L^\infty_t\dot H^1_x}
		&\lesssim \mathcal{H}(t_1)
		+\mathcal{J}(t_1,t_2)\mathcal{H}(t_1,t_2),
		\\
		\norm{u}_{L^\infty_{t,x}}\lesssim
		\mathcal{E}_0^{\frac{p-2}{2(p-1)}}\norm{u}_{L^\infty_t\dot W^{1,p}_x}^{\frac p{2(p-1)}}
		&\lesssim
		\mathcal{H}(t_1)
		+\mathcal{E}_0^{\frac{p-2}{2(p-1)}}(1+\mathcal{E}_0)^{\frac {p-2}{2(p-1)}}
		\mathcal{J}(t_1,t_2)^{\frac 1{p-1}}
		\mathcal{H}(t_1,t_2).
	\end{aligned}
\end{equation}
Furthermore, the use of \eqref{proof:9} and \eqref{proof:6} in the high-frequency control \eqref{proof:2} leads to
\begin{equation}\label{proof:7}
	\begin{aligned}
		&c^{-\frac 34}\norm{(E,B)}_{\widetilde L^\infty_t\dot B^{\frac 74}_{2,1,>}}
		+c^{\frac 14}\norm{(E,B)}_{\widetilde L^2_t\dot B^{\frac 74}_{2,1,>}}
		+\norm{(E,B)}_{\widetilde L^2_t\dot B^{1}_{\infty,1,>}}
		\\
		&\lesssim
		\mathcal{H}(t_1)
		+c^{-\frac 34}\mathcal{H}(t_1,t_2)
		\left(\mathcal{E}_0+c^{-1}\mathcal{H}(t_1)
		+c^{-1}\mathcal{E}_0\mathcal{H}(t_1,t_2)\right)^\frac 14
		\mathcal{H}(t_1,t_2)^\frac 34
		+c^{-1}\mathcal{H}(t_1,t_2)^2
		\\
		&\lesssim
		\mathcal{H}(t_1)
		+c^{-\frac 34}\mathcal{E}_0^\frac 14
		\mathcal{H}(t_1,t_2)^\frac 74
		+c^{-1}\left(1+\mathcal{E}_0^{\frac 14}\right)\mathcal{H}(t_1,t_2)^2
		\\
		&\lesssim
		\mathcal{H}(t_1)
		+\lambda \mathcal{H}(t_1,t_2)
		+\lambda^{-\frac 13}c^{-1}\mathcal{E}_0^\frac 13
		\mathcal{H}(t_1,t_2)^2
		+c^{-1}\left(1+\mathcal{E}_0^{\frac 14}\right)\mathcal{H}(t_1,t_2)^2,
	\end{aligned}
\end{equation}
for any $\lambda>0$,
whereas a similar procedure applied to the low-frequency estimate \eqref{proof:3} yields that
\begin{equation}\label{proof:8}
	\begin{aligned}
		\norm{(E,B)}_{L^\infty_t\dot H^1_x}+c\norm{E}_{L^2_t\dot H^1_x}
		+\norm{B}_{L^2_t\dot B^{2}_{2,1,<}}
		\hspace{-40mm}&
		\\
		& \lesssim
		\mathcal{H}(t_1)
		+c^{-1}\mathcal{H}(t_1,t_2)^2
		\\
		&\quad+\mathcal{E}_0^{\frac 1\alpha -1}
		\left(\mathcal{H}(t_1)
		+\mathcal{J}(t_1,t_2)\mathcal{H}(t_1,t_2)\right)
		\left(\mathcal{E}_0+c^{-1}\mathcal{H}(t_1)+c^{-1}\mathcal{E}_0\mathcal{H}(t_1,t_2)\right)
		\\
		& \lesssim
		\left(1+\mathcal{E}_0^{\frac 1\alpha}\right)\mathcal{H}(t_1)
		+\mathcal{E}_0^{\frac 1\alpha}\mathcal{J}(t_1,t_2)\mathcal{H}(t_1,t_2)
		+\left(1+\mathcal{E}_0^{\frac 1\alpha+1}\right)c^{-1}\mathcal{H}(t_1,t_2)^2.
	\end{aligned}
\end{equation}
All in all, summing estimates \eqref{proof:6}, \eqref{proof:7} and \eqref{proof:8} together, and setting $\lambda$ in \eqref{proof:7} so small that the term $\lambda\mathcal{H}(t_1,t_2)$ can be absorbed by the resulting left-hand side, we finally arrive at the crucial nonlinear energy estimate
\begin{equation*}
	\begin{aligned}
		\mathcal{H}(t_1,t_2)&\lesssim
		\left(1+\mathcal{E}_0^{\frac 1\alpha}\right)\mathcal{H}(t_1)
		+\left(1+\mathcal{E}_0^{\frac 1\alpha}\right)\mathcal{J}(t_1,t_2)\mathcal{H}(t_1,t_2)
		\\
		&\quad+\mathcal{E}_0^{\frac{p-2}{2(p-1)}}\left(1+\mathcal{E}_0^{\frac {p-2}{2(p-1)}}\right)
		\mathcal{J}(t_1,t_2)^{\frac 1{p-1}}
		\mathcal{H}(t_1,t_2)
		+\left(1+\mathcal{E}_0^{\frac 1\alpha+1}\right)c^{-1}\mathcal{H}(t_1,t_2)^2,
	\end{aligned}
\end{equation*}
for any $0\leq t_1\leq t_2$, which, using \eqref{dissipation:1}, can be slightly simplified into
\begin{equation}\label{nonlinear:energy}
	\begin{aligned}
		\mathcal{H}(t_1,t_2)&\leq
		C_* \left(1+\mathcal{E}_0^{\frac 1\alpha}\right)
		\mathcal{H}(t_1)
		+C_* \left(1+\mathcal{E}_0^{\frac 1\alpha+\frac{p-2}{p-1}}\right)
		\mathcal{J}(t_1,t_2)^\frac 1{p-1}\mathcal{H}(t_1,t_2)
		\\
		&\quad +C_* \left(1+\mathcal{E}_0^{\frac 1\alpha+1}\right)c^{-1}\mathcal{H}(t_1,t_2)^2,
	\end{aligned}
\end{equation}
where $C_*>0$ only depends on fixed parameters (in particular, it is independent of time, the speed of light $c$ and the initial data).
We are now going to show that \eqref{nonlinear:energy} entails a global bound on $\mathcal{H}(t_1,t_2)$.

\paragraph{\bf Conclusion of proof.} Let us consider a partition of time
\begin{equation*}
	0=t_0<t_1<\ldots<t_n<t_{n+1}=\infty
\end{equation*}
such that, for every $i=0,1,\ldots,n-1$,
\begin{equation*}
	C_* \left(1+\mathcal{E}_0^{\frac 1\alpha+\frac{p-2}{p-1}}\right)\mathcal{J}(t_i,t_{i+1})^\frac 1{p-1}=\frac 12
	\quad\text{and}\quad
	C_* \left(1+\mathcal{E}_0^{\frac 1\alpha+\frac{p-2}{p-1}}\right)\mathcal{J}(t_n,t_{n+1})^\frac 1{p-1}\leq\frac 12.
\end{equation*}
In particular, by virtue of \eqref{dissipation:1}, one has, for any $t\in[t_i,t_{i+1}]$, with $i=0,1,\ldots,n$, that
\begin{equation}\label{dissipation:2}
	\frac i{\left(2C_* \left(1+\mathcal{E}_0^{\frac 1\alpha+\frac{p-2}{p-1}}\right)\right)^{2(p-1)}}=\sum_{k=0}^{i-1}\mathcal{J}(t_k,t_{k+1})^2\leq \mathcal{J}(t_0,t)^2\leq \frac\sigma 2 \mathcal{E}_0^2.
\end{equation}
It then follows from \eqref{nonlinear:energy} that
\begin{equation}\label{nonlinear:energy:2}
	\begin{aligned}
		\mathcal{H}(t_i,t)&\leq 2C_* \left(1+\mathcal{E}_0^{\frac 1\alpha}\right)\mathcal{H}(t_i)
		+2C_* \left(1+\mathcal{E}_0^{\frac 1\alpha+1}\right)c^{-1}\mathcal{H}(t_i,t)^2,
	\end{aligned}
\end{equation}
for all $i=0,\ldots,n$ and $t\in[t_i,t_{i+1}]$.

Next, for fixed $i$, we introduce the quadratic polynomial
\begin{equation*}
	p(X)=2C_* \left(1+\mathcal{E}_0^{\frac 1\alpha+1}\right)c^{-1}X^2-X+2C_* \left(1+\mathcal{E}_0^{\frac 1\alpha}\right)\mathcal{H}(t_i),
\end{equation*}
whose roots
\begin{equation*}
	\lambda_\pm=\frac{1\pm\sqrt{1-16C_*^2 \left(1+\mathcal{E}_0^{\frac 1\alpha+1}\right)\left(1+\mathcal{E}_0^{\frac 1\alpha}\right)c^{-1}\mathcal{H}(t_i)}}{4C_* \left(1+\mathcal{E}_0^{\frac 1\alpha+1}\right)c^{-1}}
\end{equation*}
are real and distinct provided
\begin{equation}\label{discriminant:1}
	\mathcal{H}(t_i)<\frac{c}{16C_*^2 \left(1+\mathcal{E}_0^{\frac 1\alpha+1}\right)\left(1+\mathcal{E}_0^{\frac 1\alpha}\right)}.
\end{equation}
Observe that \eqref{nonlinear:energy:2} can be rewritten as
\begin{equation*}
	p(\mathcal{H}(t_i,t))\geq 0,
\end{equation*}
for all $t\in[t_i,t_{i+1}]$. Therefore, by continuity of $\mathcal{H}(t_i,t)$ and assuming that \eqref{discriminant:1} is satisfied, we deduce that $\mathcal{H}(t_i,t)\leq \lambda_-$, for all $t\in[t_i,t_{i+1}]$, if it is true for $t=t_i$, i.e., $\mathcal{H}(t_i)\leq \lambda_-$.

Now, it is readily seen that \eqref{discriminant:1} implies that $\mathcal{H}(t_i)\leq \lambda_-$, if we assume, without any loss of generality, that $2C_* \left(1+\mathcal{E}_0^{\frac 1\alpha}\right)\geq 1$. Thus, by continuity of $t\mapsto \mathcal{H}(t_i,t)$, we conclude that \eqref{discriminant:1} is sufficient to deduce the bound
\begin{equation*}
	\mathcal{H}(t)\leq \mathcal{H}(t_i,t)\leq\lambda_-
	\leq 4C_* \left(1+\mathcal{E}_0^{\frac 1\alpha}\right)\mathcal{H}(t_i),
\end{equation*}
for all $t\in[t_i,t_{i+1}]$, where we used the elementary inequality $1-\sqrt{1-z}\leq z$, for all $z\in[0,1]$, in the last step.
Then, a straightforward iterative process leads us to the estimate
\begin{equation*}
	\mathcal{H}(t_i,t)\leq 4C_* \left(1+\mathcal{E}_0^{\frac 1\alpha}\right) \mathcal{H}(t_i)
	\leq \left[4C_* \left(1+\mathcal{E}_0^{\frac 1\alpha}\right)\right]^{i+1}\mathcal{H}(t_0),
\end{equation*}
for each $i=0,1,\ldots,n$ and all $t\in[t_i,t_{i+1}]$, if the initial data satisfies that
\begin{equation*}
	\mathcal{H}(t_0)<\frac{c}{4C_* \left(1+\mathcal{E}_0^{\frac 1\alpha+1}\right)
	\left[4C_* \left(1+\mathcal{E}_0^{\frac 1\alpha}\right)\right]^{i+1}}.
\end{equation*}
Further noticing, for any $t\in[t_i,t_{i+1}]$, that
\begin{equation*}
	\mathcal{H}(t_0,t)\leq\mathcal{H}(t_i,t)+\sum_{k=0}^{i-1}\mathcal{H}(t_k,t_{k+1}),
\end{equation*}
we conclude, in view of \eqref{dissipation:2}, that one has the global bound
\begin{equation*}
	\begin{aligned}
		\mathcal{H}(0,t)
		&\leq \sum_{k=0}^{i}\left[4C_* \left(1+\mathcal{E}_0^{\frac 1\alpha}\right)\right]^{k+1}\mathcal{H}(0)
		=4C_* \left(1+\mathcal{E}_0^{\frac 1\alpha}\right)
		\frac{\left[4C_* \left(1+\mathcal{E}_0^{\frac 1\alpha}\right)\right]^{i+1}-1}{4C_* \left(1+\mathcal{E}_0^{\frac 1\alpha}\right)-1}
		\mathcal{H}(0)
		\\
		&\leq 2\left[4C_* \left(1+\mathcal{E}_0^{\frac 1\alpha}\right)\right]^{i+1} \mathcal{H}(0)
		\leq 2\left[4C_* \left(1+\mathcal{E}_0^{\frac 1\alpha}\right)\right]^{1+\left[2C_* \left(1+\mathcal{E}_0^{\frac 1\alpha+\frac{p-2}{p-1}}\right)\right]^{2(p-1)}\mathcal{J}(0,t)^2} \mathcal{H}(0),
	\end{aligned}
\end{equation*}
provided
\begin{equation*}
	\begin{aligned}
		\mathcal{H}(0)&<
		\frac{c}{4C_* \left(1+\mathcal{E}_0^{\frac 1\alpha+1}\right)
		\left[4C_* \left(1+\mathcal{E}_0^{\frac 1\alpha}\right)\right]^{1+\left[2C_* \left(1+\mathcal{E}_0^{\frac 1\alpha+\frac{p-2}{p-1}}\right)\right]^{2(p-1)}\frac\sigma 2\mathcal{E}_0^2}}
		\\
		&\leq\frac{c}{4C_* \left(1+\mathcal{E}_0^{\frac 1\alpha+1}\right)\left[4C_* \left(1+\mathcal{E}_0^{\frac 1\alpha}\right)\right]^{1+n}}
	\end{aligned}
\end{equation*}
holds initially.

Summarizing the preceding developments, we have now established the existence of an independent constant $C_*>0$ such that, if the smooth solution $(u,E,B)$ has an initial data satisfying
\begin{equation}\label{initial:2}
	4C_* \left(1+\mathcal{E}_0^{\frac 1\alpha+1}\right)
	\left[4C_* \left(1+\mathcal{E}_0^{\frac 1\alpha}\right)\right]^{1+\left[2C_* \left(1+\mathcal{E}_0^{\frac 1\alpha+\frac{p-2}{p-1}}\right)\right]^{2(p-1)}\frac\sigma 2\mathcal{E}_0^2}
	\mathcal{H}(0)<c,
\end{equation}
then the bound
\begin{equation}\label{global:bound:1}
	\mathcal{H}(0,t)\leq
	2\left[4C_* \left(1+\mathcal{E}_0^{\frac 1\alpha}\right)\right]^{1+\left[2C_* \left(1+\mathcal{E}_0^{\frac 1\alpha+\frac{p-2}{p-1}}\right)\right]^{2(p-1)}\mathcal{J}(0,t)^2} \mathcal{H}(0),
\end{equation}
holds globally, for any $t\in [0,\infty)$. In view of the approximation procedure laid out in Section \ref{approximation:0}, this uniform control allows us to complete the construction of solutions claimed in the statement of Theorem \ref{main:1}.

Moreover, observe that all global bounds on $(u,E,B)$ stated in \eqref{propagation:damping} are a direct consequence of \eqref{proof:9} and \eqref{global:bound:1}.

Finally, in order to deduce the simpler initial condition \eqref{initial:1} from \eqref{initial:2}, there only remains to notice, for any given $\varepsilon>0$, by taking $0<\alpha<1$ and $2<p<\infty$ sufficiently close to the values $1$ and $2$, respectively, that
\begin{equation*}
	4C_* \left(1+\mathcal{E}_0^{\frac 1\alpha+1}\right)
	\left[4C_* \left(1+\mathcal{E}_0^{\frac 1\alpha}\right)\right]^{1+\left[2C_* \left(1+\mathcal{E}_0^{\frac 1\alpha+\frac{p-2}{p-1}}\right)\right]^{2(p-1)}\frac\sigma 2\mathcal{E}_0^2}
	\leq
	C_{**}e^{C_{**}\mathcal{E}_0^{4+\varepsilon}},
\end{equation*}
for some large independent constant $C_{**}>0$. Then, since the global bound \eqref{global:bound:1} holds for that particular choice of $p$ close to $2$, one can use the ensuing uniform controls on the solution $(u,E,B)$ in combination with \eqref{proof:1} to propagate the $L^p_x$-norm of the vorticity $\omega$ for higher values $2<p<\infty$, which completes the proof of Theorem \ref{main:1}.\qed

\subsection{Proof of Theorem \ref{main:2}}\label{section:main:2}

The proof of Theorem \ref{main:2} is a continuation of that of Theorem \ref{main:1}. Thus, assuming that the solution $(u,E,B)$ produced by Theorem \ref{main:1} is already constructed, we observe that Lemma \ref{vorticity:0} provides us with the additional bound
\begin{equation*}
	\begin{aligned}
		\norm{\omega(t)}_{L^\infty_x}
		&\lesssim \norm{\omega(0)}_{L^\infty_x}
		+\norm{j}_{L^2([0,t);L^\infty_x)}
		\norm{\nabla B}_{L^2([0,t);L^\infty_x)}
		\\
		&\lesssim \norm{\omega(0)}_{L^\infty_x}
		+\left(\norm{j}_{L^2([0,t);\dot B^0_{\infty,1,<})}+\norm{j}_{L^2([0,t);\dot B^0_{\infty,1,>})}\right)
		\norm{\nabla B}_{L^2([0,t);L^\infty_x)},
	\end{aligned}
\end{equation*}
for any $t\geq 0$, which, when combined with Ohm's law, the two-dimensional embedding $\dot B^1_{2,1}\subset \dot B^0_{\infty,1}$ and the paradifferential product law \eqref{para-product:7}, further yields that
\begin{equation*}
	\begin{aligned}
		\norm{\omega(t)}_{L^\infty_x}
		&\lesssim \norm{\omega(0)}_{L^\infty_x}
		+\left(\norm{j}_{L^2_t\dot B^1_{2,1,<}}+c\norm{E}_{L^2_t\dot B^0_{\infty,1,>}}
		+\norm{P(u\times B)}_{L^2_t\dot B^1_{2,1}}
		\right)\norm{\nabla B}_{L^2_tL^\infty_x}
		\\
		&\lesssim \norm{\omega(0)}_{L^\infty_x}
		+\Big(c\norm{j}_{L^2_t\dot B^0_{2,\infty,<}}+\norm{E}_{L^2_t\dot B^1_{\infty,1,>}}
		\\
		&\quad+\norm{u}_{L^\infty_tL^2_x}\norm{B}_{L^2_t\dot B^{1}_{\infty,1}}
		+\norm{u}_{L^\infty_t\dot H^1_x}\norm{B}_{L^2_t\dot H^1_x}
		\Big)\norm{\nabla B}_{L^2_tL^\infty_x}.
	\end{aligned}
\end{equation*}

Now, recall from \eqref{infinity:bound:1} that $\nabla B$ belongs to $L^2(\mathbb{R}^+;\dot B^0_{\infty,1})\subset L^2(\mathbb{R}^+;L^\infty_x)$. Therefore, by virtue of the energy inequality \eqref{energy-inequa}, the global bounds \eqref{propagation:damping} and the assumption that the initial vorticity belongs to $L^\infty_x$, we conclude the pointwise boundedness of $\omega(t)$ for all times.

Observe, though, that the ensuing bound $\omega\in L^\infty_{t,x}$ is global in time, but it is not uniform in $c$. Nevertheless, it is possible to derive another global bound on the vorticity in $L^\infty_{t,x}$, uniformly in $c$, by employing Ohm's law to control $j$ in $L^2_t\dot B^1_{2,1,<}$ and requiring the additional initial assumption that $E_0\in \dot B^{1}_{2,1}$.

Specifically, the use of Ohm's law to expand the low frequencies of $j$ leads to the necessity of controlling $cE$ uniformly in the space $L^2_t\dot B^1_{2,1,<}$, which can be achieved by relying on the low-frequency estimates from Lemma \ref{low:freq:estimates}. More precisely, by combining \eqref{frequencies:10}, with $s=1$, and \eqref{frequencies:7}, with $s=2-\alpha$ and $0<\alpha<1$, it is possible to establish, by repeating the steps leading to \eqref{proof:11}, that
\begin{equation*}
	\begin{aligned}
		c\norm{E}_{L^2_t\dot B^{1}_{2,1,<}}
		& \lesssim
		\norm{E_0}_{\dot B^{1}_{2,1}}+\norm{B_0}_{\dot H^1_x}
		+\norm{u}_{L^\infty_{t,x} \cap L^\infty_t\dot H^1_x}
		\left(c^{-1}\norm{B}_{ L_t^{2}\dot B^2_{2,1,<} }+
		c^{-\frac 34}\norm{B}_{ \widetilde L_t^{2}\dot B^\frac{7}{4}_{2,1,>} }
		\right)
		\\
		&\quad +\norm{u}_{L_t^{\infty}L^2_x}^{1-\alpha}
		\norm{u}_{L_t^{\infty}\dot H^1_x}^\alpha
		\norm{B}_{L_t^{2}\dot H^1_x}^\alpha
		\norm{B}_{L_t^{2}\dot B^{2}_{2,1,<}}^{1-\alpha}.
	\end{aligned}
\end{equation*}
All terms in the right-hand side of this estimate are now uniformly controlled (in $c$) by the bounds \eqref{propagation:damping}, provided one further assumes that the inital data $E_0$ belongs to $\dot B^{1}_{2,1}$. This concludes the justification of the bound $\omega\in L^\infty_{t,x}$, uniformly in $c$.

We turn now to the uniqueness of solutions to \eqref{EM}, which rests upon Yudovich's fundamental ideas. To that end, suppose that
\begin{equation*}
	(u_i,E_i,B_i)\in L^\infty([0,T);L^2_x),
\end{equation*}
with $i=1,2$, are two weak solutions to the two-dimensional incompressible Euler--Maxwell system \eqref{EM}, for the same initial data and for some existence time $T>0$. We are going to establish a weak-strong uniqueness principle by requiring a control on the solution $(u_2,E_2,B_2)$ which is stronger than the one on $(u_1,E_1,B_1)$.

In a natural way, we denote the vorticities and electric currents associated to each solution by $\omega_i$ and $j_i$, respectively.
Furthermore, we assume that each solution satisfies its correponding energy inequality \eqref{energy-inequa} and that
\begin{equation*}
	u_1\in L^2([0,T);L^\infty_x),
	\qquad
	\omega_2\in \bigcap_{2\leq q\leq \infty}L^{q'}([0,T);L^q_x),
	\qquad
	j_2\in L^1([0,T);L^\infty_x).
\end{equation*}
By virtue of the Gagliardo--Nirenberg inequality \eqref{convexity:2}, recall that the above bounds are sufficient to imply that
\begin{equation*}
	u_2\in L^2([0,T);L^\infty_x),
\end{equation*}
as well. Note that we are not requiring here that the solutions have the normal structure \eqref{structure:2dim}.

Next, a straightforward duality argument on \eqref{EM}, similar to the computation which gives the energy inequality \eqref{energy-inequa}, leads to
\begin{equation*}
	\begin{aligned}
		\frac{d}{dt}\int_{\mathbb{R}^2}\left(u_1\cdot u_2+E_1\cdot E_2+B_1\cdot B_2\right)dx
		+\frac 2\sigma\int_{\mathbb{R}^2}j_1\cdot j_2dx
		\hspace{-55mm}&
		\\
		&=-\int_{\mathbb{R}^2}\left(u_2\otimes (u_1-u_2)\right):\nabla (u_1-u_2)dx
		\\
		&\quad +\int_{\mathbb{R}^2}\left(
		((j_1-j_2)\times (B_1-B_2))\cdot u_2-(j_2\times (B_1-B_2))\cdot (u_1-u_2)
		\right)dx.
	\end{aligned}
\end{equation*}
It is to be emphasized that the solutions considered here have sufficient regularity and integrability to justify the above computation rigorously.

We introduce now the modulated energy
\begin{equation*}
	F_\varepsilon(t)\bydef
	\frac12\left(\norm {\tilde u(t)}_{L^2}^2 + \|\tilde E(t)\|_{L^2}^2 + \|\tilde B(t)\|_{L^2}^2\right)
	+\varepsilon,
\end{equation*}
where
\begin{equation*}
	\tilde u\bydef u_1-u_2,\quad \tilde E\bydef E_1-E_2,\quad \tilde B\bydef B_1-B_2,
\end{equation*}
and $\varepsilon>0$ merely ensures the positivity of $F_\varepsilon$.

Then, integrating the preceding identity in time and combining the result with the energy inequality \eqref{energy-inequa} for each solution, we deduce the estimate
\begin{equation*}
	\begin{aligned}
		F_\varepsilon(t)
		+\frac{1}{\sigma}\int_0^t \norm {\tilde j(\tau)}_{L^2_x}^2 d\tau
		&\leq \varepsilon-
		\int_0^t\int_{\mathbb{R}^2}\left[\left(\left(\tilde u\cdot\nabla\right) u_2\right)\cdot \tilde u
		+(\tilde j\times \tilde B)\cdot u_2
		-(j_2\times \tilde B)\cdot \tilde u\right](\tau)
		dxd\tau
		\\
		&\leq \varepsilon+
		\int_0^t\norm{\nabla u_2}_{L^q_x}\norm{\tilde u}_{L^{\infty}_x}^{\frac{2}{q}}
		\norm{\tilde u}_{L^{2}_x}^\frac{2}{q'}d\tau
		\\
		&\quad +\int_0^t\norm{\tilde j}_{L^2_x}\left[\big\|\tilde B\big\|_{L^2_x}\norm{u_2}_{L^\infty_x}
		+\norm{j_2}_{L^\infty_x}\big\|\tilde B\big\|_{L^2_x}\norm{\tilde u}_{L^2_x}
		\right]d\tau
		\\
		&\leq \varepsilon+
		\int_0^t\norm{\nabla u_2}_{L^q_x}\norm{\tilde u}_{L^{\infty}_x}^{\frac{2}{q}}
		\norm{\tilde u}_{L^{2}_x}^\frac{2}{q'}d\tau
		+\frac{1}{2\sigma}\int_0^t \norm {\tilde j(\tau)}_{L^2_x}^2 d\tau
		\\
		&\quad +\int_0^t\left[
		\frac\sigma 2\big\|\tilde B\big\|_{L^2_x}^2\norm{u_2}_{L^\infty_x}^2
		+\norm{j_2}_{L^\infty_x}\big\|\tilde B\big\|_{L^2_x}\norm{\tilde u}_{L^2_x}
		\right]d\tau,
	\end{aligned}
\end{equation*}
for all $t\in [0,T)$ and any $2\leq q<\infty$, where we have denoted $\tilde j\bydef j_1-j_2$.

The next step relies on a classical sharp estimate on the Biot--Savart law \eqref{biot}. More precisely, by exploiting that the map $\omega\mapsto \nabla(-\Delta^{-1}\nabla\times\omega)=\nabla u$ produces a Calder\'on--Zygmund singular integral operator, it is possible to show that
\begin{equation*}
	\norm{\nabla u}_{L^a(\mathbb{R}^2)}\leq C_{\rm BS}\frac{a^2}{a-1}\norm{\nabla\times u}_{L^a(\mathbb{R}^2)},
\end{equation*}
for all $1<a<\infty$ and any divergence-free vector field $u$, where $C_{\rm BS}>0$ is independent of $a$. We refer to \cite[Section 7.1.1]{bcd11} for more details concerning the Biot--Savart law and to \cite[Section 6.2.3]{g14} for a Fourier multiplier theorem which can be used to obtain the correct dependence of the above Biot--Savart estimate in the parameter $a$.

Thus, we deduce from the previous bound that
\begin{equation*}
	F_\varepsilon(t)
	\leq \varepsilon+
	\int_0^t2C_{\rm BS}q\norm{\omega_2}_{L^q_x}\norm{\tilde u}_{L^{\infty}_x}^{\frac{2}{q}}
	F_\varepsilon(\tau)^\frac{1}{q'}+\left[
	\sigma \norm{u_2}_{L^\infty_x}^2
	+\norm{j_2}_{L^\infty_x}
	\right]F_\varepsilon(\tau)d\tau.
\end{equation*}
Moreover, denoting the right-hand side of the above inequality by $G_\varepsilon(t)$ and observing that it is absolutely continuous on $[0,T)$, we obtain that
\begin{equation*}
	\frac d{dt}G_\varepsilon(t)
	\leq
	2C_{\rm BS}q\norm{\omega_2}_{L^q_x}\norm{\tilde u}_{L^{\infty}_x}^{\frac{2}{q}}
	G_\varepsilon(t)^\frac{1}{q'}+\left[
	\sigma \norm{u_2}_{L^\infty_x}^2
	+\norm{j_2}_{L^\infty_x}
	\right]G_\varepsilon(t),
\end{equation*}
for almost every $t\in [0,T)$. Therefore, further introducing the absolutely continuous functional
\begin{equation*}
	\Theta_\varepsilon(t)\bydef
	G_\varepsilon(t)e^{-\int_0^t \left[
	\sigma \norm{u_2}_{L^\infty_x}^2
	+\norm{j_2}_{L^\infty_x}
	\right](\tau) d\tau},
\end{equation*}
we see that
\begin{equation*}
	\frac d{dt}\Theta_\varepsilon(t)
	\leq
	2C_{\rm BS}q\norm{\omega_2}_{L^q_x}\norm{\tilde u}_{L^{\infty}_x}^{\frac{2}{q}}
	\Theta_\varepsilon(t)^\frac{1}{q'}
	e^{-\frac 1q\int_0^t \left[
	\sigma \norm{u_2}_{L^\infty_x}^2
	+\norm{j_2}_{L^\infty_x}
	\right](\tau) d\tau},
\end{equation*}
whence
\begin{equation*}
	\frac d{dt}\Theta_\varepsilon(t)^\frac 1q
	\leq
	2C_{\rm BS}\norm{\omega_2}_{L^q_x}\norm{\tilde u}_{L^{\infty}_x}^{\frac{2}{q}}
	e^{-\frac 1q\int_0^t \left[
	\sigma \norm{u_2}_{L^\infty_x}^2
	+\norm{j_2}_{L^\infty_x}
	\right](\tau) d\tau}.
\end{equation*}
Observe that all technical difficulties incurred by the division by $\Theta_\varepsilon(t)^\frac{1}{q'}$, in the last step, are removed by the use of $\varepsilon>0$.

Now, integrating in time leads to
\begin{equation*}
	\begin{aligned}
		F_\varepsilon(t)^\frac 1q\leq G_\varepsilon(t)^\frac 1q
		&\leq \varepsilon e^{\frac 1q \int_0^t \left[
		\sigma \norm{u_2}_{L^\infty_x}^2
		+\norm{j_2}_{L^\infty_x}
		\right](\tau) d\tau}
		\\
		&\quad+
		2C_{\rm BS}\int_0^t\norm{\omega_2(\tau)}_{L^q_x}\norm{\tilde u(\tau)}_{L^{\infty}_x}^{\frac{2}{q}}
		e^{\frac 1q\int_\tau^t \left[
		\sigma \norm{u_2}_{L^\infty_x}^2
		+\norm{j_2}_{L^\infty_x}
		\right](s) ds}d\tau,
	\end{aligned}
\end{equation*}
whereby, letting $\varepsilon\to 0$, we end up with
\begin{equation*}
	\begin{aligned}
		\frac 12\left(\norm {\tilde u(t)}_{L^2}^2 + \|\tilde E(t)\|_{L^2}^2 + \|\tilde B(t)\|_{L^2}^2\right)
		\hspace{-50mm}&
		\\
		&\leq
		\left(
		2C_{\rm BS}
		\int_0^t\norm{\omega_2(\tau)}_{L^q_x}\norm{\tilde u(\tau)}_{L^{\infty}_x}^{\frac{2}{q}}
		d\tau
		\right)^q
		e^{\int_0^t \left[
		\sigma \norm{u_2}_{L^\infty_x}^2
		+\norm{j_2}_{L^\infty_x}
		\right](s) ds}
		\\
		&\leq
		\left(
		2C_{\rm BS}\norm{\omega_2}_{L^{q'}([0,t);L^q_x)}\right)^q
		\norm{\tilde u}_{L^2([0,t);L^\infty_x)}^2
		e^{
		\sigma\norm{u_2}_{L^2([0,t);L^\infty_x)}^2
		+\norm{j_2}_{L^1([0,t);L^\infty_x)}}
		\\
		&\leq
		\left(
		2C_{\rm BS}
		\norm{\omega_2}_{L^{1}([0,t);L^\infty_x)}\right)^{q-2}
		\left(2C_{\rm BS}
		\norm{\omega_2}_{L^{2}([0,t);L^2_x)}\right)^2
		\\
		&\quad\times\norm{\tilde u}_{L^2([0,t);L^\infty_x)}^2
		e^{
		\sigma\norm{u_2}_{L^2([0,t);L^\infty_x)}^2
		+\norm{j_2}_{L^1([0,t);L^\infty_x)}}.
	\end{aligned}
\end{equation*}
Finally, considering values $t\in [0,T)$ such that
\begin{equation*}
	\norm{\omega_2}_{L^{1}([0,t);L^\infty_x)} < \frac 1{2C_{\rm BS}}
\end{equation*}
and then letting $q$ tend to infinity, we conclude that $\big(\tilde u,\tilde E,\tilde B\big)(t)=0$, thereby establishing the uniqueness of solutions on a time interval $[0,t)$, for some $t\in(0,T)$.

Now, observe that this argument can be reproduced on any time interval $[t_0,T)$, with $t_0>0$ such that $\big(\tilde u,\tilde E,\tilde B\big)(t_0)=0$, to prove the uniqueness of solutions on $[t_0,t)$, for some $t\in(t_0,T)$. In other words, we have shown that the set
\begin{equation*}
	S=\left\{t\in [0,T) : \int_0^t\big\|\big(\tilde u,\tilde E,\tilde B\big)(s)\big\|_{L^2_x}ds=0\right\}
\end{equation*}
is open. Since the function $t\mapsto \int_0^t\big\|\big(\tilde u,\tilde E,\tilde B\big)(s)\big\|_{L^2_x}ds$ is continuous, the set $S$ is actually both open and closed. Furthermore, it is non-empty and $[0,T)$ is connected. We conclude that $S=[0,T)$ and, therefore, that both solutions $(u_1,E_1,B_1)$ and $(u_2,E_2,B_2)$ match on the whole interval of existence $[0,T)$. This completes the proof of the weak-strong uniqueness principle, thereby concluding the proof of the theorem. \qed

\begin{rem}
	In view of the estimates on the electric current $j$ established in the preceding proof, it seems also possible to propagate the boundedness of the $\dot B^s_{p,1}$-norms of the vorticity, with $p\in (1,\infty)$ and $s=\frac 2p$, by making use of the methods developed in \cite{v} to prove the global well-posedness of the two-dimensional incompressible Euler system in critical spaces. We also refer to \cite{HHK} and \cite{Z} for well-posedness results of similar models in critical spaces.
\end{rem}

\begin{rem}
	It is possible to propagate the boundedness of $L^p$-norms of the vorticity, for values $1\leq p<2$. Indeed, a variation of the proof of Lemma \ref{vorticity:0} gives that
	\begin{equation*}
		\begin{aligned}
			\norm{\omega(t)}_{L^p_x}
			&\leq \norm{\omega(0)}_{L^p_x}
			+\norm{j}_{L^2([0,t);L^2_x)}\norm{\nabla B}_{L^2([0,t);L^\frac{2p}{2-p}_x)}
			\\
			&\leq \norm{\omega(0)}_{L^p_x}
			+\norm{j}_{L^2([0,t);L^2_x)}
			\norm{\nabla B}_{L^2([0,t);L^2_x)}^{\frac{2}{p}-1}
			\norm{\nabla B}_{L^2([0,t);L^\infty_x)}^{2-\frac 2p},
		\end{aligned}
	\end{equation*}
	for any $t>0$. It is then readily seen that the terms involving $\nabla B$ are controlled by the bounds \eqref{propagation:damping} and \eqref{infinity:bound:1}, whereas the electric current $j$ remains bounded by virtue of the energy inequality \eqref{energy-inequa}.
\end{rem}

\subsection{Proof of Theorem \ref{main:3}}\label{section:main:3}

The proof of Theorem \ref{main:3} builds upon the estimates established in Theorem \ref{main:1}. Thus, following the proof of that theorem, we assume that we have a smooth solution $(u,E,B)$ of \eqref{EM}, for some smooth initial data $(u_0,E_0,B_0)$, and we only derive the bounds relevant to our argument through formal estimates on $(u,E,B)$, keeping in mind that the full justification of the result is then completed by carrying out the approximation strategy laid out in Section \ref{approximation:0}.

Now, the proof of Theorem \ref{main:1} establishes that the bound \eqref{initial:2} on the initial data $(u_0,E_0,B_0)$ implies the global uniform bound \eqref{global:bound:1} on the solution $(u,E,B)$. In particular, combining the two inequalities \eqref{initial:2} and \eqref{global:bound:1}, we see that, for any $0<A<1$, if the initial data satisfies \eqref{initial:2} with its right-hand side replaced by $Ac$, then the solution $(u,E,B)$ satisfies the estimate
\begin{equation}\label{handy:A}
	\mathcal{H}(0,T)< \frac{Ac}{2C_*\left(1+\mathcal{E}_0^{\frac 1\alpha+1}\right)},
\end{equation}
for all $T\in [0,\infty)$. This global bound will come in handy, below, with some small but fixed value for the constant $A$.

Next, in order to derive a higher-regularity estimate on the field $(E,B)$, we extend the control of high electromagnetic frequencies \eqref{proof:2} to higher smoothness parameters. Specifically, a direct application of estimate \eqref{frequencies:4} from Lemma \ref{high:freq:estimates} yields that
\begin{equation*}
	\begin{aligned}
		c^{1-s}\norm{(E,B)}_{\widetilde L^\infty_t\dot B^{s}_{2,n,>}}
		+c^{2-s}\norm{(E,B)}_{\widetilde L^2_t\dot B^{s}_{2,n,>}}
		+c^{\frac 74-s}\norm{(E,B)}_{\widetilde L^2_t\dot B^{s-\frac {3}4}_{\infty,n,>}}
		\hspace{-50mm}&
		\\
		&\lesssim
		c^{1-s}\norm{(E_0,B_0)}_{\dot B^s_{2,n,>}}
		+c^{1-s}\norm{u}_{L^\infty_{t,x} \cap L^\infty_t\dot H^1_x}\norm{B}_{\widetilde L_t^{2}\dot B^s_{2,n}},
	\end{aligned}
\end{equation*}
on any time interval $[0,T)$, for any $\frac 74<s<2$ and $1\leq n\leq\infty$. Then, recalling that the energy $\mathcal{H}(0,T)$ controls the velocity $u$ in $L^\infty_{t,x} \cap L^\infty_t\dot H^1_x$ (thanks to the Gagliardo--Nirenberg interpolation inequality \eqref{convexity:1}) and the magnetic field $B$ in $L_t^{2}\dot B^2_{2,1,<}$, we infer that the last term above can be bounded by
\begin{equation*}
	\begin{aligned}
		c^{1-s}\mathcal{H}(0,T)&
		\left(\norm{B}_{\widetilde L_t^{2}\dot B^1_{2,\infty,<}}^{2-s}\norm{B}_{\widetilde L_t^{2}\dot B^2_{2,\infty,<}}^{s-1}
		+\norm{B}_{\widetilde L_t^{2}\dot B^s_{2,n,>}}\right)
		\\
		&\lesssim
		c^{1-s}\mathcal{H}(0,T)
		\left(\norm{B}_{L_t^{2}\dot H^1_x}^{2-s}\norm{B}_{L_t^{2}\dot B^2_{2,1,<}}^{s-1}
		+\norm{B}_{\widetilde L_t^{2}\dot B^s_{2,n,>}}\right)
		\\
		&\lesssim
		c^{1-s}\mathcal{H}(0,T)
		\left(\left(\mathcal{E}_0+c^{-1}\mathcal{H}(0)
		+c^{-1}\mathcal{E}_0\mathcal{H}(0,T)\right)^{2-s}
		\mathcal{H}(0,T)^{s-1}
		+\norm{B}_{\widetilde L_t^{2}\dot B^s_{2,n,>}}\right),
	\end{aligned}
\end{equation*}
where we also employed \eqref{proof:9} in the last step to control $B$ in $L_t^{2}\dot H^1_x$.

All in all, combining the preceding inequalities and making use of \eqref{handy:A}, we arrive at
\begin{equation*}
	\begin{aligned}
		&c^{1-s}\norm{(E,B)}_{\widetilde L^\infty_t\dot B^{s}_{2,n,>}}
		+c^{2-s}\norm{(E,B)}_{\widetilde L^2_t\dot B^{s}_{2,n,>}}
		+c^{\frac 74-s}\norm{(E,B)}_{\widetilde L^2_t\dot B^{s-\frac {3}4}_{\infty,n,>}}
		\\
		&\lesssim
		c^{1-s}\norm{(E_0,B_0)}_{\dot B^s_{2,n,>}}
		+c^{1-s}\mathcal{E}_0^{2-s}\mathcal{H}(0,T)^s
		+c^{-1}\left(1+\mathcal{E}_0^{2-s}\right)\mathcal{H}(0,T)^2
		+Ac^{2-s}
		\norm{B}_{\widetilde L_t^{2}\dot B^s_{2,n,>}}.
	\end{aligned}
\end{equation*}
Thus, setting the value of the constant $A$ so small (with respect to fixed parameters only) that the last term above can be absorbed by the left-hand side, the previous estimate can be recast as
\begin{equation*}
	\begin{aligned}
		c^{1-s}\norm{(E,B)}_{\widetilde L^\infty_t\dot B^{s}_{2,n,>}}
		&+c^{2-s}\norm{(E,B)}_{\widetilde L^2_t\dot B^{s}_{2,n,>}}
		+c^{\frac 74-s}\norm{(E,B)}_{\widetilde L^2_t\dot B^{s-\frac {3}4}_{\infty,n,>}}
		\\
		&\lesssim
		c^{1-s}\norm{(E_0,B_0)}_{\dot B^s_{2,n,>}}
		+c^{1-s}\mathcal{E}_0^{2-s}\mathcal{H}(0,T)^s
		+c^{-1}\left(1+\mathcal{E}_0^{2-s}\right)\mathcal{H}(0,T)^2,
	\end{aligned}
\end{equation*}
which provides the pursued uniform control of the solution $(u,E,B)$ in higher regularity spaces.

Finally, there only remains to notice that 
\begin{equation*}
	\begin{aligned}
		c^{-\frac 34}\norm{(E_0,B_0)}_{\dot B^\frac 74_{2,1}}
		&\leq c^{-\frac 34}\norm{(E_0,B_0)}_{\dot B^\frac 74_{2,1,<}} + c^{-\frac 34}\norm{(E_0,B_0)}_{\dot B^\frac 74_{2,1,>}}
		\\
		&\lesssim
		\norm{(E_0,B_0)}_{\dot H^1}
		+
		c^{1-s}\norm{(E_0,B_0)}_{\dot B^s_{2,n}},
	\end{aligned}
\end{equation*}
which allows us to deduce that a suitable choice of independent constant $C_{**}>0$ in the initial assumption \eqref{initial:3} implies the corresponding initial condition \eqref{initial:1} in Theorem \ref{main:1} with its right-hand side replaced by $Ac$ (this is necessary to guarantee the validity of \eqref{handy:A}), thereby completing the proof of the theorem. \qed

\subsection{Uniform bounds for fixed initial data}\label{fixed:data:0}

As previously mentioned, the controls \eqref{propagation:damping} and \eqref{propagation:damping:3}, from Theorems \ref{main:1} and \ref{main:3}, hold for any families of initial data such that the left-hand sides of \eqref{initial:1} and \eqref{initial:3} remain respectively bounded. In particular, within such families, the corresponding collection of global solutions only satisfies the respective uniform bounds
\begin{equation*}
	\begin{aligned}
		c^{-\frac 34}(E,B)&\in \widetilde L^\infty(\mathbb{R}^+;\dot B^\frac 74_{2,1}),
		\\
		c^{1-s}(E,B)&\in \widetilde L^\infty(\mathbb{R}^+;\dot B^s_{2,n}).
	\end{aligned}
\end{equation*}
Thus, there is, in general, no bound on the size of the electromagnetic field $(E,B)$ in $\widetilde L^\infty(\mathbb{R}^+;\dot B^\frac 74_{2,1})$ and $\widetilde L^\infty(\mathbb{R}^+;\dot B^s_{2,n})$, uniformly in $c$, if the corresponding family of initial data $(E_0,B_0)$ only satisfies a uniform control
\begin{equation*}
	\begin{aligned}
		c^{-\frac 34}(E_0,B_0)&\in \dot B^\frac 74_{2,1},
		\\
		c^{1-s}(E_0,B_0)&\in \dot B^s_{2,n},
	\end{aligned}
\end{equation*}
respectively.

For example, such sets of initial electromagnetic fields occur naturally when considering mollifications
\begin{equation*}
	(u_0^c,E_0^c,B_0^c)=\varphi_c*(u_0,E_0,B_0),
\end{equation*}
where $\varphi_c(x)=c^2\varphi(cx)$ is a classical approximate identity and $(u_0,E_0,B_0)$ is a given initial data satisfying
\begin{equation}\label{initial:4}
	\left(\mathcal{E}_0+\norm{u_0}_{\dot H^1\cap\dot W^{1,p}}
	+\norm{(E_0,B_0)}_{\dot H^1}\right)
	C_{***}e^{C_{***}\mathcal{E}_0^{4+\varepsilon}}< c,
\end{equation}
with $C_{***}>0$. Indeed, it is readily seen that $(u_0^c,E_0^c,B_0^c)$ satisfies the bounds \eqref{initial:1} and \eqref{initial:3} provided \eqref{initial:4} holds true, for some suitable constant $C_{***}$.

We are now going to show that the solutions constructed in Theorems \ref{main:1} and \ref{main:3} have, in fact, an electromagnetic field which remains uniformly bounded in
\begin{equation*}
	\widetilde L^\infty(\mathbb{R}^+;\dot B^\frac 74_{2,1})
	\quad\text{and}\quad
	\widetilde L^\infty(\mathbb{R}^+;\dot B^s_{2,n}),
\end{equation*}
provided that their corresponding initial values are selected within a bounded family of $\dot B^\frac 74_{2,1}$ and $\dot B^s_{2,n}$, respectively. In particular, such uniform bounds hold whenever one considers a fixed initial data independent of $c$. This is of special significance in the study of the limiting regime $c\to\infty$, in order to derive sharp asymptotic bounds.

Specifically, the next result provides a suitable energy estimate on the damped Maxwell system \eqref{Maxwell:system:*} and the ensuing corollary clarifies the uniform boundedness properties of the solutions built in Theorems \ref{main:1} and \ref{main:3}.

\begin{lem}\label{technical:energy}
	Let $d=2$. Assume that $(E,B)$ is a smooth solution to \eqref{Maxwell:system:*}, for some initial data $(E_0,B_0)$ and some divergence-free vector field $u$, with the normal structure \eqref{structure:2dim}.
	
	Then, one has the estimate
	\begin{equation}\label{energy:EB:2}
		\begin{aligned}
			\norm{(E,B)}_{\widetilde L^\infty_t\dot B^s_{2,n}}
			+c\norm{E}_{\widetilde L^2_t\dot B^s_{2,n}}
			\hspace{-30mm}&
			\\
			&\lesssim \norm{(E_0,B_0)}_{\dot B^s_{2,n}}
			+\norm{u}_{L^\infty_{t,x} \cap L^\infty_t\dot H^1_x}
			\left(
			\norm{B}_{ L^2_t\dot H^1_x}
			+\norm{B}_{ L^2_t\dot B^2_{2,1,<}}
			+\norm{B}_{ \widetilde L^2_t\dot B^s_{2,n,>}}\right),
		\end{aligned}
	\end{equation}
	over any time interval $[0,T)$, for any $1<s<2$ and $1\leq n\leq \infty$.
\end{lem}

\begin{proof}
	This result follows from a direct energy estimate on the damped Maxwell system \eqref{Maxwell:system:*} and is an extension of Lemma \ref{prop:classical:energy}.
	
	In order to show \eqref{energy:EB:2}, we first localize \eqref{Maxwell:system:*} in frequencies by applying $\Delta_j$, for $j\in \mathbb{Z}$, and then perform a classical energy estimate on each dyadic frequency component $(\Delta_jE,\Delta_jB)$. This procedure yields the control
	\begin{equation*}
		\begin{aligned}
			\frac 1{2c}\norm{\left(\Delta_j E,\Delta_j B\right)(T)}_{L^2_x}^2
			+\sigma c\norm{\Delta_j E}_{L^2_t([0,T);L^2_x)}^2
			\hspace{-55mm}&
			\\
			&\leq
			\frac 1{2c}\norm{\left(\Delta_j E,\Delta_j B\right)(0)}_{L^2_x}^2
			+\sigma \norm{\Delta_j P(u\times B)}_{L^2_t([0,T);L^2_x)}
			\norm{\Delta_j E}_{L^2_t([0,T);L^2_x)}
			\\
			&\leq
			\frac 1{2c}\norm{\left(\Delta_j E,\Delta_j B\right)(0)}_{L^2_x}^2
			+\frac \sigma{2c}\norm{\Delta_j P(u\times B)}_{L^2_t([0,T);L^2_x)}^2
			+\frac{\sigma c}2\norm{\Delta_j E}_{L^2_t([0,T);L^2_x)}^2.
		\end{aligned}
	\end{equation*}
	Then, summing over $j$ in $\ell^n$, with $1\leq n\leq\infty$, we deduce that
	\begin{equation*}
		\norm{(E,B)}_{\widetilde L^\infty_t([0,T);\dot B^s_{2,n})}
		+c\norm{E}_{\widetilde L^2_t([0,T);\dot B^s_{2,n})}
		\lesssim \norm{(E_0,B_0)}_{\dot B^s_{2,n}}
		+\|P(u\times B)\|_{\widetilde L^2_t([0,T);\dot B^s_{2,n})},
	\end{equation*}
	for any $s\in\mathbb{R}$.
	
	Next, by the paradifferential product law \eqref{para-product:2}, we see that
	\begin{equation*}
		\|P(u\times B)\|_{\widetilde L^2_t\dot B^s_{2,n}}
		\lesssim
		\norm{u}_{L^\infty_{t,x} \cap L^\infty_t\dot B^1_{2,\infty}}\norm{B}_{ \widetilde L^2_t\dot B^s_{2,n}},
	\end{equation*}
	for any $-1<s<2$ and $1\leq n\leq\infty$. Therefore, combining the preceding inequalities, we conclude that
	\begin{equation*}
		\begin{aligned}
			\norm{(E,B)}_{\widetilde L^\infty_t\dot B^s_{2,n}}
			+c\norm{E}_{\widetilde L^2_t\dot B^s_{2,n}}
			\hspace{-25mm}&
			\\
			&\lesssim \norm{(E_0,B_0)}_{\dot B^s_{2,n}}
			+\norm{u}_{L^\infty_{t,x} \cap L^\infty_t\dot B^1_{2,\infty}}
			\left(\norm{B}_{ \widetilde L^2_t\dot B^s_{2,n,<}}+\norm{B}_{ \widetilde L^2_t\dot B^s_{2,n,>}}\right)
			\\
			&\lesssim \norm{(E_0,B_0)}_{\dot B^s_{2,n}}
			+\norm{u}_{L^\infty_{t,x} \cap L^\infty_t\dot B^1_{2,\infty}}
			\left(
			\norm{B}_{ \widetilde L^2_t\dot B^1_{2,\infty,<}}^{2-s}
			\norm{B}_{ \widetilde L^2_t\dot B^2_{2,\infty,<}}^{s-1}
			+\norm{B}_{ \widetilde L^2_t\dot B^s_{2,n,>}}\right),
		\end{aligned}
	\end{equation*}
	for all $1<s<2$ and $1\leq n\leq\infty$.
	
	Finally, combining the previous estimate with straightforward embeddings of Besov spaces proves \eqref{energy:EB:2}, thereby concluding the proof of the lemma.
\end{proof}

\begin{cor}\label{persistence:s-reg}
	Consider parameters $p\in (2,\infty)$, $\frac 74\leq s<2$ and $1\leq n\leq\infty$, such that $n=1$ if $s=\frac 74$. For any fixed initial data
	\begin{equation*}
		(u_0,E_0,B_0)\in
		\left((H^1\cap\dot W^{1,p})\times (H^1\cap\dot B^s_{2,n}) \times (H^1\cap\dot B^s_{2,n}) \right)(\mathbb{R}^2)
	\end{equation*}
	satisfying the assumptions of Theorem \ref{main:1} (if $s=\frac 74$ and $n=1$) or Theorem \ref{main:3} (if $s>\frac 74$), the corresponding global solutions $(u^c,E^c,B^c)$ constructed therein satisfy the bounds
	\begin{equation*}
		\begin{gathered}
			u^c\in L^\infty(\mathbb{R}^+;H^1\cap \dot W^{1,p}),
			\quad
			(E^c,B^c)\in L^\infty(\mathbb{R}^+;H^1),
			\quad
			(E^c,B^c)\in \widetilde L^\infty(\mathbb{R}^+;\dot B^s_{2,n}),
			\\
			(cE^c,B^c)\in L^2(\mathbb{R}^+;\dot H^1),
			\quad
			B^c\in L^2(\mathbb{R}^+;\dot B^2_{2,1,<}),
			\\
			c^{\frac 74 -s}(E^c,B^c)\in \widetilde L^2(\mathbb{R}^+;\dot B^{s-\frac 34}_{\infty,n,>}),
			\quad
			c E^c\in \widetilde L^2(\mathbb{R}^+;\dot B^s_{2,n}),
			\quad
			c^\frac 14 B^c\in \widetilde L^2(\mathbb{R}^+;\dot B^s_{2,n,>}),
		\end{gathered}
	\end{equation*}
	uniformly in $c$.
\end{cor}

\begin{proof}
	This result follows from a straightforward combination of Theorems \ref{main:1} and \ref{main:3} with Lemma \ref{technical:energy}. Indeed, it is readily seen that the uniform bounds \eqref{propagation:damping} and \eqref{propagation:damping:3} are sufficient to control the right-hand side of \eqref{energy:EB:2}, for appropriate values of $(s,n)$, thereby showing the improved uniform boundedness of the solutions $(u^c,E^c,B^c)$.
\end{proof}


\appendix

\section{Littlewood--Paley decompositions and Besov spaces}\label{besov:1}

We recall here the fundamentals of Littlewood--Paley decompositions and introduce a precise definition of Besov spaces.

\subsection{Littlewood--Paley decompositions}

We are going to use the Fourier transform
\begin{equation*}
	\mathcal{F}f\left(\xi\right)=\hat f(\xi)\bydef\int_{\mathbb{R}^d} e^{- i \xi \cdot x} f(x) dx
\end{equation*}
and its inverse
\begin{equation*}
	\mathcal{F}^{-1} g\left(x\right)=\tilde g(x)\bydef\frac{1}{\left(2\pi\right)^d}\int_{\mathbb{R}^d} e^{i x \cdot \xi} g(\xi) d\xi,
\end{equation*}
in any dimension $d\geq 1$.

Now, consider smooth cutoff functions $\psi(\xi),\varphi(\xi)\in C_c^\infty\left(\mathbb{R}^d\right)$ satisfying that
\begin{equation*}
	\psi,\varphi\geq 0\text{ are radial},
	\quad\supp\psi\subset\left\{|\xi|\leq 1\right\},
	\quad\supp\varphi\subset\left\{\frac{1}{2}\leq |\xi|\leq 2\right\}
\end{equation*}
and
\begin{equation*}
	1= \psi(\xi)+\sum_{k=0}^\infty \varphi\left(2^{-k}\xi\right),
	\quad\text{for all }\xi\in\mathbb{R}^d.
\end{equation*}
Defining the scaled cutoffs
\begin{equation*}
	\psi_{k}(\xi)\bydef\psi\left(2^{-k}\xi\right),
	\quad
	\displaystyle\varphi_{k}(\xi)\bydef\varphi\left(2^{-k}\xi\right),
\end{equation*}
for each $k\in\mathbb{Z}$, so that
\begin{equation*}
	\supp\psi_{k}\subset\left\{ |\xi|\leq 2^k\right\},
	\quad \supp\varphi_{k}\subset\left\{ 2^{k-1}\leq |\xi|\leq 2^{k+1}\right\},
\end{equation*}
it is readily seen that
\begin{equation*}
	1\equiv \psi+\sum_{k=0}^\infty \varphi_{k}
\end{equation*}
provides us with a dyadic partition of unity of $\mathbb{R}^d$.
Notice also that
\begin{equation*}
	1\equiv \psi_j+\sum_{k=j}^\infty \varphi_{k},
\end{equation*}
for any $j\in\mathbb{Z}$, and
\begin{equation*}
	1\equiv \sum_{k=-\infty}^\infty \varphi_{k}
\end{equation*}
away from the origin $\xi=0$.

Next, denoting the space of tempered distributions by $\mathcal{S}'$, we introduce the Fourier multiplier operators
\begin{equation*}
	S_k,\Delta_k:
	\mathcal{S}'\left(\mathbb{R}^d\right)\rightarrow\mathcal{S}'\left(\mathbb{R}^d\right),
\end{equation*}
with $k\in\mathbb{Z}$, defined by
\begin{equation}\label{dyadic:def}
	S_k f\bydef\mathcal{F}^{-1}\psi_k\mathcal{F}f
	= \left(\mathcal{F}^{-1}\psi_k\right)*f
	\quad\text{and}\quad
	\Delta_k f\bydef\mathcal{F}^{-1}\varphi_k\mathcal{F}f
	= \left(\mathcal{F}^{-1}\varphi_k\right)*f.
\end{equation}
The Littlewood--Paley decomposition of $f\in\mathcal{S}'$ is then given by
\begin{equation*}
	S_0 f+\sum_{k=0}^\infty\Delta_{ k}f=f,
\end{equation*}
where the series is convergent in $\mathcal{S}'$.

Similarly, one can verify that the homogeneous Littlewood--Paley decomposition
\begin{equation*}
	\sum_{k=-\infty}^\infty\Delta_{ k}f=f
\end{equation*}
holds in $\mathcal{S}'$, if $f\in\mathcal{S}'$ satisfies that
\begin{equation}\label{origin:1}
	\lim_{k\to -\infty}\|S_kf\|_{L^\infty}=0.
\end{equation}
Observe that \eqref{origin:1} holds if $\hat f$ is locally integrable around the origin, or whenever $S_0f$ belongs to $ L^p(\mathbb{R}^d)$, for some $1\leq p<\infty$. In particular, note that \eqref{origin:1} excludes all nonzero polynomials.

\subsection{Besov spaces}

For any $s \in \mathbb{R}$ and $1\leq p,q\leq \infty$, we define now the homogeneous Besov space $\dot B^{s}_{p,q}\left(\mathbb{R}^d\right)$ as the subspace of tempered distributions satisfying \eqref{origin:1} endowed with the norm
\begin{equation*}
	\left\|f\right\|_{\dot B^{s}_{p,q}\left(\mathbb{R}^d\right)}=
	\left(
	\sum_{k\in\mathbb{Z}} 2^{ksq}
	\left\|\Delta_{k}f\right\|_{L^p\left(\mathbb{R}^d\right)}^q\right)^\frac{1}{q},
\end{equation*}
if $q<\infty$, and
\begin{equation*}
	\left\|f\right\|_{\dot B^{s}_{p,q}\left(\mathbb{R}^d\right)}=
	\sup_{k\in\mathbb{Z}}\left(2^{ks}
	\left\|\Delta_{k}f\right\|_{L^p\left(\mathbb{R}^d\right)}\right),
\end{equation*}
if $q=\infty$. One can show that $\dot B^s_{p,q}$ is a Banach space if $s<\frac dp$, or if $s=\frac dp$ and $q=1$ (see \cite[Theorem 2.25]{bcd11}).

We also consider here the homogeneous Sobolev space $\dot H^s\left(\mathbb{R}^d\right)$, for any real value $s\in\mathbb{R}$, which is defined as the subspace of tempered distributions whose Fourier transform is locally integrable equipped with the norm
\begin{equation*}
	\left\|f\right\|_{\dot H^s}=\left(\int_{\mathbb{R}^d}|\xi|^{2s}|\hat f(\xi)|^2 d\xi\right)^\frac 12.
\end{equation*}
One verifies that $\dot H^s$ is a Hilbert space if and only if $s<\frac d2$ (see \cite[Proposition 1.34]{bcd11}). Moreover, it is possible to show that $\dot H^s = \dot B^s_{2,2}$ whenever $s<\frac d2$.

\subsection{Chemin--Lerner spaces}

For any time $T>0$ and any choice of parameters $s \in \mathbb{R}$ and $1\leq p,q,r\leq \infty$, with $s<\frac dp$ (or $s=\frac dp$ and $q=1$), the spaces
\begin{equation*}
	L^r\left( [0,T) ; \dot B^{s}_{p,q}\left(\mathbb{R}^d\right) \right)
\end{equation*}
are naturally defined as $L^r$-spaces with values in the Banach spaces $\dot B^{s}_{p,q}$. In addition to these vector-valued Lebesgue spaces, we define the spaces
\begin{equation*}
	\widetilde L^r\left( [0,T) ; \dot B^{s}_{p,q}\left(\mathbb{R}^d\right) \right)
\end{equation*}
as the subspaces of tempered distributions such that
\begin{equation*}
	\lim_{k\to -\infty}\|S_kf\|_{L^r\left([0,T);L^p\left(\mathbb{R}^d\right)\right)}=0 \, ,
\end{equation*}
endowed with the norm
\begin{equation*}
	\left\|f\right\|_{ \widetilde L^r \left( [0,T) ; B^{s}_{p,q}\left(\mathbb{R}^d\right) \right)}
	=
	\left( \sum_{k=-\infty}^\infty 2^{ksq}
	\left\|\Delta_{k}f\right\|_{L^r\left([0,T);L^p\left(\mathbb{R}^d\right)\right) }^q\right)^\frac{1}{q}\, ,
\end{equation*}
if $q<\infty$, and with the obvious modifications in case $q=\infty$.

One can easily check that, if $q \geq r $, then
\begin{equation*}
	 L^r\left( [0,T) ; \dot B^{s}_{p,q}\left(\mathbb{R}^d\right) \right)
	\subset \widetilde L^r\left( [0,T) ; \dot B^{s}_{p,q}\left(\mathbb{R}^d\right) \right)\, ,
\end{equation*}
and that, if $q \leq r $, then
\begin{equation*}
	\widetilde L^r\left( [0,T) ; \dot B^{s}_{p,q}\left(\mathbb{R}^d\right)  \right)
	\subset L^r\left( [0,T); \dot B^{s}_{p,q}\left(\mathbb{R}^d\right) \right)\, .
\end{equation*}
We refer the reader to \cite[Section 2.6.3]{bcd11} for more details on Chemin--Lerner spaces.

\subsection{Embeddings}

We present now a few embeddings and inequalities in Besov spaces which are routinely used throughout this work.

First, a direct application of Young's convolution inequality to \eqref{dyadic:def} yields that
\begin{equation}\label{embedding:1}
	\norm{\Delta_kf}_{L^r(\mathbb{R}^d)}
	\lesssim 2^{kd\left(\frac 1p-\frac 1r\right)}\norm{\Delta_kf}_{L^p(\mathbb{R}^d)},
\end{equation}
for any $1\leq p\leq r\leq\infty$. A suitable use of \eqref{embedding:1} then leads to the embedding
\begin{equation}\label{embedding:2}
	\norm{f}_{\dot B^s_{r,q}(\mathbb{R}^d)}\lesssim \norm{f}_{\dot B^{s+d\left(\frac 1p-\frac 1r\right)}_{p,q}(\mathbb{R}^d)},
\end{equation}
for any $1\leq p\leq r\leq \infty$, $1\leq q\leq\infty$ and $s\in\mathbb{R}$, which can be interpreted as a Sobolev embedding in the framework of Besov spaces.

Moreover, recalling that $\ell^q\subset\ell^r$, for all $1\leq q\leq r\leq\infty$, one has that
\begin{equation*}
	\dot B^s_{p,q}(\mathbb{R}^d)\subset \dot B^s_{p,r}(\mathbb{R}^d),
\end{equation*}
for all $s\in\mathbb{R}$, $1\leq p\leq\infty$ and $1\leq q\leq r\leq\infty$.

Next, observe that
\begin{equation}\label{embedding:3}
	\norm{f}_{L^p(\mathbb{R}^d)}=\norm{\sum_{k\in\mathbb{Z}}\Delta_k f}_{L^p(\mathbb{R}^d)}
	\leq \sum_{k\in\mathbb{Z}}\norm{\Delta_k f}_{L^p(\mathbb{R}^d)}
	=\norm{\Delta_k f}_{\dot B^0_{p,1}(\mathbb{R}^d)},
\end{equation}
for every $1\leq p\leq \infty$. Therefore, by combining \eqref{embedding:2} and \eqref{embedding:3}, we obtain that
\begin{equation*}
	\norm{f}_{L^\infty(\mathbb{R}^d)}\lesssim
	\norm{f}_{\dot B^{\frac d2}_{2,1}(\mathbb{R}^d)}.
\end{equation*}
This estimate is particularly useful in view of the failure of the embedding of the Sobolev space $\dot H^{\frac d2}(\mathbb{R}^d)$ into $L^\infty(\mathbb{R}^d)$.

Further considering any cutoff function $\chi\in C^\infty_c(\mathbb{R}^d)$ such that $\mathds{1}_{\{|\xi|\leq 1\}}\leq \chi(\xi) \leq \mathds{1}_{\{|\xi|\leq 2\}}$, one can show, for any $c>0$, $\alpha>0$ and $1\leq p\leq\infty$, that
\begin{equation*}
	\norm{\chi\left(\frac Dc\right)f}_{\dot B^{s+\alpha}_{p,1}(\mathbb{R}^d)}\lesssim
	c^\alpha \norm{f}_{\dot B^s_{p,\infty}(\mathbb{R}^d)}
\end{equation*}
and
\begin{equation*}
	c^\alpha\norm{(1-\chi)\left(\frac Dc\right)f}_{\dot B^{s}_{p,1}(\mathbb{R}^d)}\lesssim
	\norm{f}_{\dot B^{s+\alpha}_{p,\infty}(\mathbb{R}^d)},
\end{equation*}
where the operator $m\left(D\right)$ denotes the Fourier multiplier associated with the symbol $m\left(\xi\right)$, for any $m\in C^\infty_c(\mathbb{R}^d)$.

Finally, we mention another essential inequality in Besov spaces which is related to their interpolation properties. Specifically, one has the interpolation, or convexity, inequality
\begin{equation*}
	\norm f_{\dot B^{s}_{p,1}}
	\lesssim
	\norm f_{\dot B^{s_0}_{p,\infty}}^{1-\theta}
	\norm f_{\dot B^{s_1}_{p,\infty}}^{\theta},
\end{equation*}
for any $p\in[1,\infty]$, $s,s_0,s_1\in\mathbb{R}$ and $\theta\in(0,1)$ such that $s=(1-\theta)s_0+\theta s_1$ and $s_0\neq s_1$.

Note that the preceding estimates and embeddings can be adapted to the setting of Chemin--Lerner spaces in a straightforward way.

\subsection{Paradifferential product estimates}

Here, we recall the basic principles of paraproduct decompositions and some essential paradifferential product estimates that follow from it.

For any two suitable tempered distributions $f$ and $g$, introducing the paraproduct
\begin{equation*}
	T_fg=\sum_{j\in\mathbb{Z}}S_{j-2}f\Delta_j g
\end{equation*}
readily leads to the decomposition
\begin{equation*}
	fg=T_fg+T_gf+R(f,g),
\end{equation*}
where
\begin{equation*}
	R(f,g)=\sum_{\substack{j,k\in\mathbb{Z}\\|j-k|\leq 2}}\Delta_jf\Delta_kg
\end{equation*}
is the remainder. For any choice of integrability parameters in $[1,\infty]$ such that
\begin{equation*}
	\frac 1a=\frac 1{a_1}+\frac 1{a_2},
	\qquad
	\frac 1b=\frac 1{b_1}+\frac 1{b_2},
	\qquad
	\frac 1c=\frac 1{c_1}+\frac 1{c_2},
\end{equation*}
it can be shown, in the context of Chemin--Lerner spaces, that
\begin{equation}\label{para-product:4}
	\norm{T_fg}_{\widetilde L^a_t\dot B^{\alpha+\beta}_{b,c}}
	\lesssim
	\norm{f}_{\widetilde L^{a_1}_t\dot B^\alpha_{b_1,c_1}}
	\norm{g}_{\widetilde L^{a_2}_t\dot B^\beta_{b_2,c_2}},
\end{equation}
for any $\alpha<0$ and $\beta\in\mathbb{R}$, and that
\begin{equation}\label{para-product:5}
	\norm{R(f,g)}_{\widetilde L^a_t\dot B^{\alpha+\beta}_{b,c}}
	\lesssim
	\norm{f}_{\widetilde L^{a_1}_t\dot B^\alpha_{b_1,c_1}}
	\norm{g}_{\widetilde L^{a_2}_t\dot B^\beta_{b_2,c_2}},
\end{equation}
for any $\alpha,\beta\in\mathbb{R}$ with $\alpha+\beta>0$. Moreover, one has the endpoint estimates
\begin{equation}\label{para-product:6}
	\norm{T_fg}_{\widetilde L^a_t\dot B^{\beta}_{b,c}}
	\lesssim
	\norm{f}_{L^{a_1}_tL^{b_1}_x}
	\norm{g}_{\widetilde L^{a_2}_t\dot B^\beta_{b_2,c}},
\end{equation}
for all $\beta\in\mathbb{R}$,
and
\begin{equation*}
	\norm{R(f,g)}_{L^a_tL^b_x}
	\lesssim
	\norm{f}_{\widetilde L^{a_1}_t\dot B^\alpha_{b_1,c_1}}
	\norm{g}_{\widetilde L^{a_2}_t\dot B^{-\alpha}_{b_2,c_2}},
\end{equation*}
for all $\alpha\in\mathbb{R}$, provided $\frac 1{c_1}+\frac 1{c_2}=1$.
Similar bounds hold for Besov spaces and we refer to \cite[Section 2.6]{bcd11} for more details on such paradifferential estimates.

We finally recall two important product rules of paradifferential calculus in the context of Besov spaces, which are direct consequences of the preceding bounds.
First, exploiting \eqref{para-product:4} and \eqref{para-product:5} (for Besov spaces), we have that
\begin{equation*}
	\|fg\|_{\dot B^{s+t-\frac d2}_{2,1}} \lesssim \|f\|_{\dot H^s}\|g\|_{\dot H^t},
\end{equation*}
for any $-\frac d2<s,t<\frac d2$ with $s+t>0$.

Second, we find that
\begin{equation*}
	\|fg\|_{\dot H^s} \lesssim \|f\|_{L^\infty\cap \dot B^\frac d2_{2,\infty}}\|g\|_{\dot H^s},
\end{equation*}
for all $-\frac d2<s<\frac d2$, which follows from a suitable combination of \eqref{para-product:4}, \eqref{para-product:5} and \eqref{para-product:6} (for Besov spaces, as well).

\section{Oscillatory integrals and dispersion}\label{dispersion:2}

We give here a justification of the dispersive estimate \eqref{dispersion:1} employing the Stationary Phase Method. This method is classical and we will rely on \cite{bcd11}, when needed, to refer the reader to complete details on the technical results pertaining to the method.

Generally speaking, we are considering here oscillatory integrals of the form
\begin{equation*}
	I_\psi(t) = \int_{\mathbb{R}^d} e^{it \phi(\xi)}\psi(\xi) d\xi,
\end{equation*}
for smooth test functions $\psi\in C^\infty_c(\mathbb{R}^d)$ and $t\in\mathbb{R}$, where the smooth phase $\phi(\xi)$ only needs to be defined on the support of $\psi$. It is readily seen that
\begin{equation*}
	|I_\psi(t)| \leq \norm{\psi}_{L^1(\mathbb{R}^d)}.
\end{equation*}
We seek out now to understand the asymptotic behavior of $I_\psi(t)$ when $|t|$ is large, which requires us to exploit the cancellations in the integral $I_\psi(t)$ due to the oscillatory term $e^{it \phi(\xi)}$. There are two cases to consider: the stationary phase and the nonstationary phase.

\paragraph{\bf The stationary phase.}

This case analyzes the asymptotic behavior of $I_\psi(t)$ near critical points of the phase, i.e., near points in the integration domain where $\nabla\phi(\xi)=0$. More precisely, we suppose here that
\begin{equation*}
	|\nabla \phi(\xi)| \leq \varepsilon_0,
\end{equation*}
for some $\varepsilon_0 \in (0,1]$ and for all $\xi\in\supp \psi$. Under such assumptions, Theorem 8.9 from \cite{bcd11} establishes that, for any positive integer $N$, there is a constant $C_{N,\phi,\psi}>0$ such that
\begin{equation}\label{non-stationary:es:1}
	|I_\psi(t)| \leq C_{N,\phi,\psi} \int_{\supp\psi} \frac{ d\xi}{\left( 1 + \varepsilon_0 t |\nabla \phi(\xi)|^2 \right)^{N}},
\end{equation}
for all $t>0$.

\paragraph{\bf The nonstationary phase.}

The decay of $I_\psi(t)$ is better when $\nabla \phi$ does not vanish on the support of $\psi$. More precisely, assuming now that
\begin{equation*}
	|\nabla \phi(\xi)| \geq \varepsilon_0,
\end{equation*}
for some $\varepsilon_0 \in (0,1]$ and for all $\xi\in\supp \psi$, Theorem 8.8 from \cite{bcd11} shows that, for any positive integer $N$, there is a constant $C_{N,\phi,\psi}>0$ such that
\begin{equation}\label{stationary:es:1}
	|I_\psi(t)| \leq \frac{C_{N,\phi,\psi}}{(\varepsilon_0t)^N},
\end{equation}
for all $t>0$.

The asymptotic estimate \eqref{stationary:es:1} always offers a faster decay than \eqref{non-stationary:es:1} and, therefore, the oscillatory integral $I_\psi(t)$ can be treated as a remainder term wherever the phase $\nabla\phi$ does not vanish. In conclusion, the overall asymptotic behavior of $I_\psi(t)$ is, in general, determined by the critical points of the phase.

\medskip

All in all, as explained in Theorem 8.12 from \cite{bcd11}, it is possible to combine the preceding estimates to show, for all $\psi\in C_c^\infty(\mathbb{R}^d)$, $\varepsilon_0\in (0,1]$ and any positive numbers $N$ and $N'$, that there are positive constants $C_N$ and $C_{N'}$ such that
\begin{equation}\label{oscillatory:es}
	|I_\psi(t)| \leq
	\frac{C_{N}}{(\varepsilon_0 t)^N}
	+C_{N'} \int_{A_\phi} \frac{ d\xi}{\left( 1 + \varepsilon_0 t |\nabla \phi(\xi)|^2 \right)^{N'}},
\end{equation}
for all $t>0$, where the set $A_\phi$ is defined as
\begin{equation*}
	A_\phi \bydef \left\{\xi\in\supp\psi : |\nabla \phi(\xi)| \leq \varepsilon_0\right\}.
\end{equation*}

We are now in a position to prove \eqref{dispersion:1}. To be precise, we are going to establish the equivalent estimate (up to a scaling of the variable $x$)
\begin{equation}\label{dispersion:3}
	\left|\int_{\mathbb{R}^d}e^{it\phi(x,\xi)}\psi(\xi)d\xi\right|
	\leq \frac{C_\psi}{t^{\frac{d-1}2}},
\end{equation}
for all $t>0$ and $x\in\mathbb{R}^d$, where the phase is defined by
\begin{equation*}
	\phi(x,\xi)=x\cdot\xi\pm \delta(\xi),
	\quad\text{with }\delta(\xi)=\sqrt{|\xi|^2-\frac{\alpha^2}4} \text{ and }0\leq\alpha\leq\frac 12,
\end{equation*}
and the test function $\psi\in C^\infty_c(\mathbb{R}^d)$ satisfies $\supp\psi\subset\{\frac 14<|\xi|<R\}$, for some $R>\frac 14$, while the constant $C_\psi>0$ is independent of $t$, $x$ and $\alpha$.

To that end, noting that $\phi(x,\xi)$ is smooth on the support of $\psi$ and setting $\varepsilon_0=\frac 12$, $N=\frac{d-1}2$ and $N'=d$ in \eqref{oscillatory:es}, we find that
\begin{equation*}
	\left|\int_{\mathbb{R}^d}e^{it\phi(x,\xi)}\psi(\xi)d\xi\right| \lesssim
	\frac{1}{t^\frac{d-1}2}
	+ \int_{A} \frac{ d\xi}{\left( 1 + t \left|x\pm\frac{\xi}{\delta(\xi)}\right|^2 \right)^d},
\end{equation*}
where
\begin{equation*}
	A\bydef \left\{\frac 14<|\xi|<R,\ \left|x\pm\frac{\xi}{\delta(\xi)}\right|\leq \frac 12\right\}.
\end{equation*}
Now, notice that $x\neq 0$, if $A$ is non-empty. In particular, for any $\xi\in A$, we can decompose
\begin{equation*}
	\xi = \zeta_1 + \zeta', \quad \text{with } \zeta_1 \bydef \left(\frac{\xi\cdot x}{|x|^2}\right)x
	\text{ and }\zeta' \bydef \xi - \zeta_1,
\end{equation*}
whence, since $\zeta'\cdot x=0$,
\begin{equation*}
	\left|x\pm \frac{\xi}{\delta(\xi)} \right|^2
	=\left|x\pm \frac{\zeta_1}{\delta(\xi)} \right|^2+\left|\frac{\zeta'}{\delta(\xi)} \right|^2
	\geq \frac{\left|\zeta'\right|^2}{\delta(\xi)^2}\geq \frac{\left|\zeta'\right|^2}{R^2}.
\end{equation*}
We therefore conclude that
\begin{equation*}
	\left|\int_{\mathbb{R}^d}e^{it\phi(x,\xi)}\psi(\xi)d\xi\right| \lesssim
	\frac{1}{t^\frac{d-1}2}
	+ \int_{\{|\zeta'|<R\}\subset\mathbb{R}^{d-1}} \frac{ d\zeta'}{\left( 1 + t \left|\zeta'\right|^2 \right)^d}
	\lesssim
	\frac{1}{t^\frac{d-1}2},
\end{equation*}
which completes the justification of \eqref{dispersion:3}. \qed


\bibliographystyle{plain} 
\bibliography{plasma}

\end{document}